\documentclass[]{imsart}

\usepackage[english]{babel}
\usepackage[utf8]{inputenc}
\usepackage{amsmath,amsthm, amssymb, latexsym, color,mathtools}
\usepackage{enumitem}
\usepackage{verbatim}
\usepackage{comment}
\usepackage{bbm}
\usepackage{mathrsfs}
\usepackage{tikz}
\usetikzlibrary{cd}
\usetikzlibrary{shapes,arrows,cd}
\usepackage{tikz-cd}
\usepackage[all,2cell]{xy}
\usepackage{abstract}
\usepackage[colorlinks,citecolor=blue,urlcolor=blue]{hyperref}
\usepackage{mathabx}
\usepackage{cleveref}
\usepackage{commath}
\usepackage[normalem]{ulem}

\usepackage[textsize=tiny]{todonotes}

\usepackage[letterpaper,margin=1.2in]{geometry}

\startlocaldefs
\newtheorem{thm}{Theorem}

\newtheorem{ass}[thm]{Assumption}
\numberwithin{thm}{section}
\numberwithin{equation}{section}

\newcommand{\E}{\mathbb{E}}

\newcommand{\R}{\mathbb{R}}

\newcommand{\N}{\mathbb{N}}

\newcommand{\eps}{\varepsilon}

\definecolor{jz}{rgb}{0.1,0.45,0.1}

\definecolor{rr}{rgb}{0,0,1}

\definecolor{ymm}{rgb}{1,0.53,0.0}

\definecolor{sw}{rgb}{0.2,0.4,0.5}

\newcommand{\bsv}{{\boldsymbol v}}

\newcommand{\bsalpha}{{\boldsymbol \alpha}}

\DeclareMathOperator{\tr}{tr}
\DeclareMathOperator*{\argmax}{arg\,max}

\DeclareMathOperator*{\essinf}{ess\,inf}

\newcommand{\dd}{\;\mathrm{d}}

    \numberwithin{equation}{section}

    \definecolor{plum}{rgb}{.4,0,.4}
    \definecolor{BrickRed}{rgb}{0.6,0,0}

    \def\ddefloop#1{\ifx\ddefloop#1\else\ddef{#1}\expandafter\ddefloop\fi}

    \def\ddef#1{\expandafter\def\csname c#1\endcsname{\ensuremath{\mathcal{#1}}}}
    \ddefloop ABCDEFGHIJKLMNOPQRSTUVWXYZ\ddefloop

    \def\ddef#1{\expandafter\def\csname s#1\endcsname{\ensuremath{\mathsf{#1}}}}
    \ddefloop ABCDEFGHIJKLMNOPQRSTUVWXYZ\ddefloop

    \def\E{\mathbf{E}}

	\def\tr{{\mathrm{tr}}}

    \def\1{{\mathbf 1}}

    \def\eps{\varepsilon}

   \newtheorem{theorem}{Theorem}[section]
   
\newcommand{\raisedtarget}[1]{%
  \raisebox{\fontcharht\font`P}[0pt][0pt]{\hypertarget{#1}{}}%
}

   \newtheorem{lemma}[theorem]{Lemma}
   
   \newtheorem{corollary}[theorem]{Corollary}
   
   \newtheorem{remark}[theorem]{Remark}
   
   \newtheorem{definition}[theorem]{Definition}

\renewcommand{\norm}[2][]{\| #2 \|_{#1}}
\newcommand{\normc}[2][]{\left\| #2 \right\|_{#1}}
\newcommand{\snorm}[2][]{| #2 |_{#1}}

\renewcommand{\set}[2]{\{#1\,|\,#2\}}
\newcommand{\setc}[3][\Big]{#1\{#2\, #1|\,#3 #1\}}

\newcommand{\bnu}{{\boldsymbol{\nu}}}
 \newcommand{\bx}{{\boldsymbol{x}}}
\newcommand{\blambda}{{\boldsymbol{\lambda}}}
\newcommand{\balpha}{{\boldsymbol{\alpha}}}
\newcommand{\bQ}{{\boldsymbol{Q}}} \newcommand{\bB}{{\boldsymbol{B}}}

\endlocaldefs

\begin{document}

\begin{frontmatter}
  
\title{Distribution learning via neural differential equations: a nonparametric statistical perspective}
\runtitle{Distribution learning via neural ODEs}
\begin{aug}
\author[A]{\fnms{Youssef}~\snm{Marzouk}\ead[label=e1]{ymarz@mit.edu}},
\author[A]{\fnms{Zhi}~\snm{Ren}\ead[label=e2]{zren@mit.edu}},
\author[A]{\fnms{Sven}~\snm{Wang}\ead[label=e3]{svenwang@mit.edu}},
\and
\author[B]{\fnms{Jakob}~\snm{Zech}\ead[label=e4]{jakob.zech@uni-heidelberg.de}}
\address[A]{Massachusetts Institute of Technology, Cambridge, MA 02139, USA \\ \printead[presep={}]{e1,e2,e3}}
\address[B]{Heidelberg University, 69120 Heidelberg, Germany\\ \printead[presep={}]{e4}}
\end{aug}

\runauthor{Marzouk, Ren, Wang, Zech}
\end{frontmatter}

\begin{abstract} 
Ordinary differential equations (ODEs), via their induced flow maps, provide a powerful framework to parameterize invertible transformations for the purpose of representing complex probability distributions. While such models have achieved enormous success in machine learning, particularly for generative modeling and density estimation, little is known about their statistical properties. This work establishes the first general nonparametric statistical convergence analysis for distribution learning via ODE models trained through likelihood maximization.
We first prove a convergence theorem applicable to \textit{arbitrary} velocity field classes $\mathcal{F}$ satisfying certain simple boundary constraints. This general result captures the trade-off between approximation error (`bias') and the complexity of the ODE model (`variance'). We show that the latter can be quantified via the $C^1$-metric entropy of the class $\mathcal F$.
We then apply this general framework to the setting of $C^k$-smooth target densities, and establish nearly minimax-optimal convergence rates for two relevant velocity field classes $\mathcal F$: $C^k$ functions and neural networks. The latter is the practically important case of neural ODEs. 

Our proof techniques require a careful synthesis of (i) analytical stability results for ODEs, (ii) classical theory for sieved M-estimators, and (iii) recent results on approximation rates and metric entropies of neural network classes. The results also provide theoretical insight on how the choice of velocity field class, and the dependence of this choice on sample size $n$ (e.g., the scaling of width, depth, and sparsity of neural network classes), impacts statistical performance.

\end{abstract}

\tableofcontents

\section{Introduction}

The interface of nonparametric statistics with complex models given by differential equations has been a major focus of contemporary statistics and applied mathematics.
On the one hand, physically motivated differential equation models of data-generating processes are central to inverse problems and data assimilation. There has been considerable progress in understanding such models through a statistical lens, 
leading to %
structure-exploiting algorithms \cite{rudolf2018generalization,schillings2016scaling,kim2023hippylib},
new consistency and uncertainty quantification guarantees \cite{MNP19,NVW18, NT23}, and a growing understanding of computational complexity \cite{NW20,N23}.
On the other hand, differential equations underpin the construction of new expressive and flexible model classes for \emph{representing} complex probability distributions, which have enjoyed enormous success in machine learning and data science.
Examples of such models include neural ordinary differential equations \cite{chen2018neural}, score-based diffusion models \cite{song2020score,yang2022diffusion}, and flow matching methods \cite{lipman2022flow,liu2022flow,albergo2023stochastic}.
In such models, a key aspect of the %
dynamics, e.g, the velocity field of an ordinary differential equation or the drift of a stochastic differential equation, is \emph{learned from data} by minimizing a suitable objective.
These approaches have achieved leading performance in diverse applications, ranging from generative modeling of images and video \cite{grathwohl2018ffjord,song2020denoising,ho2022imagen} to density estimation in high-energy physics \cite{nachman2020anomaly} to conditional sampling and simulation-based Bayesian inference \cite{shi2022conditional,batzolis2021conditional,cranmer2020frontier}.

This paper develops statistical finite-sample guarantees for distribution learning with ordinary differential equation (ODE) models. These models are described via finite-time \textit{flow maps} of ODEs \cite{ArnoldODE} of the form
\begin{equation}
\begin{cases}
    \frac{d}{dt}  X(x,t)  & = f  ( X(x,t), t ),\\
    X(x,0) &= x,
\end{cases}~~~~~ \text{for}~~x\in D,~~t\in [0,1].
\label{eq:neuralODEintro}
\end{equation}
Here $D\subseteq \R^d~(d\ge 1)$ is some domain and the velocity field $f\in\mathcal F$ belongs to some parametric function class $\mathcal F$. For each $f\in \mathcal F$, the collection of all trajectories of \eqref{eq:neuralODEintro} 
is captured by
the flow map $(x, t) \mapsto X^f(x,t)$, a continuous-time invertible transformation that can be 
used to push forward and pull back
probability distributions. In the context of statistical learning, this framework is applied to infer
complicated unknown distributions: a velocity field $\hat f\in \mathcal F$ is computed by minimizing a statistical objective, and the unknown distribution is then approximated as the  \textit{pullback} of a reference distribution (e.g., normal or uniform) under the terminal time (i.e., $t=1$) flow map $x\mapsto X^{\hat f}(x,1)$. 
This approximation immediately provides a density estimate, and crucially further enables sampling 
(hence generative modeling)
by evaluating the inverse of the flow map on reference samples.  
For details, see Section~\ref{sec:ODEest}.

One way of interpreting such ODE-based models is to view them as a specific parameterizations of
time-independent \textit{transport maps} \cite{MMPS16}. However, a key practical advantage of the ODE formulation over models that directly represent transport maps is that for \textit{any} sufficiently regular velocity field $f$, the ODE construction guarantees that the flow maps are invertible. Moreover, given an initial condition $x$ at time $t=0$ with known probability density, the density of the state at any intermediate time $t>0$ can easily be evaluated via the so-called `instantaneous change-of-variables' formula \cite{chen2018neural}. 
Since these features hold for very generic choices of $f$, they permit using virtually any approximation class---such as polynomials, neural networks, or kernel representations \cite{owhadi2019kernel}---to describe the parameter space $\mathcal F$. By contrast, in models which directly parameterize transport maps, significant care is needed to ensure invertibility and tractable Jacobian determinants.

When the velocity field $f$ is represented as a deep neural network, the system \eqref{eq:neuralODEintro} is called a \emph{neural ODE} \cite{chen2018neural}; such models achieve state-of-the-art performance in density estimation \cite{grathwohl2018ffjord,onken2021ot} and are competitive (by various sample quality metrics) for various generative modeling tasks. While this construction is powerful, most questions regarding theoretical performance guarantees for ODE-based methods remain unexplored; the notable exceptions \cite{ishikawa2022universal,li2022deep,ruizbalet21,ruiz2023control} will be discussed further below. To the best of our knowledge, current approximation results are limited to universal approximation \cite{SupApproximation,ishikawa2022universal,li2022deep,ruizbalet21}, while quantitative approximation \textit{rates} are yet unknown; our forthcoming companion paper \cite{ren23a} provides the first such approximation rate results. The present paper considers the yet more challenging task of giving \textit{statistical} finite-sample convergence guarantees, which has thus far only been considered by \cite{ruiz2023control}, whose proof approach and results are vastly different from ours; see below for further discussion. The observational setting we consider is that of \textit{nonparametric density estimation} (e.g., \cite{T08, GN16}), which also underlies generative modeling: a finite collection of independent and identically distributed (iid) samples is given,
\[ Z_1, \ldots, Z_n \stackrel{\text{iid}}{\sim} P_0, \] and our goal is to characterize 
the unknown target distribution $P_0$. We consider general estimators $\hat f$ which arise as minimizers of a negative log-likelihood (or empirical Kullback--Leibler) objective over some class $\mathcal F$, a training strategy which is extremely common in practice \cite{grathwohl2018ffjord,chen2018neural,finlay2020train}.

\subsection{Results and contributions} 

To our knowledge, our paper provides the first rigorous statistical analysis of likelihood-based ODE density estimators,
and specifically the first such statistical convergence results for \textit{neural} ODEs.
Our approach integrates tools from nonparametric M-estimation \cite{VDG00}, recent advances in 
statistical and approximation 
theory for
neural networks (e.g.,~\cite{Schmidt_Hieber_2020}), and ODE analytical theory (e.g.,~\cite{H02}).
Our results also create the first explicit framework for understanding the impact of choosing 
different velocity field classes on statistical performance.

In Section \ref{sec:general-results}, we develop a statistical convergence result applicable to \textit{general} ODE-parameterized maximum likelihood estimators (ODE-MLEs); see Theorem \ref{thm:general}. Specifically, given any variational class $\mathcal F$ of velocity fields (satisfying mild boundedness assumptions), our result gives a bound on the rate of convergence as a sum of two terms reminiscent of the classical bias-variance tradeoff. The first term corresponds to the `best approximation' of the ground truth measure $P_0$ by elements in the class $\mathcal F$, and the second term follows from
the metric entropy of $\mathcal F$ in the $C^1$ norm. Our result thus identifies the latter as a natural \textit{statistical complexity measure} that yields an upper bound on the stochastic fluctuations of any ODE-MLE; 
see Section \ref{sec:general-results} for details.

To obtain this result, we first derive natural boundary conditions on the variational class $\mathcal F$ to ensure that the statistical objective can be formulated over $\mathcal F$ in its standard form, by ensuring that all pullback distributions under the associated ODE flow maps
possess the same support and are absolutely continuous. Then, to prove Theorem \ref{thm:general}, we derive novel analytical Lipschitz estimates for ODEs---bounding the distance between terminal-time transport maps induced by ODE flows, and between their corresponding pullback distributions, in terms of the velocity fields that underlie them. 
Such Lipschitz properties hold true \textit{locally} on sets of velocity fields which are uniformly bounded in a certain sense; see \eqref{eq:GeneralThmAssumption}. These estimates, detailed in Section \ref{sec-proof-general}, are then combined with existing convergence theory for general sieved maximum likelihood estimators \cite{VDG00} in Hellinger loss. They crucially allow us to relate so-called bracketing entropy rates, which are commonly required in theory for M-estimation \cite{VDG00}, to $C^1$-metric entropy rates of ODE-based estimators; see the proof of Theorem \ref{thm:general} for details.

Section \ref{Ck-theory} studies the case where $P_0$ possesses a $C^k$ density and $\mathcal F$ likewise consists of $C^k$-smooth velocity fields. The main convergence theorem in this section is Theorem \ref{thm-holder}. 
A key intermediate result establishes the existence of a $C^k$ velocity field coupling $P_0$ with the reference distribution \emph{and} vanishing appropriately normal to the boundary. The existence of a $C^k$ velocity field is established in our companion paper \cite{ren23a}; it is constructed using a triangular Knothe--Rosenblatt (KR) map and straight-line trajectories. The required boundary behavior is proven here, using anisotropic regularity properties of KR maps shown in \cite{TransportMiniMax}; see Theorem \ref{Thm:existencCkField}. 
Due to the additional dimension arising from the space-time structure of the ODEs, our rates of convergence are slightly suboptimal in a statistical minimax sense. Achieving minimax-optimality in this context will likely require a more refined choice of $\mathcal F$, e.g., as an anisotropic regularity class or via penalization; we leave this for future work. See also Remark \ref{rem:finalrates}.

Finally, Section \ref{sec:NN-theory} considers the case where $\mathcal F$ is given via neural network classes. Using scalings of ReLU network classes (i.e., width, depth, sparsity, and norm constraints scaling in the sample size $n$) derived in the seminal work of \cite{Schmidt_Hieber_2020}, we prove the relevant metric entropy and approximation bounds needed to apply our general result from Theorem \ref{thm:general}. In order to satisfy the regularity and boundary conditions required for our ODE setting, we make some modifications to the standard constructions of neural network classes: first, we need to work with the squared ReLU$^2$ activation functions to ensure $C^1$ regularity; and second, we multiply standard neural network classes with certain component-wise cutoff functions to create an ansatz space satisfying appropriate boundary conditions. See Section \ref{sec:NN-theory} for details. Our choice of a slightly more regular ReLU$^2$ activation, interestingly, may relate to the fact that smooth activation functions are often used in practical applications of continuous normalizing flows.

\subsection{Related work} 
The past decade has seen the emergence of increasingly expressive and powerful models for complex probability distributions that employ \emph{transportation of measure}: The central idea is to express the ``target'' distribution of interest as the pullback or pushforward of a simple reference distribution (e.g., uniform or standard Gaussian) by a learned (measurable) map. Samples from the target distribution are then produced simply by evaluating this map on samples from the reference; this enables generative modeling \cite{kingma2018glow}. When the map is invertible and sufficiently smooth, the map and the reference density yield a closed-form expression for the target density, enabling density estimation \cite{tabak2010density,anderes2011two,TransportMiniMax}. Given a family of transport maps and a reference measure, variational inference can be cast as minimization of a suitable divergence over the resulting family of pushforward measures \cite{el2012bayesian,rezende15}.

A central question in designing these methods is how to represent or parameterize the map. Initial applications of transport in machine learning emphasized normalizing flows \cite{rezende15,papamakarios2021normalizing,NormalizingFlowIntro}, which are compositions of simple, parametric, invertible transformations whose Jacobian determinants are, by design, easy to evaluate. A considerable variety of such transformations have been proposed \cite{CouplingFlows,autoregressiveflow,NeuralAutoFlow,wehenkel2019unconstrained}, sometimes under the broader label of ``invertible neural networks.'' In other settings, triangular maps \cite{bogachev2005triangular,MMPS16,ZM22a,ZM22b,BMZ20,ISPH21} and parametric approximations of optimal transport maps \cite{el2012bayesian,huang2020convex} have been popular. More recently, there has been considerable interest in ``continuous-time'' (i.e., differential) notions of normalizing flows. As explained earlier in this introduction, these models can be formalized as ODE systems \eqref{eq:neuralODEintro} and are the central topic of this paper. 

Questions of \emph{function approximation} with neural ODEs have been studied in \cite{li2022deep, ishikawa2022universal}. \cite{ishikawa2022universal} shows that neural ODEs are univeral approximators of smooth diffeomorphisms on $\mathbb{R}^d$ in appropriate Sobolev norms. \cite{li2022deep} adapts ideas from dynamical systems to show that neural ODEs are universal approximators of continuous functions from $\mathbb{R}^d$ to $\mathbb{R}^m$ (hence, not only diffeomorphisms) in a $L^2$ sense, for $d \geq 2$. Both papers compose the flow map of the ODE with a terminal mapping, meant to represent a classification or regression layer.
Yet these universal approximation results do not characterize approximation \emph{rates}, e.g., relating bounds on an approximation error to the size of the network representing the velocity field. Function approximation is also different than our present focus of statistical recovery guarantees.

The results on diffeomorphism approximation in \cite{ishikawa2022universal} do translate to universal approximation of certain classes of distributions, in a weak sense and in total variation. \cite{ruizbalet21} also proves universal approximation for certain target distributions, in Wasserstein-1 distance, using a rather different approach that is discrete and constructive. 
Our companion paper \cite{ren23a}, in contrast, establishes approximation \textit{rates} for neural ODE representations of distributions with $C^k$-smooth densities, and shows that there exist neural network representations of the velocity field $f$, with size explicitly bounded in terms of the regularity of the densities, that achieve efficient approximation.  
There has also been relevant work on approximation theory for transport maps that are not constructed via ODEs. For example, \cite{ZM22a,ZM22b} investigate sparse polynomial and neural network approximations of triangular (Knothe--Rosenblatt) maps, formulating \textit{a priori} descriptions of an ansatz space that achieves exponential convergence in the case of analytic densities. A broader framework for understanding the distributional errors of transport map approximations is proposed in \cite{baptista2023approximation}.

From the statistical perspective, one must address the impact of using a finite number of samples $n$ to estimate the ODE velocity field and the resulting pushforward or pullback densities. To our knowledge, there has been almost no statistical convergence analysis of neural ODEs. Perhaps the sole exception is \cite{ruiz2023control} (building on \cite{ruizbalet21}), which analyzes neural ODE-type models from a controllability perspective, explicitly constructing finite-difference approximations of the target density using a neural network velocity field with ReLU activations. Sample complexity results follow from assessing the convergence of the $n$-sample empirical measure to its finite-difference approximation. This construction is rather different from the maximum likelihood training typically used in neural ODEs, and similarly its analysis uses different tools than those we exploit here. Moreover, \cite{ruiz2023control,ruizbalet21} do not assume any smoothness in the reference and target densities. 

For direct parameterization of transport maps, i.e., not using an ODE construction, \cite{TransportMiniMax} develops a general statistical convergence theory for transport-based estimation of H\"older-smooth densities, %
and we build on those results here. There is also a growing body of work on the statistical estimation of \emph{optimal transport maps}; see, e.g., \cite{MBNW21,divol2022optimal}. As a corollary of such results, for a fixed reference distribution, one can obtain rates of convergence for optimal transport-based density estimation in Wasserstein distances \cite[Remark 5]{HR20}. We emphasize, however, that these constructions are distinct from the ODE models of interest here.

Let us also comment briefly on sampling and generative modeling methods based \textit{stochastic} differential equations (SDEs). As mentioned in the opening, such methods generally seek to learn the drift term of an SDE so that marginal distribution at a particular time (e.g., $t=0$ or $t \to \infty$) is a good approximation of the target distribution. Score-based diffusion models \cite{song2020score,song2021maximum,yang2022diffusion} are a widely used approach of this type. 
Yet these models---and approaches for elucidating their approximation properties and statistical behavior---are rather different in character from deterministic ODEs, 
due to the presence of the diffusion term. 
Also, the estimation problem in score-based diffusions involves an objective that is quadratic in the desired score; this is much simpler than the log-likelihood we analyze here, which is highly nonlinear in the velocity $f$ (see \eqref{eq:MLEobjective}). Very recent literature has established near-optimal minimax rates for the estimation of smooth densities (in, e.g., total variation distance) with score-based diffusion models \cite{DiffusionModelRate}; parallel efforts have analyzed the convergence of such models for target distributions supported on low-dimensional manifolds \cite{chen2023score}.
Yet it is worth noting that deterministic ODEs have a role in diffusion models as well. For instance, the deterministic ``probability flow ODE'' \cite{song2020score} (see also \cite{song2020denoising}) is sometimes used instead of a time-reversed SDE for sampling in this context, as numerical integration of the ODE can be more accurate and efficient than the comparable discretized SDE \cite{chen2023probability}.

\section{General ODE-based density
  estimators}\label{sec:general-results}

In this section we derive a key convergence result, Theorem \ref{thm:general}, which characterizes a convergence rate for
  general ODE-based density estimators: specifically, we consider
  estimators obtained through a \textit{velocity field} learned from
  the data, which in turn generates a pullback density estimate.
In subsequent sections, we apply this result to two relevant
classes of velocity fields: the class of $k$-times continuously
differentiable velocity fields (Section \ref{Ck-theory}) and neural
network parameterizations of velocity fields (Section
\ref{sec:NN-theory}).

\subsection{Notation}

We introduce a number of notations and definitions which are needed
throughout the paper.

\subsubsection{Norms for vectors and matrices}
For a vector, we denote by $\|\cdot\|_2$ its $l^2$-norm (the Euclidean norm), $\|\cdot\|_\infty$ its $l^\infty$-norm, and $\|\cdot\|_0$ its $l^0$-norm (number of nonzero entries). For a matrix, we denote by $\|\cdot\|_2$ its operator norm induced by the $l^2$-norm on vectors, $\|\cdot\|_\infty$ its operator norm induced by the
$l^\infty$-norm on vectors, $\|\cdot\|_0$ its
$l^0$-norm (number of nonzero entries), $\|\cdot\|_{\infty, \infty}$
its $l^\infty$ norm (the maximum absolute value of its entries), and $\|\cdot\|_F$ its Frobenius norm.

\subsubsection{%
Derivatives and function spaces}
Let $d\ge 1$ and let $D\subseteq \R^d$ be a bounded
domain. For $k\in\N \coloneqq \{1,2,\dots\}$, we denote by $C^k(D)$ the
space of real-valued functions $f: D\rightarrow \mathbb{R}$ which are
$k$-times continuously differentiable. Similarly, for $m\ge 1$ we
shall write $C^k(D, \R^m)$ for the space of $k$-times differentiable
vector-valued functions taking values in $\R^m$. To denote partial
derivatives of functions, we use standard multi-index
notation.  %
Given a multi-index $\bsv = (v_1,v_2,\dots,v_d) \in \mathbb{N}^d$, we
will write
$\partial^\bsv f(x)= \frac{\partial^{|\bsv|}}{\partial x_1^{v_1} \dots
  \partial x_d^{v_d}} f(x)$ for the $|\bsv|$-th order partial
derivative of $f$, whenever it exists. 

For $f\in C^1(D)$, we denote its gradient by $\nabla
f:D\to\R^{d}$. Similarly, if $f\in C^1(D,\R^m)$ for $m\ge 1$, $\nabla
f:\R^d\to \R^{m\times
  d}$ denotes the Jacobian (or gradient matrix) of $f$. If
$f$ depends on multiple variables---say a `space variable'
$x\in \R^d$ and a `time variable' $t\in
\R$---we will use the standard notation
$\nabla_x f(x,t)$ for the $x$-gradient of
$f$. Similarly, for a multi-index $\bsv\in
\N^d$,
$\partial^\bsv_xf(x,t)$ denotes the corresponding partial
derivative with respect to $x$. For continuous $f:D\to \R$, we let $\|f\|_{C(D)} \coloneqq \sup_{x\in D}|f(x)|$ and for a $C^k$ function $f\in C^k(D)$ with $k > 1$, we let $\|f\|_{C^k(D)} = \sup_{|\bsv|\leq k}\|\partial^\bsv
  f\|_{C(D)}$. For a vector field $f\in C^k(D, \mathbb R^m)$, we define $\|f\|_{C^k(D,\mathbb R^m)} = \sup_{j\in \{1,2,...,m\}}\|f_j\|_{C^k(D)}$, and we may sometimes omit the $\mathbb R^m$ by abuse of notation. If
  $f:D\to\R^m$ is Lipschitz continuous, we write $|f|_{{\rm
      Lip}(D)}$ to denote its Lipschitz constant.

For $D\subseteq\R^d$ Borel measurable, a Borel measure $\mu$ on $D$, and $p\in [1,\infty]$, we write $L^p(D,\mu)$ to denote the usual space of $p$-integrable functions w.r.t.\ $\mu$ on $D$. If $\mu$ is the Lebesgue measure, we write $L^p(D)$ instead. In case there is no confusion about $D$, we also use the notation $L^p(\mu)$.

\subsubsection{Transportation of measure} 
Let $d\ge 1$ and let $D_1$, $D_2 \subseteq \R^d$ be Borel
measurable sets equipped with the Borel $\sigma$-algebra. Then,
for any measurable function $T:D_1\to D_2$ and
probability %
distribution $\pi$ on $D_1$, we denote the
\textit{pushforward distribution} of $\pi$ under $T$ by
$T_\sharp \pi$, given by $T_\sharp \pi(A)=\pi\left (T^{-1}(A) \right )$ for any
measurable subset $A\subseteq D_2$. Given another
probability %
distribution $\rho$ on $D_2$, we say that $T$ pushes forward
$\pi$ to $\rho$ if $T_\sharp \pi =\rho $. Since we will deal only with
measures that possess densities with respect to Lebesgue measure, we
will occasionally use the same symbol to represent
a probability measure and its Lebesgue
density, in a slight abuse of notation. If
additionally $T$ is bijective, differentiable, and invertible with a
continuously differentiable inverse %
$T^{-1}:D_2\to D_1$ (i.e., $T$ is a diffeomorphism),
then the pushforward density $T_\sharp \pi $ is given by
$\rho(x)= \pi (T^{-1}(x) ) |\det \nabla T^{-1} (x)|$ (the change-of-variables formula).
In this case, we also denote the
\textit{pullback} density of $\rho$ under $T$ by $T^\sharp\rho$, and
it holds that
\[\pi(x) = [(T^{-1})_\sharp \rho ] (x)= [T^\sharp \rho](x) = \rho
  (T(x) ) |\det \nabla T(x)| .\] 
 
\subsection{Nonparametric density estimation via ODEs}
\label{sec:ODEest}
For $d\ge 1$, we denote the $d$-dimensional unit cube by
  \begin{equation*}
    D=[0,1]^d\subset \R^d
  \end{equation*}
throughout.\footnote{
  Much of what follows could also be extended to more general (bounded
  and sufficiently regular) domains $D\subseteq \R^d$ at the expense
  of additional technicalities.} We will be concerned with the problem
of nonparametric density estimation on $D$, where the observations are
given by independent and identically distributed (iid) samples 
\begin{equation}\label{eq:DataGenerating}
  (Z_i : i=1,\dots,n), \ Z_i\stackrel{\text{iid}}{\sim} P_0,    
\end{equation}
for $P_0$ some unknown probability measure supported on $D$. Our goal
is to infer $P_0$ from $(Z_i : i=1,\dots,n)$. We assume throughout
that $P_0$ possesses a Lebesgue density which we denote by $p_0$. We
denote the $n$-fold product measure of $P_0$ by $P_0^n$, and
expectations with respect to $P_0^n$ by $\E_{P_0}^n$.

Given any sufficiently regular `velocity vector field'
$f:D\times [0,1]\to \R^d$ and any initial condition $x\in D$, consider
the following ordinary differential equation
\begin{equation}\label{eq:ODE}
  \begin{cases}
    \frac{d}{dt}X^f(x,t) &= f(X^f(x,t), t),\qquad t\in [0,1],\\
    X^f(x,0) &= x.
  \end{cases} 
\end{equation} %
If $f$ is Lipschitz continuous and if the `flow lines' of $f$ do
  not leave the domain $D$ (a key technical condition to be discussed
  in more detail below), %
  then, by the Picard--Lindel\"of theorem, \eqref{eq:ODE} is solvable
  and induces trajectories %
  $t \mapsto X^f(x,t):[0,1]\to D$ for each $x\in D$. 
  They satisfy
\begin{equation}\label{eq:flowmap}
  X^f(x,t) = x + \int_0^tf(X^f(x,s),s)ds,\qquad  t\in [0,1], ~x\in D.
\end{equation}
We will refer to the mapping $x \mapsto X^f(x, t)$ as the time-$t$ \textit{flow map} of the ODE.
The transport map obtained by
evaluating this flow map at the terminal time $t=1$ is denoted by
$T^f \coloneqq X^f(\cdot,1)$. Since the
trajectories of the ODE (\ref{eq:ODE}) are unique and do not
intersect, the inverse $(T^f)^{-1}$ is also well-defined as a map
from %
$T^f(D)$ onto $D$; both maps %
$T^f$ and $(T^f)^{-1}$ can then be used to transform probability
measures.

\subsubsection{Admissible velocity fields}

Throughout, with $D=[0,1]^d$, we denote the cylindrical $d+1$-dimensional
  `space-time' %
  unit cube by
  \begin{equation*}
    \Omega = D\times[0,1]\subset\R^{d+1}.
  \end{equation*}
In the setting considered here, where the
support of the unknown density is known, it is natural to consider
only ODE flows which \textbf{(i)} do not leave the domain $D$, and
\textbf{(ii)} for which flow maps $\{ X^f(\cdot,t):t\in [0,1] \}$ are
diffeomorphisms $D\to D$. In order to ensure those properties, along
with the existence and uniqueness of the solution to (\ref{eq:ODE}),
we need to introduce %
boundary conditions on the class of velocity fields
considered. Specifically, denoting by $\nu_x$ the outward pointing
normal vector at any point $x\in \partial D$ where $\nu_x$ is
  well-defined, we let
\begin{equation}\label{eqn:VelocityField}
  \mathcal{V} = \Big\{f\in C^1(\Omega, \R^d)~:~ f(x,t)\cdot \nu_x\equiv 0~\text{for all}~(x,t)\in \partial D \times [0,1]\Big\}.
\end{equation}
This condition ensures that there is no flow outside of $D$.  In
  fact, it even implies the maps $x\mapsto X^f(x,t):D\to D$ to be
  $C^1$-diffeomorphisms for every $t\in [0,1]$:

\begin{lemma}\label{lemma:Diffeomorphism}
  Suppose that $f\in \mathcal V$, for $\mathcal V$ given by
  (\ref{eqn:VelocityField}). Then, for any $t\in [0,1]$, the ODE flow
  map $X^f(\cdot, t):D\to D$ at time $t$ is a
  diffeomorphism. In particular, the time-one map
  $T^f=X^f(\cdot,1):D\to D$ is a diffeomorphism, and the pullback
  density $(T^f)^\sharp \rho$ of any density $\rho$ supported on $D$ is given by
  \[(T^f)^\sharp \rho (x)= \rho(T^{f}(x))\det\nabla T^{f}(x),\qquad
    x\in D.\]
\end{lemma}

The proof of Lemma \ref{lemma:Diffeomorphism}, which is based on tools
from ODE theory \cite{H02} as well as Gr\"onwall's inequality,
can be found in Appendix \ref{sec2-proofs}. %
While the assumption that $f\in C^1$ %
ensures existence of a unique solution to the 
ODE \eqref{eq:ODE}, the additional requirement of the normal
component $ f\cdot \nu_x$ vanishing at the boundary $\partial D$
guarantees that the trajectories remain inside the unit cube $D$
at all times. In particular, if ${\rm supp}(\rho)=D$, then the %
`interpolating' %
distributions %
$(X^f(\cdot,t))^\sharp \rho$ for $t\in [0,1]$, all possess common
support $D$.

\subsubsection{Maximum likelihood objective}
Let $\mathcal F\subseteq \mathcal V$ be any class of admissible
velocity fields, and let us fix some reference density $\rho$ on
$D=[0,1]^d$. Assume that $\rho$ is strictly positive and upper
bounded. By Lemma \ref{lemma:Diffeomorphism}, each time-one flow map
$T^f$ is a diffeomorphism, so that we may form the collection of
pullback densities
$(T^f)^\sharp(x) \rho=\rho(T^f(x))\det \nabla T^f(x)$ as an
approximating class for the unknown ground truth distribution
$p_0$. In this paper, we %
study estimators
maximizing the likelihood: that is, with $Z_i\sim P_0$ i.i.d., we
  let %
\begin{equation}\label{eq:MLEobjective}
  \hat{f} \in \argmax_{f \in\mathcal{F}} \mathcal J(f),\qquad \mathcal J(f) := \left(\sum_{i=1}^n\log(\rho(T^f(Z_i)) + \log\det\nabla T^f(Z_i)\right).
\end{equation}
Any such $\hat f$ naturally gives rise to a plug-in estimator
$(T^{\hat f})^\sharp \rho$ for the data-generating density $p_0$ via its pullback density
\begin{equation}\label{ODE-MLE}
  (T^{\hat f})^\sharp\rho = \rho(T^{\hat{f}}(x))\det\nabla  T^{\hat{f
    }}(x), \qquad  x\in D.
\end{equation}
The rate of convergence towards $p_0$ in terms of $n$ and
$\cF$ will be the subject of our main results. Note that since all
the $(T^f)^\sharp \rho$ have common support $D$, the likelihood
objective is well-defined and finite for all
$f\in \mathcal F\subseteq \mathcal V$. We will refer to estimators (\ref{ODE-MLE}) as \emph{ODE-MLE} estimators (over the
class $\mathcal F$).

\subsection{Main convergence result}
We are now ready to formulate the first main result of this paper,
Theorem \ref{thm:general}, which provides a \textit{general}
convergence rate for ODE-MLE estimators.
The result is stated in terms of two key characteristics of the
class of velocity fields $\mathcal F$; the first of which is the `best approximation error'  $h\big((T^{f^*})^{\#}\rho,p_0\big)$ of $p_0$ by
any pullback distribution $(T^f)^\sharp \rho$ over the class $f\in \mathcal F$. The second key quantity is the metric entropy
of $\mathcal F$ in $C^1$, which is identified as a key \textit{complexity measure} that gives an upper bound for the `stochastic fluctuations' of ODE-MLE estimators over any class $\mathcal F\subseteq \mathcal V$ via the inequality (\ref{entropy-req}); see Remark \ref{C1-remark} for discussion.

We recall some standard definitions. Again let $D=[0,1]^d$ and $\Omega=D\times [0,1]$. The Hellinger distance between any two probability densities $p_1$, $p_2\in L^1(D)$ is
\[h(p_1,p_2) = \Big(\int_{D}\Big[\sqrt{p_1(x)}- \sqrt{p_2(x)}\Big]^2
  dx\Big)^\frac{1}{2}.\]
For any normed space $(X,\|\cdot\|)$ and subset $A\subseteq X$, we denote the \textit{metric entropy} of $A$ by $H(A,X,\tau)=\log N(A, X, \tau)$ $(\tau>0)$, where $N(A, X, \tau)$ is the \emph{covering number} of $A$,
\begin{equation*}
      N(A,X,\tau):=\min\setc{N\in\N}{\exists x_1,\dots,x_N\in X\text{ such that }A\subseteq\bigcup_{j=1}^NB_{\tau}(x_i)}.
\end{equation*}
Here, we used the notation $B_\eps(x) \coloneqq \set{\tilde x\in X}{\|x-\tilde x\|\le\eps}$ for all $\eps>0$, $x\in X$.

\begin{ass}[Ground truth and reference density]\label{ass:densities}
Let  $p_0, \rho$ be two %
      probability densities such that $\rho$ is Lipschitz continuous
      and for some $0<\kappa \leq K <\infty$%
    \begin{equation}\label{eq:Kkappa}
        p_0(x)\leq K \qquad\text{and}\qquad
        \kappa \leq \rho(x)\leq K\qquad\qquad\forall x\in D.
      \end{equation} %

\end{ass}

Now we define a useful subset of the admissible velocity fields $\mathcal V$ from \eqref{eqn:VelocityField}.

\begin{ass}[Boundedness of $\mathcal F$]\label{ass:Fbounded}
Let $\cF\subseteq \cV$ be a class of admissible velocity fields such that
for some $r >0$, \begin{equation}\label{eq:GeneralThmAssumption}
        \sup_{f\in \mathcal F} \Big(\|f\|_{C^1(\Omega)} +\sup_{t\in [0,1]} |\nabla_xf(\cdot, t)|_{\rm Lip(D)}\Big)=:r <\infty.
      \end{equation}
\end{ass}

Given a class $\mathcal F\subseteq \mathcal V$, we now define the crucial square root metric entropy integral of $\mathcal F$, which plays a key role in determining the convergence rate of the ODE-MLEs taken over $\mathcal F$. For any $R>0$, we shall denote 
\[ I(\mathcal F, R) := R+ \int_0^R H^{1/2}(\mathcal{F},
  C^1(\Omega), \tau) d\tau\qquad\text{for all }R>0. \]
For technical reasons, instead of working directly with $I(\mathcal F,R)$, we shall work with an \textit{upper bound} for $I(\mathcal F,R)$. We fix any such upper bound $\Psi$, satisfying $\Psi(R)\ge I(\mathcal F,R)$ on $(0,\infty)$. The following assumption on the growth of $\Psi$ is a standard technical requirement in the literature on nonparametric M-estimators; see, e.g., \cite{VDG00, NVW18} or Theorem \ref{sara10} below. It is required in standard `slicing' concentration arguments based on empirical processes (cf.~the proofs of Theorems 7.4 and 10.13 in \cite{VDG00}), and is satisfied for all sufficiently smooth %
classes of functions; see for instance our examples in Sections \ref{Ck-theory} and \ref{sec:NN-theory} below.

\begin{ass}\label{ass:psi}
   Suppose that the upper bound $\Psi:(0,\infty)\to\R$ is such that $R\mapsto\Psi(R)/R^2$ is non-increasing on $(0,\infty)$.
\end{ass}

\begin{theorem}[Convergence of general ODE-MLEs]\label{thm:general}

Suppose that $p_0$, $\rho$, and $\mathcal F\subseteq \mathcal V$ and $\Psi$ are such that Assumptions \ref{ass:densities}, \ref{ass:Fbounded}, and \ref{ass:psi} are fulfilled with some constants $0<\kappa<K$ and $r >0$, and consider the i.i.d.~sampling model (\ref{eq:DataGenerating}) with $p_0$. Let $\hat f\in\argmax_{f\in\cF}\cJ(f)$ denote an ODE-MLE estimator as in \eqref{eq:MLEobjective}.
Then, there are constants $C,C'>0$ only depending on $d$, $\kappa$, $K$, $r$, and $|\rho|_{{\rm Lip}(D)}$ such that for all $n\ge 1$ and $\delta_n >0$ with
    \begin{equation}\label{entropy-req}
      \sqrt{n}\delta_n^2 \geq C \Psi(\delta_n),
    \end{equation}
all $f^*\in \mathcal F$ and all
$\delta\ge \delta_n$, we have the concentration inequality
\begin{equation}\label{conc-ineq}
P_0^n\Big(h\big((T^{\hat{f}})^\sharp\rho, p_0\big) \geq C \big[ h\big((T^{f^*})^\sharp \rho, p_0\big) + \delta \big] \Big)\leq C\exp\Big(-\frac{n\delta^2}{C}\Big),
\end{equation}
and such that the mean squared error is bounded as follows:
\begin{equation}\label{eq:expectationbound}
\E_{P_0}^n[h^2((T^{\hat{f}})^\sharp\rho, p_0)]\le C' \Big(h^2((T^{f^*})^\sharp\rho, p_0) + \delta^2_n + \frac{1}{n}\Big).
\end{equation}        
\end{theorem}

The above theorem is \textit{non-asymptotic} in that the
constants involved are independent of $n$ and of the variational class
$\mathcal F$ of velocity fields. Thus, when applying the theorem, one may choose an `approximating sequence' of classes
$\mathcal F=\mathcal F_n$ as the number of statistical
observations grows. We have omitted %
use of $\cF_n$ merely for notational convenience.

\begin{remark}[$C^1$ metric entropy]\label{C1-remark}
The $C^1$ metric entropy of $\mathcal F$ is a natural complexity measure in the context of ODE-MLEs, in light of the intuition that the pointwise (i.e., $L^\infty$) distance  between two pullback densities can be bounded by the $C^1$-norm of the corresponding inducing velocity fields. The latter is rigorously proven in Section \ref{sec-proof-general} below, using analytical tools from ODE theory. The $L^\infty$ complexity of the pullback densities induced by a class $\mathcal F$ in turn yields a bound for the bracketing metric entropy, which is well known to play a key role in quantifying convergence rates of maximum likelihood-type estimators; see, e.g., \cite{VDG00}. Whether the $C^1$ norm is both necessary and sufficient for characterizing convergence rates, or whether a weaker norm than $\|\cdot\|_{C^1}$ would suffice (yielding smaller entropy integrals), is an interesting question for future research.
\end{remark}

\subsection{Proof of Theorem \ref{thm:general}}\label{sec-proof-general}

The proof of Theorem \ref{thm:general} relies on combining general
convergence results for nonparametric M-estimation developed in
\cite{VDG00,VDG01} with key analytical Lipschitz estimates, which will
allow us to derive metric entropy complexity bounds for the class of
densities induced by any variational class $\mathcal F$ of velocity
fields. A similar approach for obtaining convergence rates has been
used before in the context of inverse problems, where 
Lipschitz properties for the `forward map' permit to
bound the metric entropy of the observed regression functions; see,
e.g., \cite{NVW18, GN20, AW21}.

\subsubsection{A key statistical convergence rate result}
To begin, we will derive a statistical convergence result in Theorem
\ref{sara10} which regards so-called \textit{sieved} maximum likelihood
estimators---that is, MLEs which are taken over growing approximating classes as the number of samples $n$ increases. This convergence
result follows (up to minor adaptations) from classical results in
Chapter 10 of \cite{VDG00}; we nevertheless include it here since it
plays a key role in our
derivations. %
Let again $Z_1,\dots ,Z_n\in D$ be i.i.d.~samples from some $P_0$
with Lebesgue density $p_0$. Suppose that $(\mathcal P_n)_{n \ge 1}$ is a
sequence of approximating classes of densities on $D$. Then the
sieved MLE is defined by
\[ \hat p_n =\arg\max_{p\in \mathcal P_n} \sum_{i=1}^n \log
  p(Z_i). \]

In the following we will require the bracketing metric entropy for functions on $D=[0,1]^d$. This notion of entropy is different from the standard metric entropy
due to its `joint' $L^2$ and pointwise structure, but it can
straightforwardly be compared to the $L^\infty$ metric entropy; see
Lemma \ref{lemma:InfToBracketing} %
in the appendix. Let $\mu$ be a Borel measure on $D=[0,1]^d$ and recall the shorthand $L^2(\mu):=L^2(D,\mu)$.

\begin{definition}[Bracketing metric entropy]\label{def:bme}
Let $\mathcal G$ be a class of real-valued functions on $D$. Then let $N_B(\mathcal G, L^2(\mu), \tau)$ be the smallest value of $N$ such that there exist pairs of functions $\{ g_j^L,g_j^U\}_{j=1}^N$ with $\|g_j^L-g_j^U\|_{L^2(\mu)}\le \tau$ such that for every $g\in \mathcal G$, there exists $j$ with
\[ g_j^L \le g\le g_j^U ~~\text{on}~D. \]
The \emph{$L^2(\mu)$-bracketing metric entropy} of $\mathcal{G}$ is
  $H_{B}(\mathcal{G}, L^2(\mu), \tau) = \log N_B(\mathcal{G},
  L^2(\mu),\tau)$.

\end{definition}

Given the approximating classes $\mathcal P_n$, it turns out that
the key measure of statistical complexity featuring in our convergence rate result is the
bracketing metric entropy of the \textit{square root densities}
induced by $\mathcal P_n$. Specifically, let us fix some element
$p_n^*\in \mathcal P_n$ and denote
\[ \mathcal Q_n^\ast:=\Big\{ \sqrt {\frac{p+p_n^*}{2}} : p\in \mathcal
  P_n\Big\}.  \] While $p_n^*$ can be chosen arbitrarily, typically
one aims to choose it to be some `best approximation' of $p_0$ within
the class $\mathcal P_n$. We then define the \emph{bracketing
  metric entropy integral}%
\begin{equation}\label{entropy}
  I_B(\mathcal P_n, R, p^*_n) := R+ \int_0^R H_B^{1/2} (\mathcal Q_n^\ast, L^2(p_n^\ast), \tau) d\tau\qquad\text{for all }R>0.
\end{equation}
As before, we will use the notation $\Psi:(0,\infty)\to\R$ for an upper bound satisfying $\Psi(R)\ge I_B(\mathcal P_n, R, p^\ast_n)$ for all $R>0$.

\begin{ass}\label{ass:gen}
Suppose that for some constants $0<c<K$, we have that $p_0\le K$ and $p_n^*\ge c$ for all $n\ge 1$. Moreover, suppose that $\Psi$ is such that
$R\mapsto \Psi(R)/R^2$ is non-decreasing (for all $n\ge 1$).
\end{ass}

\begin{theorem}[cf.~Theorem 10.13 in \cite{VDG00}]\label{sara10}
  Suppose that $p_0$, $p_n^*$, and $\mathcal P_n$ satisfy Assumption \ref{ass:gen} for some $0<c<K$. There is a constant $C>0$ depending only on $c$ and $K$ such that for any $n\ge 1$ and  $\delta_n>0$
  satisfying
  \begin{equation}\label{entropy-cond}
    \sqrt n \delta_n^2 \ge C \Psi (\delta_n)
  \end{equation}
  and any $\delta\ge \delta_n$, we have the concentration inequality
  \[ P_0^n\Big( h(\hat p_n, p_0)\ge C\big[ h(p_n^\ast, p_0) +\delta
    \big] \Big)\le C\exp\Big(- \frac{n\delta^2}{C}\Big) . \]
\end{theorem}
Theorem \ref{sara10} is a variant of \cite[Theorem 10.13]{VDG00}.
  We provide the argument in Appendix \ref{sec2-proofs}, indicating in
  particular the required modifications compared to \cite[Theorem
  10.13]{VDG00}.

\subsubsection{Analytical estimates for ODE-based measure transport}

In order to utilize the convergence result from Theorem \ref{sara10}, we need a number of `stability'
properties which relate the distance between two ODE velocity fields
to their corresponding transport maps and pullback distributions. Recall the notation $T^f=(X^f(\cdot,1))$ for the
time-one flow map. The following lemma shows that the map
$f \mapsto T^f$ is locally Lipschitz continuous as a mapping from
$C^1(\Omega)$ to $C^1(D)$, on sets of velocity fields which are
uniformly bounded in an appropriate sense.

\begin{lemma}\label{thm:FlowMapBound}
  Fix $r>0$.  Then for all velocity fields $f$,
  $g\in\mathcal{V}$ %
  (see~\eqref{eqn:VelocityField}) satisfying
  \begin{equation}\label{B-bound}
    \max \Big\{\|f\|_{C^1(\Omega)}, \sup_{t\in [0,1]}
    | \nabla_x f(\cdot,t)|_{\rm Lip(D)}\Big\} \le r,\qquad  \|g\|_{C^1(\Omega)}\leq r,
  \end{equation}
  it holds with
    $C := \max\{e^{dr},\frac{re^{3dr} + 2de^{2dr}}{2\sqrt{d}r}\}$
    that
  \[\|T^f - T^g\|_{C^1(D)} \leq C \|f-g\|_{C^1(\Omega)}.\]
\end{lemma}

The proof relies on Gr\"onwall-type estimates for ODEs, and can
be found in Appendix \ref{sec2-proofs}. The next lemma shows that the
$L^\infty$-norm between two pullback densities is bounded by the
$C^1$-norm between their corresponding transport maps. Again, we
defer the proof to Appendix \ref{sec2-proofs}.

\begin{theorem}\label{thm:LInfDensityBound}
  Let $\rho: D\to [0,\infty)$ be a Lipschitz probability density and
  $T, G:D\rightarrow D$ two diffeomorphisms. Let
  $\lambda_1(x)\geq\cdots\geq\lambda_d(x) \textcolor{blue}{>} 0$ and
  $\eta_1(x)\geq\cdots\geq\eta_d(x)\textcolor{blue}{>}  0$ be the singular values of
  $\nabla T(x)$ and $\nabla G(x)$ respectively. Then, it holds that
  \[ \|T^\sharp \rho - G^\sharp \rho \|_{C(D)} \le \|T- G\|_{C^1(D)}
    \Big(|\rho|_{{\rm Lip(D)}}\|T\|_{C^1(D)}^d+ \tilde Cd^2\|\rho
    \|_{C(D)}\Big), \] %
    where
  \begin{equation}\label{tildeC}
    \tilde C:=\sup_{x\in D} \frac{\exp{\left(\sum_{i=1}^d\frac{|\lambda_i(x) - \eta_i(x)|}{\lambda_d(x)}\right)}\prod_{i=1}^d\lambda_i(x)}{\min\{\lambda_d(x), \eta_d(x)\}}.
  \end{equation} 
	
\end{theorem}

See Appendix \ref{sec2-proofs} for the proof. Lemma
\ref{thm:FlowMapBound} and Theorem \ref{thm:LInfDensityBound} together
yield that the map $f\mapsto (T^f)^\sharp \rho $ is locally Lipschitz
continuous on classes of velocity fields for which the constants $r$
and $\tilde C$ from (\ref{B-bound}) and (\ref{tildeC}) can be
controlled uniformly. From (\ref{tildeC}), we can see that this
requires uniform control over the largest and smallest singular values
of the Jacobian matrix of $\nabla T^f$. The next result states that
for classes $\mathcal F$ which are bounded in $C^1(\Omega)$-norm, such
uniform bounds hold true.

\begin{theorem}\label{thm:SingularValueBound}
  Let $\cF\subseteq C^1(\Omega,\R^d)$ such that
    $\sup_{f\in \mathcal F} \|f\|_{C^1(\Omega)}=:M<\infty$. Then %
  for all $f\in \mathcal{F}$
  \[ \sup_{x\in D}\|\nabla (T^{f})(x)\|_2 \leq 1 +
    dMe^{dM},\] where $\|\cdot\|_2$ denotes the
  $\R^d\to\R^d$ operator norm.  Consequently, the largest and smallest singular
  values $\lambda_1(x)$ and $\lambda_d(x)$ of $\nabla(T^{f})(x)$, are
  respectively upper and lower bounded as
  \[ \sup_{f\in\mathcal F}\sup_{x\in D}\lambda_1^f(x) \leq 1 +
    dMe^{dM}\qquad \textnormal{and}\qquad \inf_{f\in\mathcal
      F}\inf_{x\in D}\lambda_d^f(x) \geq \big(1 +
    dMe^{dM}\big)^{-1}.\]
\end{theorem}

\subsubsection{Proof of Theorem \ref{thm:general}}
	
Given a class
$\mathcal F_n\subseteq\cV$ of velocity fields, 
define the set of
corresponding pullback distributions as
\[\mathcal P_n \coloneqq \{(T^f)^\sharp \rho :f\in \mathcal F_n \}.\]
Then, by definition, an ODE-MLE $\hat f$ as in \eqref{eq:MLEobjective} %
satisfies
\[ (T^{\hat f})^\sharp\rho \in \arg\max_{p\in \mathcal P_n} \sum_{i=1}^n
  \log p(Z_i),\] i.e., the pullback distribution $(T^{\hat f})^\sharp\rho$
constitutes %
an MLE over $\mathcal P_n$. Thus our strategy will be to
verify that Theorem \ref{sara10} can be suitably applied with
approximating sieve classes $\mathcal P_n$.

{\bf Step 1: Uniform bounds on pullback densities.} 
We %
prove that all densities in $\mathcal P_n$ are uniformly upper and lower bounded. First note that \eqref{eq:Kkappa} and Theorem
  \ref{thm:SingularValueBound} imply the existence of constants
  $0<C_1<C_2<\infty$ solely depending on $r$ in
  \eqref{eq:GeneralThmAssumption} such that for all $f\in \mathcal F_n$ and $x\in D$, the spectrum $\sigma(\nabla T^f(x))$ of the Jacobian matrix $\nabla T^f(x)\in \R^{d\times d}$ satisfies %
  \begin{equation*}
    \sigma(\nabla T^f(x))\in [C_1,C_2].%
  \end{equation*}
Using the change-of-variables formula
\begin{align*}
  (T^{f})^\sharp\rho(x) = \rho(T^f(x))\det\nabla(T^{f}(x)),
\end{align*}
and since $\kappa <\rho (x) <K$, we thus find that there exists $L=L(r,\kappa)>0$ such that
\begin{equation}\label{LB}
  \inf_{f\in\mathcal F_n} \inf _{x\in D}(T^{f})^\sharp\rho(x) \geq L.
\end{equation}
Similarly there exists $U=U(r,K)$ such that
\begin{equation}\label{eq:bounded}
  \sup_{f\in \mathcal F_n} \sup_{x\in D} (T^{f})^\sharp \rho(x) \le U.
\end{equation} 

In particular, for any %
  $f^*\in \mathcal F_n$, denoting $p^\ast= (T^{f^*})^{\sharp}\rho$,
  it holds that $p^\ast \ge L$ uniformly in $D$. %
  Hence the assumption on $p^\ast$ in Theorem \ref{sara10} is
  fulfilled with $c=L$.
	
{\bf Step 2: Bounding the covering number via Lipschitz properties.}
Fix
$f^*\in\cF_n$ and denote again
$p^*=(T^{f^*})^\sharp \rho$. Define
\[ \mathcal Q_n^*:=\Big\{ \sqrt {\frac{p+p^*}{2}} : p\in \mathcal
  P_n\Big\}= \Big\{ \sqrt {\frac{(T^f)^\sharp \rho +p^*}{2}}: f\in
  \mathcal F_n\Big\}. \] Our goal is to bound the bracketing covering
number $N_B(\cQ_n^*,L^2(p^*),\tau)$; see Definition \ref{def:bme}. 
To this end, we interpret $\cQ_n^*$ as the image of $\cF_n$ under two
maps $\Phi_1$, $\Phi_2$ via
\begin{equation*}
  f\quad\stackrel{\Phi_1}{\mapsto}\quad (T^f)^\sharp\rho\quad\stackrel{\Phi_2}{\mapsto}\quad
  \sqrt{\frac{(T^f)^\sharp \rho +p^*}{2}},
\end{equation*}
and we now show that both maps are Lipschitz continuous.

We start with $\Phi_1$. Recall that $\mathcal F_n$ is bounded in the
sense \eqref{eq:GeneralThmAssumption}, and $\rho:D\to\R$ is Lipschitz
continuous. Thus Lemma \ref{thm:FlowMapBound}, Theorem
\ref{thm:LInfDensityBound}, as well as the bounds on the singular
  values of $\nabla T_f$ from Theorem \ref{thm:SingularValueBound} 
imply that there are constants $C_3=C_3(r,d,|\rho|_{{\rm Lip}(D)},K)>0$ and $C_4=C_4(r,d,K)>0$
(cp.~\eqref{eq:Kkappa}, \eqref{eq:GeneralThmAssumption}, and
\eqref{eq:bounded}) such that for all $f, g\in \mathcal F_n$
\begin{equation*}
  \|(T^f)^\sharp \rho-(T^g)^\sharp \rho
  \|_{L^\infty(D)}\le C_3 \|T^f- T^g \|_{C^1(D)}\le C_3C_4 \|
  f-g \|_{C^1(\Omega)}.
\end{equation*}
That is,
  \begin{equation}\label{eq:C3C4}
    \Phi_1:=\begin{cases}
      \cF_n\subseteq C^1(\Omega)\to L^\infty(D)\\
      f\mapsto (T^f)^\sharp\rho
    \end{cases}\qquad
    \text{has Lipschitz constant }C_3C_4.
    \end{equation}

To treat $\Phi_2$, note that the uniform lower
bound \eqref{LB} also yields a lower bound for the corresponding
square root densities:
\[ \inf_{q\in \mathcal Q_n^*} \inf_{x\in D}q(x)=\inf_{f\in \mathcal F_n}
  \inf_{x\in D} \sqrt{\frac{(T^f)^\sharp\rho + (T^{f^*})^\sharp \rho
    }{2}(x)}\ge \sqrt{L}.\] Since $\sqrt \cdot$ is Lipschitz
continuous on the interval $[\sqrt L,\infty)$, it follows that for all
$f,g \in \mathcal F_n$ and some %
$C_5=C_5(L)=C_5(r,\kappa)$
\begin{equation*}
  \left \| \sqrt{\frac{(T^f)^\sharp\rho  + (T^{f^*})^\sharp \rho }{2}} - \sqrt{\frac{(T^g)^\sharp\rho  + (T^{f^*})^\sharp \rho }{2}} \right \|_{L^\infty(D)} \le C_5 \|(T^f)^\sharp\rho  -(T^g)^\sharp\rho  \|_{L^\infty(D)}.
\end{equation*}
That is,
  \begin{equation}\label{eq:C5}
    \Phi_2:=\begin{cases}
      \cP_n\subseteq L^\infty(D)\to L^\infty(D)\\
      p\mapsto \sqrt{\frac{p+p^*}{2}}
      \end{cases}\qquad\text{has Lipschitz constant }C_5.
  \end{equation}

Applying first Lemma \ref{lemma:InfToBracketing} %
  (noting that %
  $p^*(D)=1$) and then Lemma \ref{lemma:NLip} with \eqref{eq:C5} and
  \eqref{eq:C3C4}, we obtain for all $\tau>0$
\begin{align*}
  &N_B(\mathcal Q_n^*, L^2(p^*),\tau) \le 	N\Big(\mathcal Q_n^*,%
   L^\infty(D),\frac{\tau}{2}\Big)\\
   &\le N\Big( \mathcal P_n, %
   L^\infty(D), \frac{\tau}{4C_5}\Big) \le N\Big( \mathcal F_n, %
  C^1(\Omega), \frac{\tau}{8C_3C_4C_5}\Big).
\end{align*}

{\bf Step 3: Metric entropy integral bounds.} In order to be able to apply
Theorem \ref{sara10}, %
we need to verify that
the metric entropy bound assumption \eqref{entropy-req} of Theorem
\ref{thm:general} implies the corresponding condition
\eqref{entropy-cond} in Theorem \ref{sara10}.

Without loss of generality, we may assume that %
$C_3C_4C_5\ge 1$ (by choosing these constants larger than
$1$). Then we %
obtain the following estimate for the
bracketing entropy integral:
\begin{align}\label{ent-integral-estimates}
  R+ \int_{0}^R H_B^{1/2} \big(\mathcal Q_n^*, L^2(p^*),\tau \big)d\tau& \le R+ \int_{0}^R H\big(\mathcal F_n, C^1(\Omega),\frac\tau{2C_3C_4C_5}\big) d\tau \\
&\le  R+ 2C_3 C_4C_5\int_{0}^{R/C_3C_4C_5} H\big(\mathcal F_n, C^1(\Omega) ,\tau) d\tau \\
&\le 2C_3C_4C_5  \Psi(R).
\end{align}
Now let $C_6$
be a constant with the
same value as the constant $C$ from Theorem \ref{sara10}, and let us
define $\tilde \Psi (R):= 2C_3C_4C_5\Psi(R)$. Clearly,
$R\mapsto \tilde \Psi(R)/R^2$ is still a non-decreasing
function. Moreover, any $n\ge 1,\delta_n>0$ satisfying
\[ \sqrt n\delta_n\ge 2C_3C_4C_5C_6\Psi(\delta_n) \] will also
fulfill
\[ \sqrt n\delta_n\ge C_6 \tilde \Psi(\delta_n). \]

Finally, %
we may therefore
apply Theorem \ref{sara10} to those values of $n,\delta_n$ with
$\tilde \Psi$ as an upper bound, and we obtain that for any
$\delta\ge \delta_n$,
\[ P_0^n\Big( h\big((T^{\hat f})^\sharp\rho, p_0 \big) \ge C_6\big[
  h\big((T^{f^*})^\sharp\rho,p_0 \big) +\delta \big]\Big) \le C_6\exp
  \Big(-\frac{n\delta^2}{C_6}\Big).\] This completes the proof of (\ref{conc-ineq}).

  Finally, the bound (\ref{eq:expectationbound}) for the mean squared error follows from a standard integration argument (cf.~the proof of Lemma 2.2 in \cite{VDG01}). Let us use the shorthand $\hat h=h((T^{\hat f})^\sharp\rho, p_0)$ and $h=h((T^{f^*})^\sharp\rho, p_0)$. Then (\ref{conc-ineq}) implies that $P_0^n(\hat h^2 \ge 2C(h^2+\delta^2))
\le P_0^n(\hat h \ge C(h+\delta))\le C\exp(-\frac{n\delta^2}{C})$ for all $\delta\ge \delta_n$. Moreover, by assumption $\sqrt n \delta_n^2\ge C\Psi(\delta_n)\ge C\delta_n$, such that $\delta_n \ge C/\sqrt n$. %
Thus, we obtain that
  \begin{align*}
      \mathbb E_{P_0}^n[\hat h^2] = \int_0^\infty P_0^n(\hat h^2\ge t) dt \le 2C^2(h^2+\delta_n^2)  + \int_{t > 2C^2(h^2+\delta_n^2)} P_0^n(\hat h^2\ge t) dt.
  \end{align*} %
The second term is further bounded by
\begin{align*}
    \frac{1}{2C^2}\int_{\delta^2 > \delta_n^2} P_0^n(\hat h^2\ge 2C^2(h^2+\delta^2)) d\delta^2&\le  \frac{1}{2C^2}\int_{\delta^2 > \delta_n^2} C\exp\big(-\frac{n\delta^2}{C}\big) d\delta^2\\
    &\le \frac{1}{2n}\exp\big(-\frac{n\delta^2_n}{C}\big)\le \frac{1}{2n e}.
\end{align*}
\hfill \qedsymbol

\smallskip

\section{Results for $C^k$ velocity fields}\label{Ck-theory}

We now apply the general theory from the preceding section to a
canonical nonparametric density estimation setting, where the
data-generating density $p_0$ is assumed to belong to a class of
$k$-times differentiable functions. Again, let $D=[0,1]^d$ denote the
unit cube.
Then, given some integer $k\ge 1$ and constants $0<L_1<L_2<\infty$,
let us introduce the following class of upper and lower bounded $C^k$
probability densities on $D$:
\begin{equation}\label{Ck-class}
  \mathcal M(k,L_1,L_2) = \Big\{ p \in C^k(D): \inf_{x\in D} p(x)\ge L_2,~~ \| p\|_{C^k(D)}\le  L_1, ~ \int_D p(x)dx =1 \Big\}.
\end{equation}

For Theorem \ref{thm:general} to yield `fast' rates of convergence, it
is essential to choose the variational class $\mathcal F$ of velocity
fields appropriately. A canonical possibility is to choose the class
`as small as possible' such that there exists an element
$f^*\in \mathcal F$ with $(T^{f^*})^\sharp \rho = p_0$. This leads to the
following natural question: Given some density
$p_0\in \mathcal M(k,L_1,L_2)$, what is the regularity one %
can expect a velocity field coupling $\rho$ with $p_0$ to have? Our
first result of this section, Theorem \ref{Thm:existencCkField},
proves that there exists a velocity field which lies in
$C^k\cap \mathcal V$ and exactly couples $\rho$ with $p_0$. In other
words, there exists a velocity field which is \textit{at least as
  regular} as the densities which it
couples. %

Our result follows from proving the $C^k$-regularity of one specific
velocity field, which is constructed using the Knothe--Rosenblatt (KR) transport \cite{S15,V09}. 
Roughly speaking, the KR transport map is the triangular and monotone map $T:D\to D$ which couples
$\rho$ and $p_0$. By triangular, we mean that the $l$-th component
function only depends on the first $l$ variables $(x_1,\dots, x_l)$,
\[ T(x)=\begin{bmatrix*}[l]
    T_1(x_1)\\
    T_2(x_1,x_2)\\
    \vdots\\
    T_d(x_1,\dots, x_d)
  \end{bmatrix*}, \qquad x\in D,\] and by monotone we mean that
each component function $T_l$ is strictly increasing with respect to its last 
argument $x_l$. It is well known that the KR map is unique up to
coordinate ordering, and that $T$ actually possesses an explicit
construction in terms of the CDFs of the marginal conditional
densities of $p_0$ and $\rho$. We refer the reader to \cite[Chap.\
2.3]{S15} or \cite{ZM22a} for this construction and for standard
properties of KR maps.

\subsubsection{$C^k$-regularity of a `straight-line' velocity field}

Given the KR map $T$, we now define our candidate velocity field
which we will later prove to satisfy $C^k$-regularity. First, let
$G:D\times [0,1]\to D\times[0,1]$ be the `straight-line interpolation'
(giving rise to an analogue of the displacement interpolation between $\rho$ and
$p_0$ \cite{mccann1997convexity}) between the identity map and $T$,
\begin{equation}\label{eq: G}
  G_t(x)  := tT(x) + (1-t)x.
\end{equation}
In \cite{ren23a}, it is established that %
$G_t:D\to D$
is
invertible
for each $t\in [0,1]$. Then, let
\[F:D\times[0,1]\to D,\qquad F(x,t)= G^{-1}_t(x),\] based on which we
define the following velocity field
\begin{equation}\label{eq:fexpl}
  f^\Delta_{p_0}(y,s) = T(F(y,s)) - F(y,s),\qquad\forall (y,s)\in D\times[0,1].
\end{equation}
Then, the flow induced by $f^\Delta_{p_0}: D\times[0,1]\rightarrow D$
has the straight-line trajectories
$X^{f^\Delta_{p_0}}(x,t) = tT(x) + (1-t)x$, and indeed pushes $p_0$ to
$\rho$; see \cite{ren23a} for details.

In order to state the next result, we require the following mild
assumption on the reference density.
\begin{ass}\label{ass:rho}
  Let $\rho\in C^k(D)$ be uniformly
  lower bounded by $\kappa>0$.
  Moreover, suppose that $\rho$
  factorizes into $k$-smooth marginal distributions; that is,
  there exist univariate densities $\rho_l \in C^k([0,1])$
  such that $\rho(x)=\prod_{l=1}^d \rho_l(x_l)$. 

\end{ass}
This assumption allows for many natural choices of reference
distributions on the unit cube, such as the uniform distribution, or
truncated Gaussian distributions with diagonal covariance matrix. We
also note that the assumption of $\rho$ being a product distribution is made for
convenience, and can be relaxed at the expense of further
technicalities; see Remark \ref{rem:rho} for further details.

\begin{theorem}\label{Thm:existencCkField}
  Let
  $k\ge 1$, and let $\rho$ be some
  reference density satisfying Assumption
  \ref{ass:rho}. Moreover, suppose that
  $p_0\in\mathcal{M}(k, L_1, L_2)$. Let $T: [0,1]^d\rightarrow[0,1]^d$
  and $f^\Delta_{p_0}$ respectively denote the KR map and the
  straight-line velocity field between $p_0$ and $\rho$ (constructed
  above). Then:

  \begin{enumerate}
    \item\label{item:f}
  It holds that $f^\Delta_{p_0}\in C^k(\Omega)$ with
  $\|f^\Delta_{p_0}\|_{C^k(\Omega)} \leq C$, for some $C>0$ that
  depends only on $\rho, k, d, L_1, L_2$.

  \item\label{item:g}
    For
    $g_{p_0}^\Delta: \Omega\to \R^d$ defined as
  \begin{equation} [g^\Delta_{p_0}(x,s)]_j :=
    \frac{(f^\Delta_{p_0}(x,s))_j}{x_j(1-x_j)},\qquad j= 1,\dots, d
  \end{equation}
  it holds that $g^\Delta_{p_0} \in C^k(\Omega)$, and there
  exists another constant $\tilde{C}=\tilde C( d, L_1, L_2)$, such that
  $\|g^\Delta_{p_0}\|_{C^k(\Omega)} \leq \tilde{C}$. In particular it
  holds that $f^\Delta_{p_0}\in\mathcal{V}$ (cf.~\eqref{eqn:VelocityField}), i.e., the
  normal component $f_{p_0}^\Delta(x,t) \cdot \nu_x$ vanishes at every
  point $(x,t)\in \partial D\times [0,1]$.
  \end{enumerate}  
\end{theorem}

The above result shows that for $C^k$-regular target densities $p_0$,
the velocity field $f^{\Delta}_{p_0}$ %
inherits
$C^k$ regularity. Crucially, %
Part \ref{item:g} of the theorem also shows that 
$f^{\Delta}_{p_0}$ is 
an `admissible' velocity field whose normal component vanishes on the
`tubular' boundary $\partial D \times [0,1]$.  The proof uses certain
anisotropic regularity results for KR maps developed in 
\cite{TransportMiniMax}, along with technical results showing that this
anisotropic regularity is preserved under composition and inversion of
maps. In order to deduce the boundary properties in Part \ref{item:g}, we then use a so-called Hardy inequality. For the
full proof, we refer to Appendix \ref{sec:Ck-proofs}.

\subsubsection{Convergence theorem for estimators over $C^k$-classes}

We are now ready to state the main theorem of this section, which
gives a convergence rate for ODE-MLEs whenever
$p_0%
\in \mathcal M(k,L_1,L_2)$.
, and %
For $r>0$, define
\begin{equation}\label{F-Ck}
  \mathcal F(r):= \big\{ f\in C^k(\Omega, \R^d): \|f\|_{C^k}\le r \big\} \cap \mathcal V.
\end{equation}

\begin{theorem}\label{thm-holder}
  Let $k>d/2+3/2$, $0<\gamma < k-d/2-3/2$,
  $0<L_1\le L_2<\infty$,
  and suppose $\rho$ satisfies Assumption
  \ref{ass:rho}. %
  Then, there exist constants $r=r(k,L_1,L_2)>0$
  and $C=C(k,L_1,L_2)>0$ such that for any
  $p_0 \in \mathcal M(k,L_1,L_2)$, the velocity field $\hat f$
  maximizing the objective (\ref{eq:MLEobjective}) over
  $\mathcal F (r)$ satisfies %
  \[ \E_{P_0}^n\big[ h^2((T^{\hat f})^\sharp\rho, p_0) \big]\le Cn^{-\eta
    },\qquad \text{with}~\eta=\frac{2(k-1-\gamma)}{2(k-1-\gamma)+d+1
    }>0.\]
\end{theorem}

The proof of Theorem \ref{thm-holder} can be found in
\Cref{sec:Ck-proofs}. In essence, the result follows from an
application of the general Theorem \ref{thm:general} with
$\mathcal F=\mathcal F_n = \mathcal F(r)$ and with `approximating' velocity field 
$f^*=f^*_n=f_{p_0}^\Delta$ given by Theorem
\ref{Thm:existencCkField}, using also classical metric entropy
estimates for $C^k$ classes. Note that the approximation error $h ((T^{ f^*_n})^\sharp\rho, p_0  )$ from \eqref{eq:expectationbound} then vanishes, such that there is no need for $\mathcal F$ to depend on $n$.

\begin{remark}[On the parameterization of $\mathcal F(r)$]\label{remark:cutoff}\normalfont
  The choice of $\mathcal F=\mathcal F(r)$ underlying Theorem
  \ref{thm-holder} is informed by the regularity that we can expect a
  velocity between two $C^k$ probability densities to have. In
  practical implementations, of course, one cannot employ the full class $\mathcal F(r)$ and must resort to a subclass of $\mathcal F(r)$ described by finitely many parameters, whose size would typically increase as $n$ grows. One
  example are
  neural network-based parameterizations, which %
  will be
  discussed in
  \Cref{sec:NN-theory}. %
  Alternatively, one could use classical
  approximating classes such as polynomials, wavelets, or splines
  \cite{T08,DV93}. 

  Typically, those approximating classes will \textit{not} satisfy
  that the normal component of $f(x,t)$ vanishes at the boundary. In
  order to enforce this property, one can employ a boundary cut-off
  construction where one %
  first chooses an approximating class (e.g., polynomials, wavelets,
  splines, neural networks) and
  then multiplies the field's $j$-th component
  by the `cut-off' function
  $x_j(1-x_j)$ for all $j\in \{1,%
  \dots,
  d\}$. The fact that such a
  construction still yields a sufficiently rich class $\mathcal F(r)$ is
  implied by the regularity result in  %
  Theorem \ref{Thm:existencCkField}, part \ref{item:g}: Indeed, the theorem implies
  that the triangular velocity field $f_{p_0}^\Delta$ may be expressed
  as the product of some $C^k$-velocity field
  $\tilde f\in C^k(\Omega, \R^d)$ with the above component-wise
  cutoff:
  \[ [f_{p_0}^\Delta]_j = \tilde f_j \cdot x_j(1-x_j),\qquad \forall
    j\in \{1,\dots, d\}. \] This is precisely the construction that
  will be used to construct the neural network-based `ansatz space' in
  \Cref{sec:NN-theory} below.
\end{remark}

\begin{remark}[On Assumption \ref{ass:rho}]\label{rem:rho}\normalfont
  While our general Theorem \ref{thm:general} only required $\rho$ to
  be Lipschitz continuous (and lower bounded), the present results
  hold under slightly more stringent requirements on $\rho$. The $C^k$
  regularity is crucial for guaranteeing the existence of a $C^k$
  transport map between $\rho$ and $p_0$. In contrast, the assumption
  that $\rho$ factorizes into its marginal distributions can be
  relaxed at the expense of further technicalities. %
  An inspection of the proofs reveals that the factorization property
  is only needed in the proof of Theorem \ref{Thm:existencCkField}
  because we cite a regularity result from \cite{TransportMiniMax} for
  Knothe--Rosenblatt maps which uses this assumption. The latter
  result, however, can be generalized to general $C^k$-smooth
  reference densities.

\end{remark}

\section{Neural ODEs: neural network parameterization of velocity fields}\label{sec:NN-theory}
In this section, we study the %
case where
the underlying velocity
field is parameterized by a neural network class, i.e., neural ODEs
\cite{chen2018neural, grathwohl2018ffjord}. Like in Section \ref{Ck-theory}, our
strategy will be to apply Theorem \ref{thm:general}, this time to
classes of neural networks. To do so, we will separately study the
metric entropy rates and the ``best approximation'' properties of the
neural network classes defined below.

We now introduce our notation for neural network classes with ReLU$^m$
activation function. Let $\eta_1(x) = \max\{x, 0\}$ being the ReLU
activation function, and $\eta_m(x) = \max\{x, 0\}^m$ be the ReLU$^m$
activation function.

\begin{definition}\label{def:ReLUm}
  Let $m\ge 1$ and fix $d_1,d_2\ge 1$. Then, the class of ReLU$^m$
  networks mapping from $[0,1]^{d_1}$ to $\mathbb{R}^{d_2}$, with
  height $L$, width $W$, sparsity constraint $S$, and norm constraint
  $B$, is defined by
  \begin{align*}
    &\Phi^{d_1,d_2}(L, W, S, B) = \Big\{ \big(W^{(L)}\eta_m(\cdot) + b^{(L)}\big)\circ\cdots\circ\big(W^{(1)}\eta_m(\cdot) + b^{(1)}\big): \\
    &W^{(L)} \in \mathbb{R}^{1\times W}, b^{(L)}\in\mathbb{R}^{d_2}, W^{(1)}\in\mathbb{R}^{W\times d_1},b^{(1)}\in\mathbb{R}^W, W^{(l)}\in\mathbb{R}^{W\times W},\\ &b^{(l)}\in\mathbb{R}^W (1<l<L),	\sum_{l=1}^L\big(\|W^{(l)}\|_0 + \|b^{(l)}\|_0\big) \leq S, \max_{1\leq l\leq L} \big(\|W^{(l)}\|_{\infty,\infty}\lor \|b^{(l)}\|_\infty\big)\leq B \Big\}.\\    
  \end{align*}
  We refer to an element of $\Phi^{d_1, d_2}(L, W, S, B)$ as a
  \emph{ReLU$^m$ network}. For any index $1\leq l\leq L$, we write
  $F_l$ for the network composed of the first $l$-layers, that is,
$$F_l = \big(W^{(l)}_F\eta_m(\cdot) + b^{(l)}_F\big)\circ\cdots\circ\big(W^{(1)}_F\eta_m(\cdot) + b^{(1)}_F\big).  $$
We refer to such networks as $l-$ReLU$^m$ networks. We use
$\Phi^{d_1,d_2}_l(L, W, S, B)$ to denote all such $l$-layer networks.
\end{definition}

Since we will need the $C^1(\Omega)$ metric entropy of the above network
classes, we shall also need the gradient space
$\nabla \Phi^{d_1,d_2}(L,W,S,B)$. Note that for any
$1 \leq l \leq L-1$, any $l$-ReLU$^2$ network
$F_l \in \Phi^{d_1,d_2}_l(L, W, S, B)$ is a map from
$\mathbb{R}^{d_1}$ to $\mathbb{R}^W$. For any $1\leq j\leq W$, we use
$F_{l,j}$ to denote the $j-$th component. Then, we may write $F_l(x)$
and its Jacobian $\nabla F_l(x)$ as follows:
$$F_l(x) = [F_{l,1}(x), F_{l,2}(x),\dots ,F_{l, W}(x)]^T,$$
\[\nabla F_l(x) = \begin{bmatrix} 
    \frac{\partial}{\partial x_1}F_{l,1}(x) & \frac{\partial}{\partial x_2}F_{l,1}(x) & \dots \frac{\partial}{\partial x_{d_1}}F_{l,1}(x)\\
    \vdots & \ddots & \\
    \frac{\partial}{\partial x_1}F_{l,W}(x) & \ddots &
    \frac{\partial}{\partial x_{d_1}}F_{l,W}(x)
  \end{bmatrix}
\]
When $l = L$, $F_L$ maps $\mathbb{R}^{d_1}$ to $\mathbb{R}^{d_2}$ and
the Jacobian can be written as a $d_2\times d_1$ matrix.

\subsection{Metric entropy rates}\label{sec:NN-metric-entropy}

In order to apply %
Theorem \ref{thm:general}, we need to
control the $C^1([0,1]^{d_1})$-metric entropy of these parametric
classes.
We now present our results on entropy rates of the NN class
$\Phi^{d_1,d_2}(L, W, S, B)$ in $C^1([0,1]^{d_1})$ norm. Our results
are similar to those in \cite{MLPDEStatisticalRate} except that we use
ReLU$^2$ networks, in place of the ReLU$^3$ networks considered
in \cite{MLPDEStatisticalRate}.

The following theorem gives an upper bound for the metric entropy rate
of $\Phi^{d_1,1}(L,W,S,B)$, i.e., the case where $d_2=1$. The
subsequent corollary will then deal with the case of multi-dimensional
outputs.

\begin{theorem}\label{thm:C1CoveringNN} %
  Let $d_1\in\N$.
  Consider the ReLU$^2$ network space $\Phi^{d_1, 1}(L, W, S, B)$ with $L = \mathcal{O}(1), W =
  \mathcal{O}(N), S = \mathcal{O}(N)$ and $B = \mathcal{O}(N)$. Then
	$$H(\Phi^{d_1, 1}(L, W, S, B),C^1([0,1]^{d_1}), \tau) = \mathcal{O}\big(N\log(\tau^{-1}) + N\log N\big).$$
      \end{theorem}
      \begin{proof}
        The proof of this theorem is based on translating covering
        numbers of the NN parameter space (in $l^\infty$ norm) into
        covering numbers of the NN function space (in $C^1$ norm). For this
        purpose, we shall need Lipschitz-type estimates from the NN
        parameter space into the NN function space and its gradient space, which are respectively given by Lemma
        \ref{RelationBetweenCoveringNumberofDNNandParameterSpace} and
        Lemma \ref{lemma:CoveringRelationParameterAndGradient}.

        We first fix a sparsity pattern (i.e., the locations of the
        non-zero entries are fixed) and let $k = L$ in Lemma
        \ref{RelationBetweenCoveringNumberofDNNandParameterSpace} and
        Lemma
        \ref{lemma:CoveringRelationParameterAndGradient}. Following
        the arguments in the proof of Lemma 3 in \cite{Suzuki19}, we
        get the following upper bound for the covering number with
        respect to $C^1([0,1]^{d_1})$ norm:
	$$\left(\frac{\tau}{\max\{N_{L}W^{2^{L-1}-1}(B\lor d_1)^{2^L+1}, A_LW^{2^{L-1}-1}(B\lor d_1)^{2^L}\}}\right)^{-S},$$ where $A_L$, $N_L$ are the constants
        from Lemmata \ref{RelationBetweenCoveringNumberofDNNandParameterSpace}
        and \ref{lemma:CoveringRelationParameterAndGradient}, which only depend on $L$. 
      Note that the number of possible sparsity patterns 
      is upper
        bounded by ${(W+1)^L\choose S}\leq (W+1)^{LS}$ (see
        \cite{Suzuki19, Schmidt_Hieber_2020, MLPDEStatisticalRate}). Plugging in the
        magnitudes for the network parameters, we get the following
        metric entropy bound:
        \begin{align*}
          &H(\Phi^{d_1, 1}(L, W, S, B),C^1([0,1]^{d_1}), \tau) = \log N(\Phi^{d_1, 1}(L, W, S, B),%
            C^1[0,1]^{d_1}, \tau)\\   
          &\leq \log\left[(W+1)^{LS}\left(\frac{\tau}{\max\{N_{L}W^{2^{L-1}-1}(B\lor d_1)^{2^L+1}, A_LW^{2^{L-1}-1}(B\lor d_1)^{2^L}\}}\right)^{-S} \right]\\
          &\lesssim\max\Big\{ S\log\left[\tau^{-1}(W+1)^LN_LW^{2^{L-1}-1}(B\lor d_1)^{2^L+1}\right],\\
          &~~~~~~~~~~~~~~~~~~~~~~~~~~~~S\log\left[\tau^{-1}(W+1)^LA_LW^{2^{L-1}-1}(B\lor d_1)^{2^L}\right]\Big\}\\
          &\lesssim S\left[\log(\tau^{-1}) + 2^L\log(W(B\lor d_1))\right] = \mathcal{O}\left(N\log(\tau^{-1}) + N\log N\right).
        \end{align*}
      \end{proof}
      For the purpose of modeling velocity fields as neural networks,
      we need to consider the above neural network classes with
      $d_1=d+1$ and $d_2=d$, i.e., as mappings from
      $\Omega=[0,1]^{d+1}$ to $\mathbb{R}^d$; this entropy rate is
      obtained by a tensorizing argument.

\begin{corollary}\label{NNentropy-multidim}
  Let $d\ge 1$ be fixed and let $N\ge d$ be sufficiently
  large. Consider the ReLU$^2$ network class
  $\Phi^{d+1, d}(L,W, S, B)$ with
  $L = \mathcal{O}(1)$, $W = \mathcal{O}(N)$, $S = \mathcal{O}(N)$, and $B =
  \mathcal{O}(N)$. Then, the metric entropy satisfies
$$H(\Phi^{d+1, d}(L, W, S, B),C^1(
  \Omega), \tau) = \mathcal{O}\left(N\log(\tau^{-1}) + N\log
  N\right).$$
\end{corollary}
\begin{proof}
  Let $\phi = [\phi_1,\dots ,\phi_{d}]^T\in \Phi^{d+1, d}(L, W, S, B)$.
  Then for each $j$ it holds that
  $\phi_j \in \Phi^{d+1,1}(L, W, S, B)$ with
  $L = \mathcal{O}(1)$ and $ W, S, B = \mathcal{O}(N)$.

  For $j = 1,\dots ,d$, let $\{\psi_j^{m}\}_{m=1}^{M_j}$ be a
  $\tau$-covering of the $j$-th coordinate. We now construct a
  covering set of $\Phi^{d+1,d}(L, W, S, B)$ by taking the product set
  $\Psi = \{\psi_1^{m}\}_{m=1}^{M_1} \times \cdots\times
  \{\psi_d^{m}\}_{m=1}^{M_d} $ To show that Cartesian product is
  indeed a covering set, note that for %
  any member
  $\phi = [\phi_1,\dots,\phi_d]^T\in \Phi^{d+1,d}(L, W, S, B)$, we can
  find $\psi = [\psi_1^{m_1},\dots ,\psi_d^{m_d}]^T$ such that
  $\|\phi_j - \psi_j^{m_j}\|_{C^1} \leq \tau$, where
  $1\leq m_j\leq M_j$. It is then not hard to verify
  $\|\phi -\psi\|_{C^1} \leq \tau$.
 
  Assume $M_j \leq \tilde{M}$ for $1\leq j \leq d$, then the covering
  number satisfies $|\Psi| \leq \tilde{M}^d$ and the metric entropy is
  upper bounded by $d\log\tilde{M}$. From Theorem
  \ref{thm:C1CoveringNN}, $\tilde{M}$ is upper bounded as
  $\mathcal{O}(N\log(\tau^{-1}) + N\log N)$ and since we take $d$ to
  be a fixed constant, the metric entropy for
  $\Phi^{d+1,d}(L, W, S, B)$ is the same asymptotically.
\end{proof}

\subsection{Approximation theory}\label{sec:approximation}
The goal of this section is to show that functions $f\in C^k(\Omega)$
can be efficiently approximated by neural networks of a certain
architecture.  Recall from the general Theorem \ref{thm:general} that
we not only need our approximating NN class to be able to approximate
the target function $f:\Omega\to\R^d$ in the $C^1(\Omega)$-norm, but
also require its (spatial) gradient to be Lipschitz continuous.

Approximation results for $C^k$ functions on compact domains with
neural networks are by now standard, e.g., \cite{MR1819645} or the
more recent works \cite{NNApproximation1,NNApproximation4}, however
the specific statement we require appears not to be available in the
literature, which is why we provide a full proof in Appendix
\ref{app:NNapproximation}. The argument leverages a widely recognized
technique, first introduced in \cite{MHASKAR1992350,MR1230251}, based
on spline approximation.
 
The results that are directly related to our setting are the following
theorem and corollary. Their proofs can be found in Appendix
\ref{app:NNapproximation}.
	
	\begin{theorem}\label{thm:NNapproximation}
          Let $k$, $d_1$, $m\in\N$ and $k+1\le m$.  Then there exists
          $C=C(d_1,k,m)$ such that for all $f\in C^k([0,1]^{d_1},\R)$
          and all $N\in\N$ there exists a ReLU$^{m-1}$ neural network
          $\tilde f\in \Phi^{d_1, 1}(L,W,S,B)$ with
          \begin{equation}\label{eq:WLSBbound}
            L\le C,\qquad W\le N,\qquad S\le N,\qquad B\le C\norm[{C([0,1]^{d_1})}]{f}+N^{1/d_1}
          \end{equation}
          such that $\tilde f\in C^{m-2}([0,1]^{d_1},\R)$
          and %
          \begin{equation}\label{eq:err_relum}
            \norm[{W^{r,\infty}([0,1]^{d_1})}]{f-\tilde f}\le C  N^{-\frac{k-r}{d_1}}
            \snorm[{C^k([0,1]^{d_1})}]{f}\qquad\forall r\in \{0,\dots,k\}.
          \end{equation}
	\end{theorem}

	The next corollary shows that the assumption $m>k$ in Theorem
        \ref{thm:NNapproximation} can be dropped. We emphasize,
        however, that a ReLU$^{m-1}$ network always belongs to
        $W^{m-1,\infty}$ but it generally does not belong to
        $W^{m,\infty}$. Consequently, the network approximation
        $\tilde f\in W^{k,\infty}$, where $k\ge m$ is permitted,
        constructed in the following corollary is rather specific.
        Moreover, we state the result in the more general case of
        approximating a function
        $f=(f_j)_{j=1}^{d_2}:[0,1]^{d_1}\to \R^{d_2}$ for some
        $d_2\in\N$, which is how we will use it in the following.
	
	\begin{corollary}\label{cor:NNapproximation}
          Let $k$, $d_1$, $d_2$, $m\in\mathbb{N}$, and $m\ge 3$. Then
          there exists $C=C(d_1, d_2, k,m)$ such that for all
          $f\in C^k([0,1]^{d_1},\R^{d_2})$ and all $N\in\N$ there
          exists a ReLU$^{m-1}$ neural network
          $\tilde f\in \Phi^{d_1, d_2}(L,W,S,B)$ %
          with
          \begin{equation}\label{eq:WLSBbound2}
            L\le C,\qquad W\le N,\qquad S\le N,\qquad B\le C\norm[{C([0,1]^{d_1},\R^{d_2})}]{f}+N^{1/d_1}
          \end{equation}
          such that $\tilde f\in C^{m-2}([0,1]^{d_1},\R^{d_2})$ and
          for all $j\in\{1,\dots,d_2\}$
          \begin{equation}\label{eq:err_relum2}
            \norm[{W^{r,\infty}([0,1]^{d_1})}]{f_j-\tilde f_j}\le CN^{-\frac{k-r}{d_1}}
            \snorm[{C^k([0,1]^{d_1})}]{f_j}\qquad\forall r\in \{0,\dots,k\}.
          \end{equation}
	\end{corollary}

        \subsection{Statistical convergence rates for neural ODEs}
        \subsubsection{Ansatz space}\label{subsubsec:ansatz}

        As elaborated in Section \ref{sec:general-results}, we need to
        ensure that the velocity fields in $\mathcal F$ satisfy
        certain boundary conditions in order %
        for the pullback distributions $(T^f)^\sharp\rho$,
        $f\in\mathcal F$, %
        to be supported on the same domain $D$. Lemma
        \ref{lemma:Diffeomorphism}, Theorem \ref{Thm:existencCkField}
        together with Remark \ref{remark:cutoff} suggest that a
        suitable ansatz space can be formed by multiplying the
        preceding neural network classes by `component-wise' cutoff
        functions.
	
\begin{definition}
  Let $\chi_d(x_1,\dots x_d): D\rightarrow D$ be given
  by $$\chi_d(x_1,\dots x_d) = [x_1(1-x_1),\dots ,x_d(1-x_d)]^T.$$ Let
  $\otimes$ be the coordinate-wise multiplication of two vectors (of
  the same dimension).  Then for any velocity field
  $f: \Omega = [0,1]^d\times[0,1]\rightarrow \mathbb{R}^d$,
  $f\otimes\chi_d$ yields a vector field on $D$ with vanishing normal
  components at the boundary.  Similarly, we let $\odiv$ denote
  coordinate-wise division of two vectors.
\end{definition}

	\begin{definition}
          We let
          \begin{align*}
            \Phi_\text{ansatz}^{d+1, d}(L, W, S, B)  \coloneqq \Big\{ f^{\text{NN}}(x_1,\ldots, x_d,t)\otimes\chi_d(x_1,\ldots, x_d),~f^\text{NN}\in\Phi^{d+1,d}(L, W, S, B)\Big\},
          \end{align*}
          where $\Phi^{d+1,d}(L, W, S, B)$ is the class of ReLU$^2$
          networks defined in \eqref{def:ReLUm} and $L, W, S, B$ are
          the respective network parameters. For $r\ge 0$, we further
          define the following bounded sparse neural network classes
          \begin{equation}\label{eq:FNN}
            \mathcal{F}_{\text{NN}}(L,W, S, B, r) =
            \Phi^{d+1,d}_\text{ansatz}(L, W, S, B)\cap \{f\in W^{2,
              \infty}(\Omega):\|f\|_{W^{2,\infty}(\Omega)}\leq r\}.
          \end{equation}
	\end{definition}

        \subsubsection{Main statistical convergence result}
        Finally, we obtain the following nonparametric convergence
        rate for neural ODEs, by combining the preceding results about
        approximation and statistical complexity.

        \begin{theorem}\label{thm:NNStatsRate}
          Fix an integer $k\ge 1$ and constants $0<L_1\le L_2<\infty$,
          and suppose $\rho$ is a reference density satisfying
          Assumption \ref{ass:rho}. Then there exist parameter choices
          $L = \mathcal{O}(1)$,
          $W = \mathcal{O}( n^{\frac{d+1}{d+1 + 2(k-1)}})$,
          $S = \mathcal{O}( n^{\frac{d+1}{d+1 + 2(k-1)}})$,
          $B = \mathcal{O}( n^{\frac{d+1}{d+1 + 2(k-1)}})$, and
          $r = \mathcal{O}(1)$ such that for all
          $p_0\in \mathcal M(k,L_1,L_2)$, the neural ODE estimator
          $\hat f$ given by \eqref{eq:MLEobjective} over the class of
          velocity fields $\mathcal{F}_{\text{NN}}(L,W, S, B, r)$
          satisfies the convergence rate
$$\E_{P_0}^n\big[h^2((T^{\hat{f}^\Delta})^\sharp\rho, p_0)\big] \lesssim n^{-\frac{2(k-1)}{d+1+ 2(k-1)}}\log n.$$
\end{theorem}
\begin{proof}
  Our proof strategy will be to apply our general Theorem
  \ref{thm:general} to the neural network classes of velocity fields
  $\Phi^{d+1,d}_{\text{ansatz}} (L,W,S,B)$. To this end, we bound the
  approximation error (Step 1) and the metric entropy rates (Step 2)
  separately.

  \textbf{Step 1: Approximation error.}  Suppose that
  $p_0\in \mathcal M(k,L_1,L_2)$. By Theorem
  \ref{Thm:existencCkField}, there exists a velocity field
  $f^\Delta\in C^k(\Omega)\cap\mathcal{V}$ such that
  $(T^{f^\Delta})^\sharp\rho = p_0$ and such that for any $i\in [d]$,
  the $i$-th component $f^\Delta_i$ vanishes `linearly' at the
  boundaries, i.e., $\frac{f^\Delta_i}{x_i(1-x_i)}\in C^k(\Omega)$.

  Let us now define the velocity field
  \[ f^*(x_1,\ldots,x_d) = f^\Delta\odiv \chi_d =
    \Big(\frac{f_1^\Delta(x_1)}{x_1(1-x_1)},\dots,\frac{f_d^\Delta(x_1,\dots,x_d)}{x_d(1-x_d)}\Big)^T. \]
  Theorem \ref{Thm:existencCkField} moreover implies that for any
  $k,L_1,L_2$ there exists some constant $\tilde C$ such that
  \[ \sup_{p_0\in \mathcal M(k,L_1,L_2)} \| f^*\|_{C^k(\Omega, \R^d)}
    \le \tilde C. \] Note that $f^*$ does not necessarily satisfy the
  same boundary-vanishing properties as $f^\Delta$. By Corollary
  \ref{cor:NNapproximation} with $d_1=d+1$ and $d_2=d$, there exists a
  constant $C_{d,k}$ such that for all $N\ge 1$ and with
  \begin{equation}\label{Fnn-choices}
    L \leq C_{d,k}, W\leq N, S\leq N, B\leq C_{d,k}\|f\|_{C^k(\Omega)} + N^{1/(d+1)},
  \end{equation}
  there is a ReLU$^2$ neural network
  $\tilde{f}\in\Phi^{d+1,d}(L, W, S, B)$ with
  $\tilde{f}\in C^1(\Omega)$, satisfying the approximation properties
  \begin{equation}\label{eq:C1approx}
    \|\tilde{f} - f^*\|_{C^1(\Omega)} \leq C_{d,k}d^{\frac{k-1}{d+1}}N^{-\frac{k-1}{d+1}}\|f^*\|_{C^k(\Omega)},    
  \end{equation}
  and
  \begin{equation}\label{eq:W2InfApprox}
    \|\tilde{f} - f^*\|_{W^{2,\infty}(\Omega)} \leq C_{d,k}d^{\frac{k-2}{d+1}}N^{-\frac{k-2}{d+1}}\|f^*\|_{C^k(\Omega)}.    
  \end{equation}
  Later in the proof, we will make a choice of $N$ which balances the
  approximation error analysed here with the metric entropy term
  analysed in Step 2.

  Defining $\hat{f}^\Delta = \tilde{f}\otimes\chi_d$, it then follows
  from standard multiplication inequalities that
  \[ \| \hat f^\Delta - f^\Delta \|_{C^1(\Omega)} = \|\big( \tilde f -
    f^*\big) \otimes \chi_d\|_{C^1(\Omega)}\lesssim \| \tilde f - f^*
    \|_{C^1(\Omega)} \| \chi_d\|_{C^1(\Omega)} \lesssim
    N^{-\frac{k-1}{d+1}}. \] Similarly, we see that
  $\|f^\Delta - \hat{f}^\Delta\|_{W^{2,\infty}(\Omega)} =
  \mathcal{O}(N^{-\frac{k-2}{d+1}})$. Thus, using the triangle
  inequality and the fact that $f^\Delta\in C^k(\Omega)$, it follows
  that for some $r>0$,
  \[ \sup_{p_0\in \mathcal
      M(k,L_1,L_2)}\|\hat{f}^\Delta\|_{C^1(\Omega)} +
    \|\hat{f}^\Delta\|_{W^{2,\infty}(\Omega)} \le r. \]

  In summary, we have now proved the existence of an approximating
  element
  \[\hat{f}^\Delta\in \mathcal{F}_{\text{NN}}(L, W, S, B, r) =
    \Phi^{d+1,d}_{\text{ansatz}}(L, W, S, B)\cap \{f\in
    W^{2,\infty}(\Omega): \|f\|_{W^{2,\infty}(\Omega)}\leq r\}\] which
  approximates $f^\Delta$ at rate
  $\|\hat{f}^\Delta - f^\Delta \|_{C^1(\Omega)} =
  \mathcal{O}(N^{-\frac{k-1}{d+1}})$.
  In particular, we may now also deduce an approximation for the
  corresponding pullback densities in Hellinger distance. Indeed,
  using the Lipschitz estimates from Lemma \ref{thm:FlowMapBound},
  Theorem \ref{thm:LInfDensityBound}, Theorem
  \ref{thm:SingularValueBound} and Lemma \ref{lemma:LipschitzH2Linf},
  we obtain that
$$h((T^{\hat{f}^\Delta})^\sharp\rho,  (T^{f^\Delta})^\sharp\rho) =  h((T^{\hat{f}^\Delta})^\sharp\rho,  p_0) = \mathcal{O}(N^{-\frac{k-1}{d+1}}).$$
	
\textbf{Step 2: Metric entropy bound.} Given $N\ge 1$, we now derive
the required upper bound for the square-root metric entropy for the
neural network class $\mathcal F_{NN}(L,W,S,B,r)$, again with the
choices from (\ref{Fnn-choices}). Later on, we will choose $N$ to be
of the same order as $B$, so let us assume now that $B\leq N$. Note
that for any $f,g\in \Phi^{d+1,d}(L,W,S,B)$ (such that
$f\otimes \chi_d,g\otimes \chi_d\in
\Phi_{\text{ansatz}}^{d+1,d}(L,W,S,B)$) it holds that
\[ \| f \otimes \chi_d - g \otimes \chi_d \|_{C^1(\Omega)} \lesssim \|
  f - g \|_{C^1(\Omega)} \| \chi_d\|_{C^1(\Omega)} \lesssim \|
  f-g\|_{C^1(\Omega)}, \] which implies that for some constant
$c\ge 1$ and any $\tau>0$,
\[ N(\Phi^{d+1,d}_{\text{ansatz}}(L,W,S,B), C^1(\Omega), \tau ) \le
  N(\Phi^{d+1,d}(L,W,S,B), C^1(\Omega), \tau/c ). \] Thus, using the
upper bound from Corollary \ref{NNentropy-multidim} regarding metric entropy of neural network classes, we
obtain using \eqref{eq:FNN} that
\begin{align*}
  I(R) &= R + \int_0^RH^{1/2}( \mathcal{F}_\text{NN}(L,W,S,B,r), C^1(\Omega), \tau )d\tau\\
       &\lesssim R + \int_0^RH^{1/2}( \Phi^{d+1,d}_{\text{ansatz}}(L,W,S,B), C^1(\Omega), \tau )d\tau\\
       &\le R + \int_0^RH^{1/2}(\Phi^{d+1,d}(L,W,S,B), C^1(\Omega), \tau/c )d\tau\\
       &\lesssim R + \int_0^R \sqrt{N(\log\tau^{-1} + \log(N))}d\tau\\
       &\lesssim R + \sqrt{N}\int_0^1\sqrt{\log\tau^{-1} + \log N}d\tau + \sqrt{N}\int_1^R\sqrt{\log\tau^{-1} + \log N}d\tau\\
       &\lesssim R + \sqrt{N}\int_0^1\sqrt{\log\tau^{-1}}d\tau + \sqrt{N}\int_0^1\sqrt{\log(N)}d\tau  \\
       &+\sqrt{N}\int_1^R\sqrt{\log\tau^{-1} + \log N}d\tau\\
       &\lesssim R + \sqrt{N}\frac{\pi}{2} + \sqrt{N\log N} + \sqrt{N\log N}(R-1)\lesssim \sqrt{N\log N}R =: \Psi(R)
\end{align*}
With this choice of upper bound $\Psi(R)$, it is clear that
$\Psi(R)/R^2$ is non-increasing in $R$. Then, we can re-write the
condition \ref{entropy-req} as
$\Psi(\delta_n) \lesssim \sqrt{n}\delta_n^2$, which is equivalent to
$$\delta_n \gtrsim \sqrt{\frac{N\log N}{n}}.$$
	
\textbf{Step 3: Balancing terms.}  In order to balance the
approximation error with the metric entropy term, we will choose $N$
such that
\[N^{-\frac{k-1}{d+1}} \simeq \sqrt{\frac{N\log N}{n}}.\] Up to the
$\sqrt {\log N}$ factor, this is achieved by choosing
$N \simeq n^{\frac{d+1}{d+1 + 2(k-1)}}$. Now, applying the general
Theorem \ref{thm:general} with this choice yields the convergence rate
$$\E_{P_0}^n[h^2((T^{\hat{f}^\Delta})^\sharp\rho, p_0)] \lesssim n^{-\frac{2(k-1)}{d+1+ 2(k-1)}}\log n.$$

\end{proof}
\begin{remark}[On the rate from Theorem
  \ref{thm:NNStatsRate}]\normalfont \label{rem:finalrates}
  The final rate obtained in Theorem~\ref{thm:NNStatsRate}, up to a
  logarithmic factor, equals the optimal minimax rate
  \[n^{-\frac{2(k-1)}{d+1 + 2(k-1)}}\] for nonparametric estimation of
  a $(k-1)$-smooth function or density on a $(d+1)$-dimensional
  domain, in $L^2$ or Hellinger loss.
  The presence of $d+1$ (in place of $d$) in our rate is due to the
  fact that we are considering time-dependent velocity fields: given
  any transport map, there are infinitely many velocity fields whose
  time-one flow map matches this transport, and maximum likelihood
  estimation does not impose restrictions on the intermediate ODE
  trajectories between $t=0$ and $t=1$. 
  Some recent work, e.g., \cite{finlay2020train,onken2021ot,ren23a},
  considers neural ODEs with \textit{regularized} trajectories. In
  such settings, one might be able to improve the $d+1$ term to $d$;
  see below for further discussion.

  One may also wonder why the smoothness index appearing in the final
  rate is $k-1$, rather than $k$. Indeed, this is due to the fact that
  for the given $k$-smooth reference and target densities from
  $\mathcal{M}(k, L_1, L_2)$, the velocity field whose time-one flow
  realizes the corresponding KR map also belongs to $C^k$. Considering
  the $C^1$ metric entropy then yields the index $k-1$. This
  sub-optimality could possibly be resolved by using additional
  information about the coupling velocity field: as observed in
  \cite{TransportMiniMax,ISPH21}, KR maps between $C^k$ densities
  actually possess \textit{anisotropic} regularity---specifically,
  higher regularity in their `diagonal' input variables. It can be
  shown that the corresponding velocity field also satisfies this
  property (see Appendix \ref{sec:Ck-proofs}). With this additional
  smoothness, one might be able to improve the convergence index from
  $k-1$ to $k$. However, even with this additional knowledge, it is
  unclear how to construct neural network classes with such
  anisotropic regularity, rendering this observation less relevant for
  practical settings; we have thus omitted a generalization to this
  setting.

  The second term appearing in the final rate, $\log n$, appears due
  to the metric entropy integral of the neural network class. This
  $\log n$ factor is commonly present in statistical theory for neural
  networks; see, e.g., \cite{Schmidt_Hieber_2020}, which studies
  nonparametric regression using ReLU networks,
  \cite{MLPDEStatisticalRate}, which studies the problem of learning
  PDE solution fields with neural networks, and more recently
  \cite{DiffusionModelRate}, which studies the statistical convergence
  of diffusion models.
\end{remark}

\section{Discussion and future work}
\label{sec:discussion}

We have developed the first statistical finite-sample guarantees for likelihood-based distribution learning with neural ODEs. Our results show that neural ODE models are efficient distribution estimators, under relatively mild assumptions. We obtained these results by first developing a broader framework for analyzing ODE-parameterized maximum likelihood density estimators. This framework is applicable to any class of velocity field, and characterizes the impact of the chosen class on statistical performance. We then specialized this theory to $C^k$ velocity fields and to specific spaces of velocity fields described by neural networks, obtaining concrete minimax rates.

Our work suggests many important avenues for further work.
First, our analysis exposes an interesting impact of the time-dependent construction intrinsic to neural ODEs, i.e., the fact that one seeks a velocity field $f$ that depends on both space ($x \in \mathbb{R}^d$) and time ($t \in [0,1]$). While this construction confers several advantages (e.g., invertibility of maps, computational tractability of maps and densities), as noted in Remark~\ref{rem:finalrates}, the additional degree of freedom $t$ raises the dimension-dependence of the minimax convergence rate to $d+1$, from the optimal value of $d$. Several regularization schemes \cite{finlay2020train,onken2021ot,ren23a} have recently been proposed to control this ``extra'' freedom by promoting smooth or even straight-line ODE trajectories, with good empirical success. These regularization methods take the form of penalty terms added to the log-likelihood training objective, and it is desirable to understand their impact on statistical rates. To that end, \cite{TransportMiniMax} develops convergence theory for \textit{penalized} nonparametric density estimation using transport maps, and it would be fruitful to integrate such results with the ODE framework developed in this paper.

Second, we note that our work only considers density estimation on the hypercube $[0,1]^d$. Indeed, some of our arguments---for example, the construction of a suitable neural network ansatz space for velocity fields in Section~\ref{subsubsec:ansatz}, satisfying the no-flow boundary condition; and the lower bounds for densities used in the proof of Theorem \ref{Thm:existencCkField}---rely crucially on this fact. In future work, however, it would be useful to extend the present statistical convergence analysis to more general bounded domains and to unbounded domains. The latter will require a more refined understanding of the tail properties of the associated ODE flow maps.

To our knowledge, the question of \textit{computational} guarantees for neural ODE training is quite open. It remains challenging to characterize the loss landscape and its interaction with optimization algorithms; here one must also assess the impact of ODE time discretization, and the potential impact of different ways of computing gradients in this setting, e.g., ``discretize-then-optimize'' versus ``optimize-then-discretize'' approaches that use continuous adjoints \cite{gholami2019anode}.

We also note that several recently proposed generative modeling methods, e.g., flow matching \cite{lipman2022flow}, rectified flow \cite{liu2022flow}, and stochastic interpolants \cite{albergo2023stochastic}, produce deterministic ODEs but depart from the maximum likelihood training approach considered in this paper. It would be interesting to elucidate the statistical performance of such methods as well.

\begin{acks}[Acknowledgments]
SW and YM acknowledge support from the US Air Force Office of Scientific Research (AFOSR) MURI, Analysis and
Synthesis of Rare Events, award number FA9550-20-1-0397. ZR and YM acknowledge support from the US Department of Energy (DOE), Office of Advanced Scientific Computing Research, under grants DE-SC0021226 and DE-SC0023187. ZR also acknowledges support from a US National Science Foundation Graduate Research Fellowship. %
\end{acks}

\bibliographystyle{abbrv} \bibliography{ref,statsbib}

\appendix
\newpage

\section{Remaining proofs for Section \ref{sec:general-results}}\label{sec2-proofs}
\subsection{Proofs from statements in the main text}
\begin{proof}[Proof of Lemma \ref{lemma:Diffeomorphism}] We divide the
  proof into two steps. To simplify notation we drop the index $f$
    and write $X=X^f$. Recall that this is a map from
    $D=[0,1]^d\times [0,1]\to\R^d$. Its components are denoted by
    $X=(X_1,\dots,X_d)$.

  {\bf Step 1: Trajectories remain in $D$.} By definition
    \begin{equation}\label{eq:Xxt}
      X(x,t) = x + \int_0^t f(X(x,s),s)ds,\qquad t\in [0,1].
    \end{equation}
    For an interior point $x\in (0,1)^d$, we show that
    $X(x,t)\in (0,1)^d$ for all $t\in (0,1)$, i.e.,  $0<X_j(x,t)<1$ for
    all $j=1,\dots,d$. By symmetry, it suffices to show
    $X_1(x,t)>0$.

  Consider a point
    $y=(0,y_2,\dots,y_d)\in \{0\}\times (0,1)^{d-1}\subseteq\partial
    D$. Then, the normal outer vector $\nu_y=(-1,0,\dots,0)\in\R^d$
    at $y$ is well-defined. Since $f\in\cV$, the definition of $\cV$
    in \eqref{eqn:VelocityField} yields $f(y)\cdot \nu_y=0$, and thus
    \begin{equation*}
      f_1(y,t)=0\qquad\forall y\in\{0\}\times (0,1)^{d-1}.
    \end{equation*}
    Moreover, $f\in C^1(\Omega)$ implies
  $\|\partial_{y_1}f\|_{C(D)}<\infty$.
  Hence, using the fundamental theorem of calculus
  \begin{equation*}
    |f_1(y,t)|\le y_1 \|\partial_{y_1}f\|_{C(D)}\le y_1 \|f\|_{C^1(D)}
    \qquad\text{for all }y\in [0,1]\times (0,1)^{d-1},
  \end{equation*}
  and by continuity of $f_1$ the inequality extends to all $y\in D$.

  Thus for $x\in (0,1)^d$, $t\in [0,1]$
  \begin{equation*}
    X_1'(x,t)= f_1(X(x,t),t) \geq  -X_1(x,t)\|f\|_{C^1},\qquad X_1(x,0)=x_1>0,
  \end{equation*}
  or equivalently $-X_1'(x,t)  \leq -\|f\|_{C^1}(-X_1(x,T))$.
  Applying Gr\"onwall's inequality (in its differential form), we
  obtain that
  \[ -X_1(x,t)\le -x_1\exp (-t\|f\|_{C^1}) \iff X_1(x,t) \ge x_1\exp
    (-t\|f\|_{C^1}) >0. \]
	
  {\bf Step 2: Bijectivity and differentiability.} For any
    interior point $x\in (0,1)$, by Step 1 and the Picard-Lindel\"of
    theorem, there exists a unique solution
    $t\mapsto (X(x,t),t):[0,1]\to (0,1)^d\times [0,1]$ of
    \eqref{eq:ODE} (or \eqref{eq:Xxt}).  Consider the rime-reversed
    ODE
    \begin{equation}\label{eq:Yyt}
      Y'(y,s) = -f(Y(y,s),1-s),\qquad Y(y,0) =y.
    \end{equation}
    Clearly $\tilde f(z,\delta):=-f(z,1-\delta)$
    also belongs to $\cV$, cp.~\eqref{eqn:VelocityField}, and hence
    for any interior point $y\in (0,1)^d$, by Step 1 and the
    Picard-Lindel\"of theorem, there exists a unique solution
    $t\mapsto (Y(y,t),t):[0,1]\to (0,1)^d\times [0,1]$ of
    \eqref{eq:Yyt}. In either case, since the trajectories cannot
    cross, both maps $x\mapsto X(x,1):(0,1)^d\to (0,1)^d$ and
    $y\mapsto Y(y,1):(0,1)^d\to (0,1)^d$ are injective. Furthermore
    $X(Y(y,1),1)=y$ for all $y\in (0,1)^d$. By a continuity argument,
     we conclude that $X(\cdot,1):D\to D$ is
    bijective. The same argument yields that $x\mapsto X(x,t):D\to D$
    is bijective for any $t\in [0,1]$.

  Finally, using Corollary 3.1 in \cite{H02} as well as the
  subsequent remark (which are applicable since $f\in C^1$ by
  assumption), we see that in fact %
  $X(\cdot,t)\in C^1((0,1)^d)$
  and %
  $\det (\nabla X(x, t))\ne 0$.
  Once more by symmetry in the forward and backward in time ODEs,
    also $X(\cdot,t)^{-1}\in C^1((0,1)^d)$.  Thus $X(\cdot,t):D\to D$
    is a $C^1$-diffeomorphism.
\end{proof}

\begin{proof}[Proof of Theorem \ref{sara10}]
  The proof can be seen from making quantitative the arguments
  underlying Theorem 10.13 in \cite{VDG00}, in combination with
  several straightforward modifications of the assumptions there.
	
  Let $c$ be the constant from (10.70) in \cite{VDG00}. Then, since
  $p_n^*\ge c$ is lower bounded and $p_0\le K$ is upper bounded, it
  holds that $p_0/p_n*\le Kc^{-1}$ , whence the assumption (10.69) in
  \cite{VDG00} is clearly fulfilled. Next, we notice that the relevant
  entropy integral in Theorem 10.13 of \cite{VDG00} is given by the
  expression (for some constant $c>0$)
  \begin{align}\label{sara-entropy}
    \max \Big\{ R,  \int_{R^2/c}^R H_B^{1/2} \Big(\big\{ p\in \mathcal Q_n^\ast : h\big(\frac{p+p_n^\ast}{2}, p_n^\ast\big) \le \delta \big\}, L^2(p_n^*), \tau \Big) d\tau \Big\},
  \end{align}
  which is clearly upper bounded by our entropy integral
  $I_B(\mathcal P_n,R,p_n^\ast)$ from (\ref{entropy}). Thus,
  any choice $\Psi$ fulfilling the hypotheses of our theorem
  automatically also represents a desired upper bound for the entropy
  integral (\ref{sara-entropy}). It follows that Theorem 10.13 in
  \cite{VDG00} is applicable, and we obtain the convergence in
  probability
  \[ h((T^{\hat{f}})^\sharp\rho, p_0) = \mathcal{O}_{P_0^n}(\delta_n +
    h(p^*, p_0)).\]

  It remains to show the non-asymptotic concentration inequality
  from (\ref{conc-ineq}), which is a stronger statement than mere
  convergence in probability. This follows from an inspection
  of the proof of Theorem 10.13 of \cite{VDG00}, which we now
  detail. Indeed, the last step of the latter proof is based on the following
  case distinction for the terms $I$ and $II$ defined on p.191 of
  \cite{VDG00}.
	
  {\bf Case 1:} $I\le II$. In this case, denoting
  $\hat p:= (T^{\hat f})^\sharp \rho$, one obtains
  \[ h^2\Big(\frac{\hat p+p^*_n}{2}, p^*_n\Big)\le 4(1+c_0)h(p^*_n,p_0). \]
Here, $c_0$ can be any constant such that $p_0/p^*_n \le c_0^2$ (cf.\ (10.69) in \cite{VDG00}); in particular we may set $c_0 :=\sqrt{Lc^{-1}}$. Using Lemma 4.2 from
  \cite{VDG00}, it follows that
  \[ h^2(\hat p, p^*)\le 16 h^2\Big(\frac{\hat p+p^*}{2}, p^*\Big) \le
    64 (1+c_0)h(p^*,p_0).\]
	
  {\bf Case 2:} $I < II$. In this case, one obtains that
  \[h^2\Big(\frac{\hat p+p^*}{2}, p^*\Big)\le \int \log
    \Big(\frac{\hat p+p^*}{2p^*}\Big)d (P_n-P_0),\] where $P_n$
  denotes the empirical measure and $P_0$ is the data-generating
  law. In this case, using the same concentration arguments as in
  Theorem 7.4 of \cite{VDG00}, one obtains that for any
  $\delta\ge \delta_n$, and some $C>0$ only depending on $c_0$ (and thus only depending on $c,K$),
  \[ P_0^n\big(h(\hat p, p^*) \ge \delta \big) \leq C\exp\Big(
    -\frac{n\delta^2}{C}\Big).\] Then, using the triangle inequality
  \[ h(\hat p,p_0)\le h(\hat p,p^*)+h(p^*,p_0) \] completes the proof.
\end{proof}

\begin{proof}[Proof of Lemma \ref{thm:FlowMapBound}]\label{pf:lemma3.1}

For notational convenience let us write $\|f - g\|_{C^1(\Omega)} = \epsilon$ for some $\epsilon > 0$.
Then, for any $x\in D\subset\mathbb{R}^d$, we have that 
  \begin{align*}   
  \left |\frac{dX^f(x,t)}{dt} - \frac{dX^g(x,t)}{dt} \right | &= |f(X^f(x,t),t) -
  g(X^g(x,t),t)| \\
  &\leq |f(X^f(x,t),t) - f(X^g(x,t), t)|\\
  &+ |f(X^g(x,t), t) -
  g(X^g(x,t),t)|.
  \end{align*}
  By assumption, we have
  $\sup_{t\in[0,1]}\max_{i,j}\|(\nabla_x f(\cdot,
  t))_{i,j}\|_{C(D)}\leq r$, and it follows that
  \begin{align*}
    \sup_{t\in [0,1]}\|\nabla_x f(\cdot,t)\|_{C(D,\R^{d\times d})} &= \sup_{t\in[0,1], x\in D}\|\nabla_xf(x, t)\|_2 \leq \sup_{t\in[0,1],x\in D}\|\nabla_xf(x,t)\|_F \\
&\leq d \sup_{t\in[0,1]}\max_{i,j}\|(\nabla_x f(\cdot, t))_{i,j}\|_{C(D)}\leq dr,
  \end{align*}
  where we have equipped $\R^{d\times d}$ with the usual operator norm
  for matrices. Therefore, we can conclude that
  $|f(X^f(x,t),t) - f(X^g(x,t), t)| \leq dr|X^f(x,t) - X^g(x,t)|.$ Next, we have
  \begin{align*}
    &|X^f(x,t) - X^g(x,t)| = \Big|\int_0^tf(X^f(x,s),s) - g(X^g(x,s),s)ds\Big|\\
    &\leq \int_0^t|f(X^f(x,s),s) - g(X^g(x,s),s)|ds\\
    &\leq \int_0^t|f(X^f(x,s),s) - f(X^g(x,s),s)|ds + \int_0^t|f(X^f(x,s), s) - g(X^g(x,s), s)|ds\\
    &\leq dr\int_0^t |X^f(x,s) - X^g(x,s)|ds + t\epsilon    
  \end{align*}
  Using Gr\"onwall's inequality (integral form), we get
  \[|T^f(x) - T^g(x)| = |X^f(x,1) - X^g(x,1)|\leq \epsilon e^{dr},~~~~\forall
  x.\] Therefore
  $\max_{j}\|X_j^f(\cdot,1)-X_j^f(\cdot,1)\|_{C(D)}\leq \epsilon
  e^{dr}$.
	
  Now, it remains to bound
  $\max_{i,j}\|(\nabla_x T^f(\cdot) - \nabla_x
  T^g(\cdot))_{i,j}\|_{C(D)}$, which could be achieved by bounding
  the Frobenius norm of the difference in Jacobian
  $\|\nabla_x T^f(x) - \nabla_xT^g(x)\|_F$ by equivalence of
norms. 
  Similarly as above, we can write for $t\in[0,1]$
  \begin{align*}
    &\|\nabla_x X^f(x,t) - \nabla_xX^g(x,t)\|_F = \|\int_0^t\nabla_x(f(X^f(x,s),s) - g(X^g(x,s),s))ds\|_F\\
    &\leq \int_0^t\|\nabla_x(f(X^f(x,s),s) - g(X^g(x,s),s))\|_Fds\\
    &= \int_0^t\Big\|\left(\nabla_Xf(X^f(x,s),s)\nabla_xX^f(x,s) - \nabla_Xg(X^g(x,s),s)\nabla_xX^g(x,s)\right)\Big\|_Fds\\
    &\leq \int_0^t\|(\nabla_Xf(X^g(x,s),s) - \nabla_Xg(X^g(x,s),s))\nabla_x X^g(x,s)\|_Fds\\  &~~~~~~+\int_0^t\|(\nabla_Xf(X^f(x,s),s)-\nabla_Xf(X^g(x,s),s))\nabla_xX^g(x,s)\|_Fds \\ &~~~~~~+\int_0^t\|\nabla_Xf(X^f(x,s),s)(\nabla_xX^f(x,s) - \nabla_xX^g(x,s))\|_Fds\\
    &=: I+II+III.
  \end{align*}

  To bound term $I$, note that $\|f - g\|_{C^1(\Omega)} = \epsilon$ gives
  \[|(\nabla_Xf(X^g(x,t),t) - \nabla_Xg(X^g(x,t),t))_{i,j}|
  \leq \epsilon ~~~~\text{for all}~(x,t)\in \Omega,~(i,j)\in [d]^2.\] To establish bounds on
  $\nabla_xX^g(x,t)$, we note that
  \begin{align*}
  \nabla_x X^g(x,t) &= I_{d\times d} + \int_0^t\nabla_x \big[g(X^g(x,s),s)\big]ds\\
  &= I_{d\times d} + \int_0^t(\nabla_Xg(X^g(x,s), s))\nabla_xX^g(x,s)ds, \\  
  \end{align*}

where $I_{d\times d}$ is the identity matrix of dimension $d$. Using the 
standard multiplication inequality $\|M_1M_2\|_F\le \|M_1\|_F\|M_2\|_F$ for the Frobenius norm and $\|g\|_{C^1(\Omega)}\leq r$, it follows that for all points $(x,t)$,
 \begin{align*}
     \|\nabla_x X^g(x,t)\|_F&\leq\sqrt{d} +\int_0^t\|\nabla_Xg(X^g(x,s),s)\|_F\|\nabla_x X^g(x,s)\|_Fds\\
     &\leq \sqrt{d} +dr\int_0^t\|\nabla_xX^g(x,s)\|_Fds.
 \end{align*}
By Gr\"onwall's inequality, it follows that $\|\nabla_x X^g(x,t)\|_F \leq \sqrt{d}e^{drt}$ and in particular
        $\|\nabla_x X^g(x,1)\|_F \leq \sqrt{d}e^{dr}$. Therefore,
	\begin{align*}
          &\|(\nabla_Xf(X^g(x,s),s) - \nabla_Xg(X^g(x,s),s))\nabla_x X^g(x,s)\|_F\\
          &\leq \|\nabla_x X^g(x,s)\|_F\|\nabla_Xf(X^g(x,s),s) - \nabla_Xg(X^g(x,s),s)\|_F\leq  d^{\frac{3}{2}}\epsilon e^{drs},    
	\end{align*}
        and term $I$ may be bounded as
        \begin{align*}
          I &= \int_0^t\|(\nabla_Xf(X^g(x,s),s) - \nabla_Xg(X^g(x,s),s))\nabla_x X^g(x,s)\|_Fds \\
            &\leq \int_0^td^{\frac{3}{2}}\epsilon e^{drs}ds = \frac{(e^{drt}-1)\sqrt{d}\epsilon}{r}\leq\frac{(e^{dr}-1)\sqrt{d}\epsilon}{r}.
        \end{align*}
	
	To bound $II$, by the Lipschitz property, we have at any
        $(x,t)$,
        $$\|\nabla_Xf(X^f(x,t),t)-\nabla_Xf(X^g(x,t),t)\|_F\leq r|X^f(x,t) -
        X^g(x,t)|.$$ Since
        $|X^f(x,s) - X^g(x,s)|\leq \epsilon e^{dr}$ at all point $(x,s)$
        from the previous part, we obtain that
	\begin{align*}
          \|(\nabla_Xf(X^f(x,s),s)-&\nabla_Xf(X^g(x,s),s))\nabla_xX^g(x,s)\|_F \\
          &\leq \|\nabla_xX^g(x,s)\|_F\|\nabla_Xf(X^f(x,s),s)-\nabla_Xf(X^g(x,s),s)\|_F\\
&\leq \sqrt{d}e^{drs}r\epsilon e^{drs} = \sqrt{d}r\epsilon e^{2drs}.
	\end{align*}
	Then, we have
\begin{align*}
    II &= \int_0^t\|(\nabla_Xf(X^f(x,s),s)-\nabla_Xf(X^g(x,s),s))\nabla_XX^g(x,s)\|_Fds \\
    &\leq \int_0^t\sqrt{d}r\epsilon e^{2drs}ds \leq \frac{(e^{2dr} - 1)r\epsilon}{2\sqrt{d}r}
\end{align*}

	Finally, to bound $III$, we
        have
        \begin{align*}
        \|\nabla_Xf(X^f(x,s),s)&(\nabla_xX^f(x,s) -\nabla_xX^g(x,s))\|_F\\
&\leq\|\nabla_Xf(X^f(x,s),s)\|_F\|\nabla_xX^f(x,s)
        - \nabla_xX^g(x,s)\|_F,
        \end{align*}
where we can bound
        $\|\nabla_Xf(X^f(x,s),s)\|_F$ by $dr$.

Combining all the terms, we obtain
 \begin{align*}
 \|\nabla_x &(X^f(x,t) -X^g(x,t))\|_F\\
 &\leq \frac{(e^{dr}-1)\sqrt{d}\epsilon}{r}  + \frac{(e^{2dr} - 1)r\epsilon}{2\sqrt{d}r} + dr\int_0^t\|\nabla_xX^f(x,s) - \nabla_xX^g(x,s)\|_Fds.     
 \end{align*}

	Gr\"onwall's inequality then gives
        $$\|\nabla_x (X^f(x,1) -X^g(x,1))\|_F \leq \epsilon
        \frac{(e^{2dr}-1)r + 2(e^{dr}-1)d}{2\sqrt{d}r}e^{dr}\leq
        \frac{re^{3dr} + 2de^{2dr}}{2\sqrt{d}r}\epsilon$$.
	
	Since the above inequality holds for all $x$, considering
        pointwise entries in the Jacobian
        gives
        $$\max_{i,j}\|(\nabla_xX^f(\cdot,1)
        -\nabla_xX^g(\cdot,1))_{i,j}\|_{C(D)} \leq \frac{re^{3dr} +
          2de^{2dr}}{2\sqrt{d}r}\epsilon.$$
	
Using the $C^1$ norm, we conclude that
\begin{align*}
\|T^f(x) - T^g(x)\|_{C^1(D)} &= \|X^f(x,1) - X^g(x,1)\|_{C^1(D)}\\
&\leq \max\Big\{\epsilon e^{dr}, \frac{re^{3dr} + 2de^{2dr}}{2\sqrt{d}r}\epsilon\Big\} \leq C\|f - g\|_{C^1(\Omega)},
\end{align*}

where
$C = \max\{e^{dr},\frac{re^{3dr} + 2de^{2dr}}{2\sqrt{d}r}\}$.
\end{proof}

\begin{proof}[Proof of Theorem \ref{thm:LInfDensityBound}]\label{pf:thm3.2}
  For notational convenience, let us write $\|T - G\|_{C^1(D)} = \epsilon$.
  \begin{align*}
    \|T^\sharp\rho - G^\sharp\rho\|_{C(D)} 
    &= \|\rho(T)\det\nabla T - \rho(G)\det\nabla G\|_{C(D)}\\
    &\leq \|\rho(T)\det\nabla T - \rho(G)\det\nabla T\|_{C(D)} +\|\rho(G)\det\nabla T- \rho(G)\det\nabla G\|_{C(D)}\\
    &\leq \|\rho(T) - \rho(G)\|_{C(D)}\|\det\nabla T)\|_{C(D)}+\|\rho(G)\|_{C(D)}\|\det\nabla T - \det\nabla G\|_{C(D)}.
  \end{align*}
We bound these two terms separately. By the Lipschitz continuity of $\rho$, we have
  $\|\rho(T) - \rho(G)\|_{C(D)} \leq \|\rho\|_{\textnormal{Lip}}\|T -
  G\|_{C(D)} \leq |\rho|_{\textnormal{Lip}}\epsilon$.
	
Moreover, by \cite[Lemma E.1]{ZM22a},
  \begin{align*}
    |\det\nabla T(x) - \det\nabla G(x)| &= \left |\prod_{i=1}^d\lambda_i(x) - \prod_{i=1}^d\eta_i(x) \right | \\
    &\leq \frac{\exp{\left(\sum_{i=1}^d\frac{|\lambda_i(x) - \eta_i(x)|}{\lambda_d(x)}\right)}\prod_{i=1}^d\lambda_i(x)}{\min\{\lambda_d(x), \eta_d(x)\}}\sum_{i=1}^d|\lambda_i(x) - \eta_i(x)|
  \end{align*}
  which is upper bounded by
  $\tilde{C}\sum_{i=1}^d|\lambda_i(x) - \eta_i(x)|.$

  By Weyl's theorem,
  $\max_i|\lambda_i(x) - \eta_i(x)| \leq \|\nabla T(x) - \nabla
  G(x)\|_2$. Furthermore,
  $\forall x\in D, \|\nabla T(x) - \nabla G(x)\|_2 \leq
  d\epsilon$. Thus we can conclude that
  $|\det\nabla T(x) - \det\nabla G(x)| \leq \tilde{C}d^2\epsilon$ at
  all $x\in D$, from which it follows that
  $\|\det\nabla T - \det\nabla G\|_{C(D)} \leq
  \tilde{C}d^2\epsilon$.
	
  Putting everything together, we have
	$$\|T^\sharp\rho - G^\sharp\rho\|_{C(D)} \leq \left(|\rho|_{\textnormal{Lip}}\|\det\nabla T(x)\|_{C(D)} + \tilde{C}d^2\|\rho\|_{C(D)}\right)\epsilon.$$
	
	Finally, using the fact that for a $d$-dimensional matrix $A$,
        $|\det A| \leq (\frac{\tr A}{d})^d$, we have
        \[\|T_\sharp\rho - G_\sharp\rho\|_{C(D)} \leq
          \left(|\rho|_{\textnormal{Lip}}\|T\|_{C^1(D)}^d +
            \tilde{C}d^2\|\rho\|_{C(D)}\right)\epsilon.\]
      \end{proof}

\begin{proof}[Proof of Theorem \ref{thm:SingularValueBound}]
  For notational convenience we write $X=X^f$. The map $(T^{f}): D \rightarrow D$ is obtained by integrating
  the ODE \ref{eq:ODE} forward in time, i.e., $(T^{f})(x) = x + \int_0^1f(X(x,t), t)dt$. Taking the
  operator norm of the Jacobian, we obtain
  \begin{align*}
    \|\nabla_x (T^{f})(x)\|_2 &= \Big\| I_{d\times d} + \int_0^1\nabla_yf(X(y,t), t)dt \Big\|_2\\
&= \Big\| I_{d\times d} + \int_0^1\nabla_Xf(X(x,t), t)\nabla_xX(x,t)dt \Big\|_2\\
&\leq 1 + \int_0^1\|\nabla_Xf(X(x,t), t)\|_2\|\nabla_xX(x,t)\|_2dt\\
  \end{align*}
	
  Since $\|f\|_{C^1(D\times[0,1])} \leq M$,
  we have $\|\nabla_Xf(X(x,s), s)\|_2 \leq dM, \forall
  s\in[0,1]$. On the other hand,
  \begin{align*}
    \|\nabla_x X(x,t)\|_2 &=  \|I_{d\times d} + \int_0^t\nabla_X f(X(x,s), t)\nabla_x X(x,s)ds\|_2\\ 
&\leq 1+ \int_0^t\|\nabla_X f(X(x,s),s)\|_2\|\nabla_x X(x,s)\|_2ds\\
&\leq 1 + dM\int_0^t\|\nabla_x X(x,s)\|_2ds.
  \end{align*}
  It follows from Gr\"onwall's inequality that
  $\|\nabla_x X(x,t)\|_2\leq e^{dMt} \leq e^{dM}$,
  $\forall t\in[0,1]$. Putting things together, we get
  $\|\nabla_x T^{f}(x)\|_2 \leq 1 + dMe^{dM}, \forall
  x\in D$, from which it follows that
  $\lambda^f_1(x) \leq 1 + dMe^{dM}, \forall x\in D$.

  On the other hand, consider $\lambda^f_d(x)$, the smallest singular
  value of $\nabla_x T^{f}(x)$. By the inverse function theorem, for all $x\in D$, writing  $y = T^f(x)$, we have that
  $\nabla_y (T^{f})^{-1}(y) = (\nabla_x T^f(x))^{-1}$. It follows that
  \[\frac{1}{\lambda^f_d(x)} = \big\|[\nabla_x
  T^f(x)]^{-1}\big\|_2 = \big\|\nabla_y
  (T^{f})^{-1}(y)\big\|_2.\]
  Observe that the inverse
  transport $(T^f)^{-1}$ is given by integrating the ODE backwards in
  time.  For this purpose, consider the following reverse ODE. For
  $y \in D$ so that $y = x + \int_0^1 f(X(x,t),t)dt$ and
  $Y(y,t) = X(x, 1-t)$, we have
  \begin{equation}
  \begin{cases}
    \frac{d Y(y,t)}{dt} &= -f(Y(y,t), 1-t),\\
    Y(y, 0) &= y.
  \end{cases}
  \end{equation}
Then, by a similar argument as above, we can show
  $\|\nabla_y (T^f)^{-1}(y)\|_2 \leq 1 + dMe^{dM}$. Thus we have shown that
  $\lambda_d^f(x) \geq \frac{1}{1 + dMe^{dM}}, \forall x\in D$.
\end{proof}

\subsection{Auxiliary results}
We %
show three elementary lemmas. %
The first two provide bounds on the (bracketing) metric entropy.
\begin{lemma}\label{lemma:InfToBracketing}
  Let $\mu$ be a measure on $D=[0,1]^d$ with positive Lebesgue %
  density
  and let $\cF\subseteq C(D,\R^d)$.
  Then for
  all $\tau>0$ it holds that
  \[N_B(\mathcal{F}, L^2(D,\mu),\tau) \leq N\Big(\mathcal{F},
    L^\infty(D), \frac{\tau}{2\sqrt{\mu(D%
        )}} \Big).\]
\end{lemma}
 
\begin{proof}
  Let $N:=N(\mathcal{F}, %
  {L^\infty(D)},
  \tau)$. By the definition of metric entropy, there exist functions 
$f_1,\dots,f_N$ on $D$ such that for
    each $f\in\cF$ exists $i\in\{1,\dots,N\}$ with
    $\|f - f_i\|_{L^\infty(D)} \leq \tau$.
  For each $i\le N$, set $f_{i, L} := f_i - \tau$ and
  $f_{i,U} := f_i + \tau$. Then %
$f_{i,L} \leq f \leq f_{i, U}$ on $D$. 
Since $\|f_{i,L}(x) - f_{i, U}(x) \|_{L^2(\mu)} \leq
  2\tau\sqrt{\mu(%
    D)}$, this implies
\begin{equation*}
  N_B(\cF,L^2(\mu), 2\tau\sqrt{\mu(D)})\le
  N(\cF,L^\infty(D),\tau)
\end{equation*}
  for all
  $\tau>0$ (cp.~Definition \ref{def:bme}).
\end{proof}

Similarly:
\begin{lemma}\label{lemma:NLip}
  Let $(X,\norm[X]{\cdot})$, $(Y,\norm[Y]{\cdot})$ be two normed spaces
  and $A\subseteq X$. Let $\Phi:A\to Y$ be Lipschitz continuous with Lipschitz constant
  $L$. Then for all $\tau>0$
  \begin{equation*}
    N(\Phi(\cF),Y,\tau)\le N(\cF,X,\frac{\tau}{2L}).
  \end{equation*}
\end{lemma}
\begin{proof}
  Fix $\tau>0$ and set
  $\tilde\tau:=\frac{\tau}{2L}$ and
  $N:=N(\cF,X,\tilde\tau)$. Then we can find $x_1,\dots,x_N\in
  X$ such that for each $x\in \cF$ exists
  $i\in\{1,\dots,N\}$ with $\norm[X]{x-x_i}\le
  \tilde\tau$. In particular, we can find $\tilde x_1,\dots,\tilde
  x_N\in\cF$ such that for each $x\in\cF$ exists
  $i\in\{1,\dots,N\}$ with $\norm[X]{x-\tilde x_i}\le
  2\tilde\tau$.

  Let $y\in\Phi(\cF)$ arbitrary, i.e.,  $y=\Phi(x)$ for some $x\in\cF$.
  Then there exists $i\in\{1,\dots,N\}$ such that
  \begin{equation*}
    \norm[Y]{y-\Phi(\tilde x_i)}=
  \norm[Y]{\Phi(x)-\Phi(\tilde x_i)}
  \le L\norm[X]{x-\tilde x_i}
  \le 2L\tilde \tau=\tau.
\end{equation*}
This shows the claim.
\end{proof}

The next lemma states that the Hellinger distance is bounded by the
$L^\infty$-distance whenever %
the maximum of both densities is bounded from below.
 
\begin{lemma}\label{lemma:LipschitzH2Linf}
  Let $L>0$ and $D\subseteq \R^d$ measurable. Then for all
    probability densities $p_1(x)$,
    $p_2(x)$ %
    on $D$ with $\essinf_{x\in D}\max\{p_1(x),p_2(x)\}\ge L$, it
    holds
	$$h(p_1, p_2) \leq \frac{1}{\sqrt{2L}}\| p_1-p_2 \|_{L^\infty(D)}.$$
      \end{lemma}

\begin{proof} %
  We have
  \begin{align*}
    h(p_1, p_2)^2 &= \frac{1}{2}\int_D(\sqrt{p_1(x)} - \sqrt{p_2(x)})^2dx= \frac{1}{2}\int_D \left(\frac{p_1(x) - p_2(x)}{\sqrt{p_1(x)} + \sqrt{p_2(x)}}\right)^2dx\\
    &\leq \frac{1}{2}\int_D \left(\frac{\|p_1 - p_2\|_{L^\infty(D)}}{\sqrt{L}}\right)^2dx = \frac{1}{2L}\|p_1 - p_2\|_{L^\infty(D)}^2.
  \end{align*}
\end{proof}

\section{Proofs for Section \ref{Ck-theory}}

\subsection{Proof of Theorem \ref{Thm:existencCkField}}\label{sec:Ck-proofs}

In order to prove Theorem \ref{Thm:existencCkField}, we need three auxiliary Lemmas \ref{lem:anisotropic}, \ref{lemma:hardy1}, and \ref{lemma:hardy2}. These auxiliary statements regard certain anisotropic regularity classes which describe Knothe--Rosenblatt maps between $C^k$-smooth densities, as was observed in \cite{TransportMiniMax}. For any $k\ge 1$, we write $[k]=\{1,\dots, k\}$.

\subsubsection{Step 1: Anisotropic regularity classes}

For $k\ge 1$ integer, let us define the following classes of triangular functions on $D$ with anisotropic regularity :
\[ C^{k}_{diag}(D, \R^d)=\big\{f\in C^k(D,\R^d)~\text{triangular}: \forall j\in [d]: ~\partial_j f_j \in C^k(D) \big\}, \]
with norm
\[ \|f\|_{C^k_{diag}(D)} := \sum_{j=1}^d \|f_j\|_{C^k([0,1]^j)} + \|\partial_jf_j\|_{C^k([0,1]^j)}. \]
We further introduce the class of bijective, monotone triangular maps with such anisotropic regularity:
\begin{align*}
    \mathcal A^k_{diag} &:= \Big\{S:D\to D~\text{triangular and bijective},\\
    &S\in C^k_{diag}(D,\R^d),~\forall j\in [d]:\partial_jS_j > 0 \Big\}.
\end{align*}
For any constants $0<c_{mon}<L<\infty$, we will also need the sub-classes with bounded norm
\begin{align*}
    \mathcal A^k_{diag}&(c_{mon},L):=\Big\{ S\in \mathcal A^k_{diag},~\|S\|_{C^k_{diag}(D, \R^d)}\le L, \inf_{x\in [0,1]^k}\partial_kS_k(x)\ge c_{mon} \Big\}.
\end{align*}
The following lemma shows that the above classes are closed under composition and inversion.
\begin{lemma}\label{lem:anisotropic}
\textbf{(i)} If $S,R \in \mathcal A^k_{diag}$, then $S\circ R\in A^k_{diag}$. Moreover, for any $c_{mon}, L>0$ there exist $c_{mon}', L'>0$ such that for any $S,R \in \mathcal A^k_{diag}(c_{mon},L)$, it holds that $S\circ R \in \mathcal A^k_{diag}(c_{mon}',L')$.

\smallskip 

\textbf{(ii)} If $S\in \mathcal A^k_{diag}$, then also $S^{-1}\in A^k_{diag}$. Moreover, for any $c_{mon}, L>0$ there exist $c_{mon}', L'>0$ such that for any $S \in \mathcal A^k_{diag}(c_{mon},L)$, it holds that $S^{-1} \in \mathcal A^k_{diag}(c_{mon}',L')$.
\end{lemma}

\begin{proof}
We begin by proving part \textbf{(i)}. First, we observe that $S\circ R$ is still bijective and triangular. To see the triangularity, we observe that
\[ [S\circ R] (x)= \begin{bmatrix*}[l]
    S_1(R_1(x_1))\\
    \qquad \vdots\\
    S_j(R_1(x_1),\dots, R_j(x_{[j]}))  \\
    \qquad \vdots\\
    S_d(R_1(x_1),\dots\dots\dots, R_d(x))
\end{bmatrix*}. \]
Thus, the $j$-th component map only depends on the first $j$ coordinates of $x$. Next, it is also clear that since $S,R\in C^k(D,\R^d)$, also $S\circ R\in C^k(D,\R^d)$. It remains to assert the regularity of the `diagonal derivatives' of $S\circ R$. For any $j\in [d]$, denoting $y(x)=[R_1(x_1),\dots,R_j(x_{[j]})]$, using the chain rule and the fact that $R_l, ~l<j$ is independent of $x_j$, we obtain that
\begin{align*}
    \frac{\partial}{\partial x_j}(S\circ R)_j&=
    \frac{\partial}{\partial x_j} \big[S_j(R_1(x_1),\dots, R_k(x_{[j]}) \big] \\
    &= \sum_{l=1}^j \Big(\frac{\partial}{\partial y_l} S_j\Big) (y(x)) \Big(\frac{\partial}{\partial x_j} R_l\Big) (x_{[l]})\\
    &=\Big(\frac{\partial}{\partial y_j} S_j\Big) (y(x)) \Big(\frac{\partial}{\partial x_j} R_j\Big) (x_{[j]}).
\end{align*}
Since $\frac{\partial}{\partial y_j} S_j: (0,1)^j\to \R$ is a $C^k$ function, $y: x\mapsto y(x),~(0,1)^j\to (0,1)^j$ is $C^k$ function, and finally also $\frac{\partial}{\partial x_j} R_j: (0,1)^j\to \R$ is $C^k$ function, we overall obtain that $\frac{\partial}{\partial x_j}(S\circ R)_j$ is also $C^k$. It is clear from the chain rule and standard multiplication inequalities for $C^k$-norms that we may choose the upper bound $L'$ for the norm of $S\circ R$ just depending on $L$. 
Moreover, the preceding calculation clearly implies that for any $j=1,\dots, d$, 
\[\frac{\partial}{\partial x_j}(S\circ R)_j \ge c_{mon}^2=: c_{mon}'>0,\] 
which completes the proof of part \textbf{(i)}.

Let us now turn to part \textbf{(ii)}. Let $\rho$ be the uniform density on $D$. Using Proposition 2.1 in \cite{TransportMiniMax}, we know that for any $S\in \mathcal A_{diag}^k(c_{mon}, L)$, the pullback distribution $p_S:=S^{\sharp}\rho$ is an upper and lower bounded $C^k(D)$ density, where the $C^k$ norm, and the upper and lower bounds only depend on $c_{mon}$ and $L$. Moreover, clearly $S^{-1}$ is again triangular and bijective. Moreover, it satisfies $(S^{-1})^{\sharp}p_S= \rho$. By uniqueness of the KR-transport map, $S^{-1}$ consitutes the unique Knothe--Rosenblatt transport map between two $C^k$ densities. Thus, again using Proposition 2.1 in \cite{TransportMiniMax}, we see that $S^{-1}\in \mathcal A_{diag}^k(c_{mon'},L')$ for some $c_{mon}',L>0$. This concludes the proof of the lemma.
\end{proof}

\subsubsection{Step 2: A Hardy-type inequality for functions with anisotropic regularity}

\begin{lemma}\label{lemma:hardy1}
Let $d$, $k\in\N$, $D=[0,1]^d$, and
$f\in C^k(D)$ such that $\partial_{x_d} f\in C^k(D)$
    and additionally $f(x)=0$ whenever $x_d=0$.
    Then $g(x):=\frac{f(x)}{x_d}\in C^k(D)$ and
    there exists $C=C(d,k)$ such that
    \begin{equation}\label{eq:fg}
      \norm[{C^k(D)}]{g}\le C\norm[{C^k(D)}]{\partial_{x_d}f}.
    \end{equation}
  \end{lemma}
  \begin{proof}
    Since $f(x_1,\dots,x_{d-1},0)=0$
    and $x_d\mapsto f(x)\in C^{k+1}([0,1])$ for all $x_{[d-1]}\in [0,1]^{d-1}$, it follows that for all $x\in D$ and any $l\in \{0,\dots ,k\}$,
    \begin{equation*}
      f(x) = \sum_{j=1}^l \partial_{x_d}^{j}f(x_1,\dots,x_{d-1},0)\frac{x_d^j}{j!} + \int_0^{x_d}\partial_{x_d}^{l+1}f(x_1,\dots,x_{d-1},\textcolor{blue}{t})\frac{(x_d-t)^l}{l!}\dd t
    \end{equation*}
    and thus for any $l\in\{0,\dots,k\}$
    \begin{equation*}
      \frac{f(x)}{x_d} =\underbrace{\sum_{j=0}^{l-1} \partial_{x_d}^{j+1}f(x_1,\dots,x_{d-1},0)\frac{x_d^j}{j!}}_{=:g_1(x)} + \underbrace{\frac{1}{x_d}\int_0^{x_d}\partial_{x_d}^{l+1}f(x_1,\dots,x_{d-1},t)\frac{(x_d-t)^l}{k!}\dd t}_{=:g_2(x)}.
    \end{equation*}

    Now, fix a multiindex $\bsv\in\N^d$ such that $|\bsv|\le k$. To prove the lemma, we need to show that $\sup_{x\in D}|\partial^\bsv f(x)|$
    is bounded by the right-hand side of \eqref{eq:fg} with some $C$
    solely depending on $k$ and $d$. Set $l:=k-v_d\ge 0$. Clearly
    \begin{equation*}
      |\partial^\bsv g_1(x)|
      \le \sum_{j=0}^{l-1}\Big|\partial^\bsv\Big(\partial_{x_d}^{j+1}f(x_1,\dots,x_{d-1},0)\frac{x_d^j}{j!}\Big)\Big|
      \le C\norm[C^k(D)]{\partial_{x_d}f}.
    \end{equation*}
    For $g_2$, we first observe that with the change of variables $t=x_ds$, we obtain
    \begin{equation*}
      g_2(x)=%
      \int_0^1 \partial_{x_d}^{k+1}f(x_1,\dots,x_{d-1},x_ds)\frac{x_d^l(1-s)^l}{l!}\dd s.
    \end{equation*}
    Exchanging the integral with the derivative
    and repeatedly applying the product rule we find
    \begin{equation*}
      |\partial^\bsv g_2(x)|
      \le\int_0^1 \Big|\partial^\bsv\Big(\partial_{x_d}^{k+1} f(x-1,\dots,x_{d-1},x_d s)\frac{x_d^k(1-s)^k}{k!}\Big)\Big|\dd s
      \le C \norm[C^k(D)]{\partial_{x_d}f}.\qedhere
    \end{equation*}
  \end{proof}

  \begin{lemma}\label{lemma:hardy2}
    Consider the setting of Lemma \ref{lemma:hardy1}, and additionally
    assume that $f(x)=0$ whenever $x_d=1$. Then,
    $g(x):=\frac{f(x)}{x_d(1-x_d)}\in C^k(D)$ and there exists
    $C=C(k,d)$ such that $\norm[{C^k(D)}]{g}\le C\norm[{C^k(D)}]{\partial_{x_d}f}$.
  \end{lemma}
  \begin{proof}
  We already know from the preceding lemma that the map $ x\mapsto f(x)/x_d$ belongs to $C^k$. In order to show that $g \in C^k$, we only need to prove that the restrction of $g$ to the `half-cube' $\tilde D= \{x\in D: x_d\ge 1/2\}$ belongs to $C^k$. To this end, let us define
  \[\tilde f(x):= \frac{f(x_1,\dots ,x_{d-1},1-x_d)}{1-x_d}. \]
  Clearly, showing that $g\in C^k(\tilde D)$ is equivalent to showing that $x\mapsto \tilde f(x) / x_d$, restricted to the `other half-cube' $\{x\in D: x_d\le 1/2\}$. For this, we just need to show that $\tilde f$ satisfies the conditions of Lemma \ref{lemma:hardy1}. Since $1-x_d$ is bounded below when $x_d\le 1/2$, clearly $\tilde f$ has the needed regularity. Moreover, for any $x$ with $x_d=0$, $\tilde f(x)=f(x_1,\dots ,x_{d-1},1)=0$ by assumption. We may thus apply Lemma \ref{lemma:hardy1} and the proof is complete.
  \end{proof}

\subsubsection{Step 3: The main argument}

With the previous lemmas in hand, we are now ready to prove  Theorem \ref{Thm:existencCkField}.

\begin{proof}[Proof of Theorem \ref{Thm:existencCkField}]
The first assertion \textbf{(i)} of the Theorem is proven in \cite{ren23a}, we thus only need to show the second part.

Let $p_0\in \mathcal M(k,L_1,L_2)$. Then, it is proven in \cite{TransportMiniMax} that the unique KR map $T$ between $p_0$ and $\rho$ belongs to the anisotropic regularity class $\mathcal A^k_{diag}(k,L',c_{mon}')$ for some $L',c_{mon}'$. Since  the identity map $\mathrm{Id}:x\mapsto x$ also belongs to $\mathcal A^k_{diag}(k,L',c_{mon}')$, and since $\mathcal A^k_{diag}(k,L',c_{mon}')$ is a convex set, we know that for any $t\in [0,1]$, $G_t= tT+(1-t)\mathrm{Id}$ also belongs to $\mathcal A^k_{diag}(k,L',c_{mon}')$. By Lemma \ref{lem:anisotropic}, it follows that for any $t\in [0,1]$, $F(\cdot,t)$ also satisfies the same isotropic regularity. As a result, we know that the triangular velocity field $f_{p_0}^\Delta$ belongs to the class $C^k_{diag}(D,\R^d)$. Moreover, since $f^\Delta_{p_0}$ is the difference between two bijective triangular maps $D\to D$, we know that for every $j\in [d]$ and for every $x_{[j-1]}\in [0,1]^{j-1}$ each component map satisfies $(f^\Delta_{p_0})_j(x_{[j-1]},0)=(f^\Delta_{p_0})_j(x_{[j-1]},1)=0$. Thus, all the assumptions of Lemma \ref{lemma:hardy2} are satisfied, and it follows that for every $j\in [d]$, the function 
\[g_j(x) = \frac{(f^\Delta_{p_0})_j(x)}{x_j(1-x_j)}\] belongs to $C^k(D)$. The corresponding norm bound for $g_j$ also follows from Lemma \ref{lemma:hardy2}.
\end{proof}

\subsection{Proof of Theorem \ref{thm-holder}}

\subsubsection{Metric entropy bounds for H\"older-Zygmund spaces}

To prove Theorem \ref{thm-holder}, we will begin by deriving the necessary metric entropy bounds for $\mathcal F(r)$. 
For non-integer $s>0$ we denote by $C^{s}$ the standard H\"older spaces of $\lfloor s \rfloor$-times differentiable functions with $s-\lfloor s\rfloor$-H\"older continuous $s$-th partial derivatives, normed by 
\[\|f\|_{C^{s}(\Omega)} = \|f\|_{C^{\lfloor s\rfloor}(\Omega)} + \max_{|\bsalpha|=\lfloor s\rfloor}\sup_{x\neq y\in\Omega}\frac{|\partial^\bsalpha f(x) - \partial^\bsalpha f(y)|}{|x - y|^{s-\lfloor s\rfloor}}.\]
For $s\ge 0$, we will further denote by $B^{s}_{\infty\infty}(\Omega)$ the classical Besov spaces with indices $p=q=\infty$; see \cite{T08} for definitions. It is well known that those spaces are equal to the H\"older-Zygmund spaces $\mathcal C^s(\Omega)$, $B^s_{\infty\infty}(\Omega)=\mathcal C^s(\Omega)$. Moreover, for non-integer $s>0$, they are equivalent to H\"older spaces,
\[ B^{s}_{\infty\infty}(\Omega)=\mathcal C^s(\Omega)=C^{s}(\Omega). \]
For any $s>0$ and $R>0$, let us denote the closed ball with radius $R$ in $B^s_{\infty\infty}(\Omega)$ by 
\[A^{s}(R) := \{ f\in L^2(\Omega): \|f\|_{B^{s}_{\infty\infty}(\Omega)}\le R\},\qquad R>0.\] 

The following lemma on metric entropies of Besov spaces is based on classical results which can be found e.g., in Triebel \cite{T08}.
\begin{lemma}\label{lem:entropy}
Let $s_1, s_2 > 0$ and $s_1 > s_2$. Then, there exists some constant $C=C(d,s_1,s_2)>0$ such that for any $R>0$, $\tau>0$,
	\[ H(A^{s_1}(R), B^{s_2}_{\infty\infty}(\Omega), \tau ) \le C\big( R/\tau \big)^{\frac{d}{s_1-s_2}}. \]
\end{lemma}

\begin{proof}
	This result follows from Theorem 4.33 in \cite{T08} with $p_0,p_1,q_0,q_1=\infty$ and $s_1,s_2$ in place of $s_0,s_1$ there. Note that with those choices, the requirement (4.126) in \cite{T08} is satisfied. Indeed, the theorem in \cite{T08} implies that for any $k$, the unit ball $A^{s_1}(1)$ in $B^{s_1}_{\infty\infty}$ can be covered by $2^k$ many balls of $\|\cdot\|_{B^{s_2}_{\infty\infty}}$ -radius at most $ ck^{-\frac{s_1-s_2}{d}}$, where $c>0$ is some constant. Therefore, given any $\tau>0$, setting $k_\tau = \lfloor (\tau/c)^{-\frac{d}{s_1-s_2}} \rfloor +1$, we obtain that the $\tau$-covering number of $A^s(1)$ is upper bounded by
	\[ H(A^{s_1}(1),B^{s_2}_{\infty\infty}(\Omega) ,\tau)\le \log ( 2^{k_\tau})= (\lfloor (\tau/c)^{-\frac{d}{s_1-s_2}} \rfloor +1 ) \log 2 \lesssim \tau^{-\frac{d}{s_1-s_2}} \]
	for any $\tau \le 1$. [Note that for $\tau \ge 1$, we have that the left hand side is $0$, so that the upper bound in the lemma trivially holds true for $R=1$.] Then, the result for covering $A^{s_1}(R)$ follows from noting that $H(A^{s_1}(R),B^{s_2}_{\infty\infty}(\Omega) ,\tau)= H(A^{s_1}(1),B^{s_2}_{\infty\infty}(\Omega) ,\tau/R)$.
\end{proof}

\subsubsection{Proof of Theorem \ref{thm-holder}}

Let $p_0\in \mathcal M(k,L_1,L_2)$. Since $k\ge 2,$ clearly the Assumptions \ref{ass:densities}, 
 \ref{ass:Fbounded} are fulfilled. By Theorem \ref{Thm:existencCkField}, the velocity field $f^\Delta_{p_0}$ coupling $p_0$ and $\rho$ belongs to $C^k\cap \mathcal V$, and we have that
\[ \sup_{p_0\in \mathcal M(k,L_1,L_2)} \|f^\Delta_{p_0}\|_{C^k(\Omega)} =: \bar L <\infty. \]
	Thus, by choosing $r > \bar L$ we can ensure that $f_{p_0}^\Delta\in \mathcal F(r)$. We will now employ Theorem \ref{thm:general}. By what precedes, we may choose $f^*=f_{p_0}^\Delta$, so that $(T^{f^*})^\sharp \rho =p^*=p_0$. We now calculate the metric entropy integral of $\mathcal F(r)$ in the $C^1(\Omega)$-norm. To do so, let us fix $\gamma\in (0,1)\cap k-d/2-3/2$, we have that
\[ B^{1+\gamma}_{\infty\infty}(\Omega)= \mathcal C^{1+\gamma}(\Omega) = C^{1+\gamma}(\Omega)\subseteq C^1(\Omega), \]
where the last inclusion is a continuous embedding. %
	
Combining the preceding bounds and using Lemma \ref{lem:entropy}, it follows that for all $R>0$ and some constants $0<C_1,C_2,C_3<\infty$,
	\begin{align*}
		\int_0^R H^{1/2}(\mathcal F(r), C^1(\Omega), \tau) d\tau &\le     \int_0^R H^{1/2}(\mathcal F(r), B^{1+\gamma}_{\infty\infty}(\Omega), C_1\tau) d\tau \\
		& \le \int_0^R H^{1/2}(A^k_{\infty\infty}(C_2r), B^{1+\gamma}_{\infty\infty}(\Omega), C_1\tau) d\tau\\
		& \le C_3\int_0^R \Big( \frac{C_2r}{C_1\tau}\Big)^{\frac{d+1}{2(k-1-\gamma)}}d\tau \\
	&\lesssim R^{1-\frac{d+1}{2(k-1-\gamma)}}. \\
	\end{align*}
	Thus, the requirement (\ref{entropy-req}) from Theorem \ref{thm:general} simplifies to
	\[ \sqrt n \delta_n^2\gtrsim \delta_n + \delta_n^{1-\frac{d+1}{2(k-1-\gamma)}}. \]
	This is satisfied if both $\delta_n \gtrsim n^{-1/2}$ as well as \[ \sqrt n \gtrsim \delta_n^{-1-\frac{d+1}{2(k-1-\gamma)}}, \qquad \text{which is equivalent to}\qquad \delta_n \gtrsim n^{-\frac{k-1-\gamma}{2(k-1-\gamma)+d+1}}.  \]
	The desired result now follows directly from Theorem \ref{thm:general}. \hfill\qedsymbol

        \section{Auxiliary results for Section
          \ref{sec:NN-metric-entropy}}\label{app:RemainingProofsNN-theory}
        In this appendix, we prove some auxiliary results about the
        uniform boundedness and Lipschitz properties of the ReLU$^2$
        neural network class $\Phi^{d_1,1}(L, W, S, B)$ and its
        gradient space $\nabla\Phi^{d_1,1}(L, W, S, B)$, which will be
        used in the proof of Theorem \ref{thm:C1CoveringNN}. Our arguments are similar to those in \cite{Schmidt_Hieber_2020}, \cite{Suzuki19}, and \cite{MLPDEStatisticalRate} with two key differences: (i) to ensure smoothness of the gradient space, we consider ReLU$^2$ networks, whereas \cite{Schmidt_Hieber_2020} and \cite{Suzuki19} consider ReLU networks, and \cite{MLPDEStatisticalRate} considers ReLU$^3$ networks; and (ii) to obtain the $C^1$ metric entropy rate in Theorem \ref{thm:C1CoveringNN}, we construct a covering of both the NN function space and its gradient space.

\begin{lemma}\label{lemma:LInfinityDNNspace}
  For any $1\leq l\leq L$, the following inequality holds for the
  class of ReLU$^2$ networks $\Phi^{d_1, 1}(L, W, S, B)$:
	$$\sup_{x\in D, F_l\in\Phi^{d_1,1}_l(L, W, S, B)}\|F_l(x)\|_\infty \leq C_lW^{2^{l-1}-1}(B\lor d_1)^{2^l-1},  $$
	where $C_l$ is a constant independent of $W$, $B$, $d_1$, depending only
         on $l$.
      \end{lemma}
      \begin{proof}[Proof of Lemma \ref{lemma:LInfinityDNNspace}]
        We prove the lemma by induction. First note for any matrix
        $A\in\mathbb{R}^{d\times d}$, $\|A\|_{\infty, \infty} \leq B$
        implies $\|A\|_\infty\leq dB$. When $l = 1$, we have for all
        $x\in D$,
$$\|F_1(x)\|_\infty = \|W^{(1)}_Fx + b_F^{(1)}\|_\infty\leq \|W^{(1)}_F\|_\infty\|x\|_\infty + \|b_F^{(1)}\|_\infty \leq d_1B + B \leq 2(B\lor d_1)^2.$$

Assuming the claim holds for $l-1$, where $l \geq 2$, we have that
\begin{align*}
  &\|F_l(x)\|_\infty = \|W^{(l)}_F\eta_2(F_{l-1}(x)) + b_F^{(l)}\|_\infty \leq WB\|F_{l-1}(x)\|_\infty^2 + B\\
  &\leq W(B\lor d_1)\left(C_{l-1}W^{2^{l-2}-1}(B\lor d_1)^{2^{l-1}-1}\right)^2 + B\\
  &\leq C_{l-1}^2W^{2^{l-1} - 2 + 1}(B\lor d_1)^{2^l - 2 + 1} + (B\lor d_1)\\
  &\leq (C_{l-1}^2+1)W^{2^{l-1}-1}(B\lor d_1)^{2^l-1} = C_lW^{2^{l-1}-1}(B\lor d_1)^{2^l-1}.
\end{align*}
Hence the claim follows from induction.
\end{proof}

\begin{lemma}\label{RelationBetweenCoveringNumberofDNNandParameterSpace}
  For any $1 \leq l \leq L$, suppose that a pair of two different
  ReLU$^2$ networks $F_l, G_l\in \Phi^{d_1,1}_l(L, W, S, B)$ are given
  by
	$$F_l(x) =  (W^{(l)}_F\eta_2(\cdot) + b^{(l)}_F)\circ\cdots\circ(W^{(1)}_F\eta_2(\cdot) + b^{(1)}_F), $$
	$$G_l(x) = (W^{(l)}_G\eta_2(\cdot) + b^{(l)}_G)\circ\cdots\circ(W^{(1)}_G\eta_2(\cdot) + b^{(1)}_G). $$
        Assume that the $l_\infty$ norm between the neural network
        weights is uniformly upper bounded by $\delta$, i.e.,
        $\|W_F^{(l')} - W_G^{(l')}\|_{\infty,\infty} \leq \delta,
        \|b_F^{(l')} - b_G^{(l')}\|_\infty \leq \delta$, for all
        $1\leq l' \leq l$.  Then we have
	$$\sup_{x\in D}\|F_l(x) - G_l(x)\|_\infty \leq A_l\delta W^{2^{l-1}-1}(B\lor d_1)^{2^l},$$
	for some constant $A_l$ that only depends on $l$.
      \end{lemma}
      \begin{proof}[Proof of Lemma
        \ref{RelationBetweenCoveringNumberofDNNandParameterSpace}]
	We prove the lemma by induction. For any $x\in D$ and
        $F_1, G_1\in \Phi^{d_1,1}_1(L, S, W, B)$, it holds that
	\begin{align*}
          \|F_1(x) - G_1(x)\|_\infty &= \|W_F^{(1)}x + b_F^{(1)} - W_G^{(1)}x - b_G^{(1)}\|_\infty
          \\
                                     &\leq \|W_F^{(1)} - W_G^{(1)}\|_\infty\|x\|_\infty + \|b_F^{(1)} - b_G^{(1)}\|_\infty\\
                                     &\leq \delta d + \delta  = \delta(d_1+1)\leq 2\delta (B\lor d_1)\leq 2\delta (B\lor d_1)^2.
	\end{align*}

	Now suppose the claim holds for $l-1$. %
        For the induction step, we will use that $\eta_2(x)=x^2$
        satisfies $|\eta_2(x)-\eta_2(y)|\le 2\max\{|x|,|y|\}|x-y|$. Thus, for any $x\in D$ and
        $F_l, G_l\in\Phi^{d_1, 1}_l(L, W, S, B)$, we have
	\begin{align*}
          \|F_l(x) &- G_l(x)\|_\infty =  \|W_F^{(l)}\eta_2(F_{l-1}(x)) + b_F^{(l)} - W_G^{(l)}\eta_2(G_{l-1}(x)) - b_G^{(l)}\|_\infty\\  
                   &\leq \|W_F^{(l)}\eta_2(F_{l-1}(x)) - W_G^{(l)}\eta_2(G_{l-1}(x))\|_\infty + \|b_F^{(l)} - b_G^{(l)}\|_\infty\\
                   &\leq \|W_F^{(l)}\eta_2(F_{l-1}(x)) -W_G^{(l)}\eta_2(F_{l-1}(x))\|_\infty \\
                   &+ \|W_G^{(l)}\eta_2(F_{l-1}(x)) - W_G^{(l)}\eta_2(G_{l-1}(x))\|_\infty + \delta\\
                   &\leq \|W_F^{(l)} - W_G^{(l)}\|_\infty\|\eta_2(F_{l-1}(x))\|_\infty\\
                   &+ \|W_G^{(l)}\|_\infty\|\eta_2(F_{l-1}(x))-\eta_2(G_{l-1}(x))\|_\infty + \delta\\
                   &\leq W\delta\|F_{l-1}(x)\|_\infty^2\\
                   &+ WB(2\sup_{F_{l-1}\in \Phi^{d_1,1}_{l-1}(L, W, S, B)}\|F_{l-1}(x)\|_\infty)\|F_{l-1}(x) - G_{l-1}(x)\|_\infty + \delta\\
                   &\leq W\delta (C_{l-1}W^{2^{l-2}-1}(B\lor d_1)^{2^{l-1}-1})^2 \\
                   &+ 2WB(C_{l-1}W^{2^{l-2}-1}(B\lor d_1)^{2^{l-1}-1})(A_{l-1}\delta W^{2^{l-2}-1}(B\lor d_1)^{2^{l-1}}) + \delta\\
                   &\leq \delta C_{l-1}^2W^{2^{l-1}-1} (B\lor d_1)^{2^l -2}\\
                   &+2\delta C_{l-1}A_{l-1}W^{2^{l-1}-1}(B\lor d_1)^{1+2^{l-1}-1 + 2^{l-1}} + \delta\\
                   &\leq \delta W^{2^{l-1}-1}\left(C_{l-1}^2 (B\lor d_1)^{2^l -2} + 2C_{l-1}A_{l-1}(B\lor d_1)^{2^l} + 1\right)\\
                   &\leq A_l\delta W^{2^{l-1}-1}(B\lor d_1)^{2^l},
	\end{align*}
	for some constant $A_l$ that only depends on $l$.  Hence the
        claim follows from induction.

      \end{proof}

      \begin{lemma}\label{lemma:InftyNormGraident} For any
        $1\leq l \leq L$, the following inequality holds for the class
        of ReLU$^2$ networks:
	$$\sup_{x\in D, F_l\in\Phi^{d_1,1}_l(L, W, S,B)} \|\nabla F_l(x)\|_\infty \leq  M_kW^{2^{l-1}-1}(B\lor d_1)^{2^l},  $$
	for some constant $M_l$ that only depends on $l$.
      \end{lemma}
      \begin{proof}
	We prove the lemma by induction. When $l = 1$, it holds that
        for all $x\in D$
	$$\|\nabla F_1(x)\|_{\infty} \leq \|W_F^{(1)}\|_{\infty} \leq d_1B \leq (B\lor d_1)^2.$$
	
	Suppose the claim holds for $l-1$. Then, we may compute
	\begin{align*}
          \|\nabla F_l(x)\|_\infty &= \|W_F^{(l)}\nabla[\eta_2\circ F_{l-1}](x)\|_\infty \leq \|W_F^{(l)}\|_\infty\|\nabla[\eta_2\circ F_{l-1}](x)\|_\infty\\
                                   &\leq WB\|\nabla[\eta_2\circ F_{l-1}](x)\|_\infty\leq W(B\lor d_1)\|\nabla[\eta_2\circ F_{l-1}](x)\|_\infty.
	\end{align*}
	
	Since the operator $\infty$-norm of a matrix equals the
        maximum row sum, we have
	\begin{align*}
          &\|\nabla[\eta_2\circ F_{l-1}](x)\|_\infty = \sup_{1\leq j \leq W}\sum_{i=1}^{d_1}|\eta_2'(F_{l-1, j}(x))\frac{\partial F_{l-1, j}}{\partial x_i}|\\
          &\leq 2\|F_{l-1}\|_\infty\sup_{1\leq j \leq W}\sum_{i=1}^{d_1}|\frac{\partial F_{l-1, j}}{\partial x_i}|\\
          &\leq 2C_{l-1}W^{2^{l-2}-1}(B \lor d_1)^{2^{l-1}-1}\|\nabla F_{l-1}(x)\|_\infty. 
	\end{align*}
	
	Then, we get
	\begin{align*}
          &\|\nabla F_l(x)\|_\infty \leq  W(B\lor d_1)\|\nabla[\eta_2\circ F_{l-1}](x)\|_\infty\\ &\leq W(B\lor d_1)2C_{l-1}W^{2^{l-2}-1}(B \lor d_1)^{2^{l-1}-1}\|\nabla F_{l-1}(x)\|_\infty\\
          &\leq W(B \lor d_1)2C_{l-1}W^{2^{l-2}-1}(B \lor d_1)^{2^{l-1}-1}W^{2^{l-2}-1}(B\lor d_1)^{2^{l-1}}\\
          &\leq M_lW^{2^{l-1}-1}(B\lor d_1)^{2^l},
	\end{align*}
	if we absorb all the constants into $M_l$. The claim then
        follows from induction.
      \end{proof}

\begin{lemma}\label{lemma:CoveringRelationParameterAndGradient} For
  any $1 \leq l \leq L$, suppose that a pair of two different ReLU$^2$
  networks $F_l, G_l\in \Phi^{d_1,1}_l(L, W, S, B)$ are given by
	$$F_l(x) =  (W^{(l)}_F\eta(\cdot) + b^{(l)}_F)\circ\cdots\circ(W^{(1)}_F\eta(\cdot) + b^{(1)}_F), $$
	$$G_l(x) = (W^{(l)}_G\eta(\cdot) + b^{(l)}_G)\circ\cdots\circ(W^{(1)}_G\eta(\cdot) + b^{(1)}_G). $$
        Assume that the $l_\infty$ norm between the neural network
        weights is uniformly upper bounded by $\delta$, i.e.,
        $\|W_F^{(l')} - W_G^{(l')}\|_{\infty, \infty} \leq \delta,
        \|b_F^{(l')} - b_G^{(l')}\|_\infty \leq \delta, 1\leq l' \leq
        l$.  Then we have
	$$\sup_{x\in D}\|\nabla F_l(x) - \nabla G_l(x)\|_\infty \leq   \delta N_{l}W^{2^{l-1}-1}(B\lor d_1)^{2^l+1},$$
	where $N_l$ is a constant the only depends on $l$.

      \end{lemma}
      \begin{proof}[Proof of Lemma
        \ref{lemma:CoveringRelationParameterAndGradient}]
	We prove this lemma by induction. When $l = 1$, it holds that
        for all $x\in D$,
	$$\|\nabla F_1(x) - \nabla G_1(x)\|_\infty = \|W_F^{(1)} - W_G^{(1)}\|_\infty\leq \delta d_1 \leq \delta (B\lor d_1) \leq \delta (B\lor d_1)^3.  $$
	
        Assume that the claim holds for $l-1$. Then, for any $x\in D$,
        and $F_l, G_l\in\Phi^{d_1, 1}_l(L, S, W, B)$ satisfying the
        conditions in the lemma, we can bound
        $\|\nabla F_l(x) - \nabla G_l(x)\|_\infty$ using the chain
        rule and triangular inequality as follows:
	\begin{align*}
          &\|\nabla F_l(x) - \nabla G_l(x)\|_\infty = \|W_F^{(l)}\nabla[\eta_2\circ F_{l-1}](x) -  W_G^{(l)}\nabla[\eta_2\circ G_{l-1}](x)\|_\infty\\
          &\leq \|W_F^{(l)}\nabla[\eta_2\circ F_{l-1}](x) -  W_G^{(l)}\nabla[\eta_2\circ F_{l-1}](x)\|_\infty\\
          &+ \|W_G^{(l)}\nabla[\eta_2\circ F_{l-1}](x) - W_G^{(l)}\nabla[\eta_2\circ G_{l-1}](x)\|_\infty\\
          &\leq \|W_F^{(l)} - W_G^{(l)}\|_\infty\|\nabla[\eta_2\circ F_{l-1}](x)\|_\infty\\
          &+ \|W_G^{(l)}\|_\infty\|\nabla[\eta_2\circ F_{l-1}](x) - \nabla[\eta_2\circ G_{l-1}](x)\|_\infty\\
          &\leq \delta W\|\nabla[\eta_2\circ F_{l-1}](x)\|_\infty + BW\|\nabla[\eta_2\circ F_{l-1}](x) - \nabla[\eta_2\circ G_{l-1}](x)\|_\infty\\
          &:= I + II. 
	\end{align*}
	
	From Lemma \ref{lemma:LInfinityDNNspace} and
        \ref{lemma:InftyNormGraident}, we can bound $I$ by
	\begin{align*}
          I &= \delta W\|\nabla[\eta_2\circ F_{l-1}](x)\|_\infty \leq  \delta W2\|F_{l-1}\|_\infty\|\nabla[F_{l-1}](x)\|_\infty\\
            &\leq2\delta W C_{l-1}W^{2^{l-2}-1}(B\lor d_1)^{2^{l-1}-1} M_{l-1}W^{2^{l-2}-1}(B\lor d_1)^{2^{l-1}}\\
            &= 2C_{l-1}M_{l-1}\delta W^{2^{l-1}-1}(B\lor d_1)^{2^l-1}
	\end{align*}

	To bound $II$, note that
        \begin{align*}
          II & = BW\|\nabla[\eta_2\circ F_{l-1}](x) - \nabla[\eta_2\circ G_{l-1}](x)\|_\infty\\
             & = BW\sup_{1\leq j\leq W^{(l-1)}} (\sum_{i=1}^{d_1}|\eta_2'(F_{l-1, j})\frac{\partial F_{l-1, j}}{\partial x_i} - \eta_2'(G_{l-1, j})\frac{\partial G_{l-1, j}}{\partial x_i}|),\\
        \end{align*}
	and it holds that for all $j, 1\leq j\leq W^{(l-1)}$,
	\begin{align*}
          & \sum_{i=1}^{d_1}|\eta_2'(F_{l-1, j})\frac{\partial F_{l-1, j}}{\partial x_i} - \eta_2'(G_{l-1, j})\frac{\partial G_{l-1, j}}{\partial x_i}|\\
          &\leq \sum_{i=1}^{d_1}|\eta_2'(F_{l-1, j})\frac{\partial F_{l-1, j}}{\partial x_i} - \eta_2'(G_{l-1, j})\frac{\partial F_{l-1, j}}{\partial x_i}| \\
          &+ \sum_{i=1}^{d_1}|\eta_2'(G_{l-1, j})\frac{\partial F_{l-1, j}}{\partial x_i} - \eta_2'(G_{l-1, j})\frac{\partial G_{l-1, j}}{\partial x_i}| := III + IV.
	\end{align*}
	
	$III$ can be bounded as follows:
	\begin{align*}
          III &= \sum_{i=1}^{d_1}|\eta_2'(F_{l-1, j})\frac{\partial F_{l-1, j}}{\partial x_i} - \eta_2'(G_{l-1, j})\frac{\partial F_{l-1, j}}{\partial x_i}| \\
              &\leq \sum_{i=1}^{d_1}|\eta_2'(F_{l-1, j}) - \eta_2'(G_{l-1, j})||\frac{\partial F_{l-1, j}}{\partial x_i}|\\
              &\leq 2\|F_{l-1} - G_{l-1}\|_\infty\sum_{i=1}^{d_1} |\frac{\partial F_{l-1, j}}{\partial x_i}|\leq 2\|F_{l-1} - G_{l-1}\|_\infty\|\nabla[F_{l-1}](x)\|_\infty  \\
              &\leq A_{l-1}\delta W^{2^{l-2}-1}(B\lor d_1)^{2^{l-1}}M_{l-1}W^{2^{l-2}-1}(B\lor d_1)^{2^{l-1}}\\
              &= \delta A_{l-1}M_{l-1}W^{2^{l-1}-2}(B\lor d_1)^{2^l},
	\end{align*}
	where the last inequality follows from Lemma
        \ref{RelationBetweenCoveringNumberofDNNandParameterSpace} and
        \ref{lemma:InftyNormGraident}.
	
        Applying the inductive hypothesis and Lemma
        \ref{lemma:LInfinityDNNspace}, $IV$ can be bounded as follows:
	\begin{align*}
          IV = &\sum_{i=1}^{d_1}|\eta_2'(G_{l-1, j})\frac{\partial F_{l-1, j}}{\partial x_i} - \eta_2'(G_{l-1, j})\frac{\partial G_{l-1, j}}{\partial x_i}| \\
               &\leq    2\sup_{x\in D}\|G_{l-1}(x)\|_\infty\sum_{i=1}^{d_1}|\frac{\partial F_{l-1, j}}{\partial x_i} - \frac{\partial G_{l-1, j}}{\partial x_i}|\\
               &\leq 2\sup_{x\in D}\|G_{l-1}(x)\|_\infty\|\nabla[F_{l-1}](x) - \nabla[G_{l-1}](x)\|_\infty\\
               &\leq 2C_{l-1}W^{2^{l-2}-1}(B\lor d_1)^{2^{l-1}-1}(\delta N_{l-1}W^{2^{l-2}-1}(B\lor d_1)^{2^{l-1}+1})
	\end{align*}
	
	Putting everything together,
        $\|\nabla F_l(x) - \nabla G_l(x)\|_\infty$ is upper bounded
        by:
	\begin{align*}
          & \delta 2C_{l-1}M_{l-1}W^{2^{l-1}-1}(B\lor d_1)^{2^l-1} + \delta BA_{l-1}M_{l-1}W^{2^{l-1}-1}(B \lor d_1)^{2^l} +\\
          &BW(2C_{l-1}W^{2^{l-2}-1}(B\lor d_1)^{2^{l-1}-1}(\delta N_{l-1}W^{2^{l-2}-1}(B\lor d_1)^{2^{l-1}+1}))\\
          &\leq 2\delta C_{l-1}M_{l-1}W^{2^{l-1}-1}(B\lor d_1)^{2^l-1}
            + \delta A_{l-1}M_{l-1}W^{2^{l-1}-1}(B \lor d_1)^{2^l+1}\\
          &+ 2\delta C_{l-1}N_{l-1}W^{2^{l-1}-1}(B\lor d_1)^{2^l+1}\leq \delta N_{l}W^{2^{l-1}-1}(B\lor d_1)^{2^l+1},\\
	\end{align*}
	where $N_l$ is a constant that only depends on $l$. Hence the
        claim follows by induction.

      \end{proof}

\section{Neural network approximation theory}\label{app:NNapproximation}

In the following we work with the standard normalized one-dimensional
B-spline of order $m\ge 1$ with equidistant knots, see e.g., 
\cite[(4.46)--(4.47)]{schumaker}:
	\begin{subequations}\label{eq:spline}
		\begin{equation}
			B^m(x):=\sum_{i=0}^m (-1)^i \frac{\binom{m}{i}\max\{0,x-i\}^{m-1}}{(m-1)!}\in W^{m-1,\infty}(\R)
		\end{equation}
		where $0^0:=0$. Additionally, for
		$n\in\N$ we\footnote{In \Cref{app:NNapproximation} only, $n$ denotes a stretching parameter rather than the sample size in the maximum likelihood estimation problem \eqref{eq:MLEobjective}.} consider the stretched and shifted versions
		\cite[(4.49)]{schumaker}
		\begin{equation}\label{eq:Bnjm}
			B_{n,j}^m(x):=B^m(nx-j)\in W^{m-1,\infty}(\R),\qquad j\in\mathbb{Z}.
		\end{equation}
	\end{subequations}
	Note that $B_{n,j}^m|_{[0,1]}\in C^{m-2}([0,1])$ is a piecewise
	polynomial of degree $m-1$ on the intervals $[\frac{j}{n},\frac{j+1}{n}]$, and thus the function is $C^\infty$ on 
	\begin{equation*}
		M_n:=[0,1]\backslash \setc{\frac{j}{n}}{1\le j\le n-1}.
	\end{equation*}
	Moreover
	${\rm supp}(B_{n,j}^m)\subseteq [\frac{j}{n},\frac{m+j}{n}]$.
	
	\subsection{One dimensional spline approximation}
	It is well-known that one can construct continuous linear
	functionals $\lambda_{n,j}^m:C([0,1])\to\R$
	such that
	\begin{equation}\label{eq:Qnm}
		Q_n^m[f](x):=\sum_{j=-m+1}^{n-1}\lambda_{n,j}^m[f]B_{n,j}^m(x)
	\end{equation}
	yields an approximation to $f$ that converges at a rate depending on
	the regularity $k\in\N$ of the target function $f\in C^k([0,1])$ as
	long as the the order $m\in\N$ of the spline is larger or equal to
	$k+1$.\footnote{Here and throughout \Cref{app:NNapproximation}, we always interpret continuity of a
		functional from $C([0,1]^d)\to\R$ w.r.t.\ to the topology of
		pointwise convergence on $C([0,1]^d)$.}  While various
	approximation results for Sobolev or Besov spaces have been
	established in the literature, e.g., \cite{MR1050615}, for our
	purposes approximation of $C^k$ functions as stated in the following
	variant\footnote{The main difference to the presentation in
		\cite{schumaker} is our treatment of the boundary, which avoids the
		use of different spline basis functions near the endpoints $0$ and
		$1$ of the interval.}  of \cite[Theorem 6.20]{schumaker} is
	sufficient:
	
	\begin{theorem}\label{thm:spline1d}
		Let $k\in\N_0$, $m\in\N$ and $k+1\le m$. %
		Then there exists $C=C(k,m)$ such that for every $n\in\N$, there
		exist continuous (w.r.t.\ the topology of pointwise convergence)
		linear functionals $\lambda_{n,j}^m:C([0,1])\to\R$,
		$j\in\{-m+1,\dots,n-1\}$, such that
		\begin{enumerate}
			\item\label{item:lambdacont} for all $n\in\N$, $j\in\{-m+1,\dots,n-1\}$, $f\in C([0,1])$
			\begin{equation}\label{eq:lambdanjbound}
				|\lambda_{n,j}^m[f]|\le C \norm[{C([0,1])}]{f},
			\end{equation}
			\item for all $r\in\{0,\dots,k\}$, $f\in C^k([0,1])$
			and with $Q_n^{m}$ as in \eqref{eq:Qnm}
			\begin{equation}\label{eq:err_spline1d}
				\sup_{x\in M_n}\Big|\frac{d^r}{dx^r}(f - Q^m_n[f])\Big|
				\le C  n^{-(k-r)}
				\snorm[{C^k([0,1])}]{f}.
			\end{equation}
		\end{enumerate}
	\end{theorem}
	\begin{proof}
		We proceed in three steps: In step 1 we show an extension result for
		functions in $C^k([0,1])$, in step 2 we verify the error bound
		\eqref{eq:err_spline1d} and in step 3 we show continuity of the
		$\lambda_{n,j}^m$ and \eqref{eq:lambdanjbound}.
		
		{\bf Step 1.} Using standard techniques, we wish to define a
		bounded linear extension operator $E:C([0,1])\to C([-m,1+m])$
		that additionally is stable between
		$C^r([0,1])\to C^r([-m,1+m])$ for each $r\in\{0,\dots,k\}$.

		Fix distinct numbers $-\frac{1}{m}<\gamma_0<\dots <\gamma_k<0$ and
		let $g\in C^k([0,1])$. Set $\tilde g(x):=g(x)$ if $x\in [0,1]$ and
		\begin{equation}\label{eq:tildeg}
			\tilde g(x):=\sum_{j=0}^k\alpha_j g(\gamma_j x)\qquad\forall x\in [-m,0],
		\end{equation}
		for certain $\alpha_j\in\R$ that remain to be determined. It holds
		$\tilde g\in C^k([-m,1])$ iff $\tilde g^{(r)}(0)= g^{(r)}(0)$ for
		all $r\in\{0,\dots,k\}$, i.e.,
		\begin{equation*}
			g^{(r)}(0) =g^{(r)}(0)\sum_{j=0}^k\alpha_j \gamma_j^r 
			\qquad\forall r\in\{0,\dots,k\}.
		\end{equation*}
		This condition being satisfied for arbitrary $g\in C^k([0,1])$ is
		equivalent to
		\begin{equation}\label{eq:vandermonde}
			\begin{pmatrix}
				1 &1 &\cdots &1\\
				\gamma_0^1 &\gamma_1^1&\cdots &\gamma_k^1\\
				\vdots&&\ddots &\vdots\\
				\gamma_0^k &\gamma_1^k&\cdots &\gamma_k^k
			\end{pmatrix}
			\begin{pmatrix}
				\alpha_0\\
				\vdots\\
				\alpha_k
			\end{pmatrix}
			=
			\begin{pmatrix}
				1\\
				\vdots\\
				1
			\end{pmatrix}.
		\end{equation}
		Since the matrix on the left-hand side is a Vandermonde matrix
		with distinct nodes $\gamma_0,\dots,\gamma_k$, it is regular.
		Hence there exists a unique set of numbers
		$(\alpha_j)_{j=0}^k$ satisfying \eqref{eq:vandermonde}.
		
		In the same fashion $\tilde g(x)$ can be extended to $x\in [1,1+m]$.
		This yields a linear map $E:C([0,1])\to C([-m,1+m])$
		that evidently (cp.~\eqref{eq:tildeg}) satisfies
		\begin{equation}\label{eq:extension}
			\snorm[{C^r([-m,1+m])}]{Eg}\le C
			\snorm[{C^r([0,1])}]{g}\qquad\forall g\in C^k([0,1]),~\forall r\in\{0,\dots,k\},
		\end{equation}
		for some constant $C$ depending on $(\gamma_j)_{j=0}^k$ and
		$(\alpha_j)_{j=0}^k$ (and hence on $k$ and $m$) but independent of
		$g$.
		
		{\bf Step 2.} %
		According
		to \cite[Theorem 6.20]{schumaker}, there exist bounded linear
		functionals $\tilde\lambda_{n,j}^m:C([-m,m+1])\to\R$ such that for
		each $l\in\{0,\dots,n-1\}$ and $r\in\{0,\dots,k\}$ it
		holds\footnote{ In the notation of \cite[Theorem 6.20]{schumaker},
			we use equidistant knots ``$y_l:=\frac{l}{n}$'' for
			$l\in\{-mn,\dots,(1+m)n\}$ on the interval ``$[a,b]:=[-m,1+m]$''
			with ``$\sigma:=k+1$'' and ``$q:=\infty$''.}
  
  \begin{equation}\label{eq:err_spline1}
	\normc[{L^\infty([\frac{l}{n},\frac{l+1}{n}])}]{\frac{d^r}{dx^r}\left(f-\sum_{j=l-m+1}^l \tilde\lambda_{n,j}^m(Ef) B^m_{n,j}(x)\right)}\le C n^{-(k-r)}\omega\Big((Ef)^{(k)},\frac{1}{n}\Big)_{C{[\frac{l+1-m}{n},\frac{l+m}{n}]}},
		\end{equation}
		where $C=C(m)$ is independent of $f$, $l$ and $n$, and
		\begin{equation*}
			\omega\Big((Ef)^{(k)},\frac{1}{n}\Big)_{C{[\frac{l+1-m}{n},\frac{l+m}{n}]}}
			=\sup_{\substack{x,y\in [\frac{l+1-m}{n},\frac{l+m}{n}]\\ |x-y|\le \frac{1}{n}}}
			|(Ef)^{(k)}(x)-(Ef)^{(k)}(y)|
		\end{equation*}
		denotes the modulus of continuity for the $k$th derivative of $Ef$. %
		Using
		\eqref{eq:extension} this term can be bounded by
		$2\snorm[{C^k([-m,1+m])}]{Ef}\le 2C \snorm[{C^k([0,1])}]{f}$.
		
		With
		\begin{equation}\label{eq:lambdanj}
			\lambda_{n,j}^m[f]:=\tilde\lambda_{n,j}^m[Ef]\qquad j\in\{-m+1,\dots,n-1\}
		\end{equation}
		we obtain by \eqref{eq:Qnm}
		\begin{equation*}
			Q_n^m[f] = \sum_{j=-m+1}^{n-1}\lambda_{n,j}^m[f]B^m_{n,j}(x)
			=\sum_{j=-m+1}^{n-1}\tilde\lambda_{n,j}^m[Ef]B^m_{n,j}(x).
		\end{equation*}
		Since ${\rm supp}(B_{n,j}^m)\subseteq [\frac{j}{n},\frac{j+m}{n}]$
		as pointed out earlier,
		\eqref{eq:err_spline1} shows the error bound \eqref{eq:err_spline1d}
		on the interval $[\frac{l}{n},\frac{l+1}{n}]$. Because
		$l\in\{0,\dots,n-1\}$ was arbitrary, this shows
		\eqref{eq:err_spline1d}.
		
		{\bf Step 3.} It remains to argue continuity of $\lambda_{n,j}^m$ and
		the bound \eqref{eq:lambdanjbound}. By construction of
		$\tilde \lambda_{n,j}^m$, see\footnote{We use the notation %
			$\tilde\lambda_{n,i}^m$ for ``$\lambda_i$'' in \cite[Chapter
			6]{schumaker}.}  \cite[(6.39)]{schumaker}, for
		$j\in\{-m+1,\dots,n-1\}$ the term $\tilde\lambda_{n,j}^m[f]$ is a
		linear combination of finitely many point evaluations of $f$ in
		$[-m+1,m]$. Hence $\tilde\lambda_{n,j}^m:C([-m+1,m])\to\R$ is
		continuous w.r.t.\ the topology of pointwise convergence. Now
		suppose that $(g_i)_{i\in\N}\subseteq C([0,1])$ is a sequence of
		functions converging pointwise to $g\in C([0,1])$. Then the
		construction of $E$ (cp.~\eqref{eq:tildeg}) implies that
		$Eg_i\to Eg\in C([-m+1,m])$ pointwise, and thus by definition of
		$\lambda_{n,j}^m$ in \eqref{eq:lambdanj}
		\begin{equation*}
			\lambda_{n,j}^m[g] = \tilde\lambda_{n,j}^m[Eg_i]\to \tilde\lambda_{n,j}^m[Eg]\in\R\qquad\text{as }i\to\infty,
		\end{equation*}
		which shows the claimed continuity of $\lambda_{n,j}^m:C([0,1])\to\R$.
		
		Moreover, as shown in the proof of \cite[Theorem 6.22]{schumaker}
		\begin{equation*}
			|\tilde\lambda_{n,j}^m[g]|\le (2m)^m \norm[{C([-m,1+m])}]{g}\qquad\forall g\in C([-m,1+m]),
		\end{equation*}
		so that for any $f\in C([0,1])$
		\begin{equation*}
			|\lambda_{n,j}^m[f]|=|\tilde\lambda_{n,j}^m[Ef]|\le (2m)^m \norm[{C([-m,1+m])}]{Ef}
			\le C
			(2m)^m \norm[{C([0,1])}]{f}
		\end{equation*}
		for some $C$ depending on $k$ and $m$ but independent of $n$, $j$ and $f$.
	\end{proof}
	
	\subsection{Multidimensional spline approximation}
	We next extend \Cref{thm:spline1d} to the multidimensional case. In
	principle such a statement is provided in \cite[Theorem
	12.7]{schumaker}, however this result requires mixed regularity of the
	target function, which we wish to avoid.
	
	To give the statement, we first introduce some notation. Fix $m$,
	$n\in\N$.  With $\lambda_{n,j}^m:C([0,1])\to\R$ as in
	\Cref{thm:spline1d}, for $f\in C([0,1]^d)$ and a multiindex
	$\bnu=(\nu_1,\dots,\nu_d)\in\{-m+1,\dots,n-1\}$ define
	\begin{equation}\label{eq:blambda}
		\blambda_{n,\bnu}^m[f]:=\lambda_{n,\nu_d}^{m,x_d}\dots \lambda_{n,\nu_1}^{m,x_1}[f]\in\R.
	\end{equation}
	Here $\lambda_{n,\nu_i}^{m,x_i}:C([0,1])\to\R$ is understood to act on
	the $x_i$ variable only. Additionally with $B_{n,j}^m$ in \eqref{eq:Bnjm}
	\begin{equation*}
		\bB^m_{n,\bnu}(x_1,\dots,x_d):=\prod_{i=1}^d B_{n,\nu_i}^m(x_i),
	\end{equation*}
	and
	\begin{equation*}
		\bQ^m_n[f](x_1,\dots,x_d):=\sum_{-m+1\le \nu_1,\dots,\nu_d\le n-1} \blambda_{n,\bnu}^m[f]\bB_{n,\bnu}(x_1,\dots,x_d).
	\end{equation*}
	
	\begin{lemma}\label{lemma:blambda}
		Let $m$, $n\in\N$, $\bnu\in\{-m+1,\dots,n-1\}^d$ and
		$\nu\in \{-m+1,\dots,n-1\}$.
		\begin{enumerate}
			\item\label{item:multilambda} Equation \eqref{eq:blambda} defines a
			continuous (w.r.t.\ the topology of pointwise convergence) linear functional $\lambda_{n,\bnu}^m:C([0,1]^d)\to\R$.
			\item\label{item:multilambdabound} There exists $C=C(m,d)$
			independent of $n$ and $\bnu$ such that
			\begin{equation*}
				|\blambda_{n,\bnu}^m[f]|\le C\norm[{C([0,1]^d)}]{f}\qquad\forall f\in C([0,1]^d).
			\end{equation*}
			\item If $f\in C^{k}([0,1]^d)$ then for all $j\in\{1,\dots,d\}$ and $\balpha\in\N_0^d$
			with $|\balpha|\le k$ and $\alpha_j=0$ it holds
			\begin{equation}\label{eq:commutate}
				\partial_\bx^{\balpha}\lambda_{n,\nu}^{m,x_j}[f] =\lambda_{n,\nu}^{m,x_j}[\partial_\bx^{\balpha}f]\in C^{k-|\balpha|}([0,1]^{d-1}).
			\end{equation}
		\end{enumerate}
	\end{lemma}
\begin{proof}
		Throughout fix $\nu\in\{-m+1,\dots,n-1\}$,
		$\bnu\in\{-m+1,\dots,n-1\}^d$ and $f\in C([0,1]^d)$ arbitrary.
		
		We first show that $\lambda_{n,\nu}^{m,x_j}[f]\in C([0,1]^{d-1})$
		for each $j\in\{1,\dots,d\}$. This then implies that
		$\blambda_{n,\bnu}^m[f]\in\R$ in \eqref{eq:blambda} is
		well-defined. Wlog $j=1$. Let $\bx_i\in[0,1]^{d-1}$, $i\in\N$, be a
		sequence of points converging to $\bx^*\in [0,1]^{d-1}$. Then
		$g_i(x_1):=f(x_1,\bx_i)$, $i\in\N$, defines a sequence of functions
		in $C([0,1])$ converging pointwise to $g(x_1):=f(x_1,\bx^*)$.  By
		\Cref{thm:spline1d} we thus have
		$\lambda_{n,\nu}^{m,x_1}[g_i]\to \lambda_{n,\nu}^{m,x_1}[g]\in\R$ as
		$i\to\infty$, i.e.,
		$\lambda_{n,\nu}^{m,x_1}[f](\bx_i)\to\lambda_{n,\nu}^{m,x_1}[f](\bx^*)$
		as $i\to\infty$. This shows continuity of
		$\bx\mapsto \lambda_{n,\nu}^{m,x_1}[f](\bx)$ for $\bx\in [0,1]^d$.
		
		Next we claim that
		$\lambda_{n,\nu}^{m,x_j}:C([0,1]^d)\to C([0,1]^{d-1})$ is
		continuous w.r.t.\ the topologies of pointwise convergence on both
		spaces for all $j\in\{1,\dots,d\}$. This then immediately yields
		that $\blambda_{n,\bnu}^m:C([0,1]^d)\to\R$ in \eqref{eq:blambda}
		(obtained by repeated application of such operators) is
		continuous. Wlog $j=1$. Let $f_i\in C([0,1]^d)$, $i\in\N$, be a
		sequence of functions converging pointwise to $f\in C([0,1]^d)$
		and fix $\bx^*\in [0,1]^{d-1}$. Then $g_i(x_1):=f_i(x_1,\bx^*)$,
		$i\in\N$, is a sequence of functions in $C([0,1])$ that converges
		pointwise to $g(x_1):=f(x_1,\bx^*)\in C([0,1])$. Thus by
		\Cref{thm:spline1d}
		\begin{equation*}
			\lambda_{n,\nu}^{m,x_1}[f_i](\bx^*)
			=\lambda_{n,\nu}^{m,x_1}[g_i]\to
			\lambda_{n,\nu}^{m,x_1}[g]
			=
			\lambda_{n,\nu}^{m,x_1}[f](\bx^*)\qquad\text{as }i\to\infty,
		\end{equation*}
		which shows the claimed continuity and concludes the proof of
		\ref{item:multilambda}.
		
		Next, \ref{item:multilambdabound} follows directly by
		$d$ fold application of \eqref{eq:lambdanjbound} to the
		definition \eqref{eq:blambda} of $\blambda_{n,\bnu}^m$.
		
		Finally we show \eqref{eq:commutate} and assume $d\ge 2$. Wlog $j=1$.
		Fix $x_2\in [0,1]$ and $\bx^* \in [0,1]^{d-2}$. Then
		\begin{align*}
			\lambda_{n,\nu}^{m,x_1}[\partial_{x_2}f](x_2,\bx^*)
			&=
			\lambda_{n,\nu}^{m,x_1}\Big[\lim_{h\to 0}\frac{f(x_1,x_2+h,\bx^*)-f(x_1,x_2,\bx^*)}{h}\Big]\nonumber\\
			&=
			\lim_{h\to 0}\frac{\lambda_{n,\nu}^{m,x_1}[f](x_2+h,\bx^*)-\lambda_{n,\nu}^{m,x_1}[f](x_2,\bx^*)}{h}.
		\end{align*}
		The second equality follows by the fact that the difference quotient
		defines a family of pointwise convergent functions in $C([0,1]^d)$
		indexed over $h$, and the operator
		$\lambda_{n,\nu}^{m,x_1}:C([0,1]^d)\to C([0,1]^{d-1})$ is
		continuous w.r.t.\ the topology of pointwise convergence as shown
		above. Hence the last limit converges pointwise for all
		$(x_2,x^*)\in [0,1]^{d-1}$, which shows that
		$\lambda_{n,\nu}^{m,x_1}[f](x_2,\dots,x_d)$ is indeed differentiable
		in $x_2$ and the derivative in $x_2$ may be exchanged with
		$\lambda_{n,\nu}^{m,x_1}$. Repeatedly applying this argument yields
		the claim.
	\end{proof}
	
	\begin{theorem}\label{thm:splinemultid}
		Let $k\in\N_0$, $d$, $m\in\N$ and $k+1\le m$.
		Then there exists $C=C(d,k,m)$ such that
		for
		all $r\in \{0,\dots,k\}$,
		$\balpha\in\N_0^d$ with $|\balpha|=r$,
		 $f\in C^k([0,1]^d)$, and $n \geq 1$,
		\begin{equation}\label{eq:err_splinemultid}
			\sup_{\bx\in M_n^d}|\partial_\bx^{\balpha}(f(\bx) - \bQ^m_n[f](\bx))|\le C  n^{-(k-r)}
			\snorm[{C^k([0,1]^d)}]{f}.
		\end{equation}
	\end{theorem}
	\begin{proof}
		In the following we use the notation
		\begin{equation*}
			Q_{n}^{m,x_j}[f]:=\sum_{j=-m+1}^{n-1}\lambda_{n,j}^{m,x_j}[f],
		\end{equation*}
		so that
		\begin{equation}\label{eq:bQQ}
			\bQ_n^m[f] = Q_{n}^{m,x_d}\dots Q_{n}^{m,x_1}[f].
		\end{equation}
		In this proof we will use the following facts:
		\begin{itemize}
			\item By Lemma \ref{lemma:blambda}, for any
			$\balpha\in\N_0^d$ with $|\balpha|\le k$ and $\alpha_j=0$
			holds
			\begin{equation}\label{eq:commutingproperty}
				\partial_\bx^\balpha Q_n^{m,x_j}[f]=
				\sum_{i=-m+1}^{n-1}\partial_\bx^\balpha\lambda_{n,i}^{m,x_j}[f]
				B_{n,i}^m(x_j)
				=
				\sum_{i=-m+1}^{n-1}\lambda_{n,i}^{m,x_j}[\partial_\bx^\balpha f]
				B_{n,i}^m(x_j)
				= Q_n^{m,x_j}[\partial_\bx^\balpha f],
			\end{equation}
			i.e.,  $Q_n^{m,x_j}$ commutes with $\partial_\bx^{\balpha}$.
			\item %
			From \eqref{eq:err_spline1d} (with ``$k=r$'') we conclude
			that for any $g\in C^{\alpha_j}([0,1]^d)$ and
			$0\le \alpha_j\le m-1$
			\begin{equation}\label{eq:multi_stability}
				\sup_{x_j\in M_n}|\partial_{x_j}^{\alpha_j}Q_n^{m,x_j}[g](x_1,\dots,x_d)|\le C \sup_{x_j\in [0,1]}|\partial_{x_j}^{\alpha_j}g(x_1,\dots,x_d)|
			\end{equation}
			where $x_i\in [0,1]$ is arbitrary for all $i\neq j$, and $C=C(m,d)$
			is independent of $g$.
			\item Again by \eqref{eq:err_spline1d}, for $g\in C^r([0,1]^d)$,
			$0\le \alpha_j\le r\le m-1$ and $x_i\in [0,1]$ arbitrary for all $i\neq j$,
			\begin{equation}\label{eq:multi_convergence}
				\sup_{x_j\in M_n}|\partial_{x_j}^{\alpha_j}(Q_n^{m,x_j}[g](x_1,\dots,x_d)-g(x_1,\dots,x_d))|\le
				C n^{-(r-\alpha_j)}\sup_{x_j\in [0,1]}|\partial_{x_j}^rg(x_1,\dots,x_d)|.
			\end{equation}
		\end{itemize}
		
		Now fix $\balpha\in\N_0^d$ with $|\balpha|\le k$. Then for
		any $\bx=(x_1,\dots,x_d)\in M_n^d$ (cp.~\eqref{eq:bQQ})
		\begin{equation}\label{eq:multid_triangle}
			|\partial_{\bx}^{\balpha}(f(\bx)-\bQ_n^{m}[f](\bx))|
			\le \sum_{j=1}^{d}\big|\partial_{\bx}^{\balpha}(Q_{n}^{m,x_d}\dots Q_{n}^{m,x_{j+1}}[f](\bx)-Q_n^{m,x_d}\dots
			Q_{n}^{m,x_{j}}[f](\bx))
			\big|,
		\end{equation}
		where for $j=d$ the term
		$Q_{n}^{m,x_d}\dots Q_{n}^{m,x_{j+1}}[f](\bx)$ is understood
		as $f(\bx)$. Fix $j\in\{1,\dots,d\}$
		and denote
		\begin{equation*}
			\balpha_-:=(\alpha_1,\dots,\alpha_{j-1},0,\dots,0)^\top
			\qquad\text{and}\qquad
			\balpha_+:=(0,\dots,0,\alpha_{j+1},\dots,\alpha_{d})^\top.
		\end{equation*}
		With \eqref{eq:commutingproperty} and
		\eqref{eq:multi_stability} we get
		\begin{align*}
			&\Big|\partial_{\bx}^{\balpha}\big(Q_{n}^{m,x_d}\dots Q_{n}^{m,x_{j+1}}[f](\bx)-Q_n^{m,x_d}\dots
			Q_{n}^{m,x_{j}}[f](\bx)\big)\Big|\\
			&\qquad=|\partial_{x_d}^{\alpha_d} Q_{n}^{m,x_d}\cdots
			\partial_{x_{j+1}}^{\alpha_{j+1}} Q_{n}^{m,x_{j+1}}[\partial_{x_j}^{\alpha_j}\partial_{\bx}^{\balpha_-}f - \partial_{x_j}^{\alpha_j}Q_{n}^{m,x_{j}}[\partial_{\bx}^{\balpha_-}f]](\bx)|\nonumber\\
			&\qquad\le C \sup_{x_d\in [0,1]}|\partial_{x_d}^{\alpha_d}\partial_{x_{d-1}}^{\alpha_{d-1}} Q_{n}^{m,x_d}\cdots
			\partial_{x_{j+1}}^{\alpha_{j+1}} Q_{n}^{m,x_{j+1}}[\partial_{x_j}^{\alpha_j}\partial_{\bx}^{\balpha_-}f - \partial_{x_j}^{\alpha_j}Q_{n}^{m,x_{j}}[\partial_{\bx}^{\balpha_-}f]](\bx)|\nonumber\\
			&\qquad\le \cdots\le C \sup_{x_d,\dots,x_{j+1}\in [0,1]}|\partial_{\bx}^{\balpha_+}(\partial_{x_j}^{\alpha_j}\partial_{\bx}^{\balpha_-}f - \partial_{x_j}^{\alpha_j}Q_{n}^{m,x_{j}}[\partial_{\bx}^{\balpha_-}f])(\bx)|\nonumber\\
			&\qquad = C \sup_{x_d,\dots,x_{j+1}\in [0,1]}|\partial_{x_j}^{\alpha_j}\partial_{\bx}^{\balpha_+}\partial_{\bx}^{\balpha_-}f(\bx) - \partial_{x_j}^{\alpha_j}Q_{n}^{m,x_{j}}[\partial_{\bx}^{\balpha_+}\partial_{\bx}^{\balpha_-}f](\bx)|.
		\end{align*}
		By assumption $\partial_{\bx}^{\balpha_+}\partial_{\bx}^{\balpha_-}f\in C^{k-|\balpha_++\balpha_-|}([0,1]^d)$
		and thus by \eqref{eq:multi_convergence} the last term is bounded by
		\begin{equation*}
			C n^{-(k-|\balpha|)}\sup_{x_d,\dots,x_{j}\in [0,1]}|\partial_{x_j}^{k-|\balpha_++\balpha_-|}\partial_{\bx}^{\balpha_+}\partial_{\bx}^{\balpha_-}f(\bx)|.
		\end{equation*}
		Applying this estimate to \eqref{eq:multid_triangle} and taking the supremum over $\bx\in M_n^d$ we find for any $\balpha\in\N_0^d$ with $|\balpha|\le k$
		\begin{equation*}
			\sup_{\bx\in M_n^d}|\partial_\bx^{\balpha}(f(\bx)-\bQ_n^m[f](\bx))|
			\le C n^{-(k-|\balpha|)}\snorm[{C^k([0,1]^d)}]{f}
		\end{equation*}
		for some $C=C(k,m,d)$ as claimed.
	\end{proof}

\subsection{Translating spline approximation to neural networks}
	\begin{proof}[Proof of theorem \ref{thm:NNapproximation}]
		We wish to express the function
		\begin{equation}\label{eq:toexpress}
			\tilde f:=\bQ_{n}^m[f]=\sum_{-m+1\le\nu_1,\dots,\nu_{d_1}\le n-1} \blambda_{n,\bnu}^m[f]B_{n,\bnu}^m(\bx)
		\end{equation}
		by a ReLU$^{m-1}$ network. To this end we use the
		following facts:
		\begin{itemize}
			\item According to \cite[Theorem 2.5]{MR4040947}, there exists a
			network of finite width and depth that exactly expresses the
			square function $x^2$ on $\R$. It is now a standard observation,
			that using the polarization formula
			$xy=\frac{(x+y)^2-x^2-y^2}{2}$, we may also express the product of
			two numbers as a neural network.  Repeatedly stacking such
			networks, we conclude that there exists a neural network
			$\tilde p$ of finite width and depth that takes
			$(x_1,\dots,x_{d_1})\in\R^{d_1}$ as input and outputs
			$\tilde p(x_1,\dots,x_{d_1})=\prod_{i=1}^{d_1}x_i\in\R$. That is, for some
			fixed $C_{\tilde p}=C_{\tilde p}({d_1})$ holds $\tilde p\in \Phi^{d_1, 1}(L,W,S,B)$ with $L$, $W$,
			$S$, $B\le C_{\tilde p}$.
			\item For each $n\in\N$ 
   and each $j\in\{-m+1,\dots,n-1\}$ the spline
			(cp.~\eqref{eq:spline})
			\begin{equation*}
				B_{n,j}^m(x) = \sum_{i=0}^m (-1)^m
				\frac{\binom{m}{i}\max\{0,nx-(ni+j)\}^{m-1}}{(m-1)!}
			\end{equation*}
			corresponds to a ReLU$^{m-1}$ network in $\Phi^{d_1, 1}(L,W,S,B)$
			with $L=2$, $W=m+1$, $S=3(m+1)$ and $B=nm+n-1$.
			For the bound on $B$ we used that the maximum bias
			occurs in the term $ni+j$ with $i=m$ and $j=n-1$.
		\end{itemize}
		
		We first compute in parallel the terms
		\begin{equation*}
			B_{n,j}^{m}(x_i)\qquad\qquad\forall j\in\{-m+1,\dots,n-1\},~i\in\{1,\dots,d_1\}.
		\end{equation*}
		This can be achieved by a network
		$\tilde f_1:\R^{d_1}\to \R^{(n+m-1)d_1}$ of depth $2$, width
		$(m+1)(n+m-1)d_1$, and sparsity $O(3(m+1)(n+m-1)d_1)$. Additionally all
		weights and biases are upper bounded by $nm+n-1$.
		
		Next, given the output of $\tilde f_1$, we consider a network
		$\tilde f_2:\R^{(n+m-1)d_1}\to \R^{(n+m-1)^{d_1}}$ consisting of
		$(n+m-1)^{d_1}$ parallel product networks $\tilde p$, such that
		$\tilde f_2\circ \tilde f_1$ produces the outputs
		\begin{equation*}
			B_{n,\bnu}^m(\bx) = \tilde p(B_{n,\nu_1}^m(x_1),\dots,B_{n,\nu_{d_1}}^m(x_{d_1}))\qquad
			-m+1\le\nu_1,\dots,\nu_{d_1}\le n-1.
		\end{equation*}
		Then $\tilde f_2$ has depth at most $C_{\tilde p}$, width at most
		$C_{\tilde p}(m+n-1)^{d_1}$, sparsity at most $C_{\tilde p}(m+n-1)^{d_1}$,
		and all weights and biases are bounded by $C_{\tilde p}$.
		
		Given the output of $\tilde f_2\circ \tilde f_1$, a network
		$\tilde f_3:\R^{(n+m-1)^{d_1}}\to\R$ consisting of only one linear
		transformation is used to produce the function in
		\eqref{eq:toexpress}.  This network has depth $1$, width
		$(m+n-1)^{d_1}$, sparsity $(m+n-1)^{d_1}$, and upper bound
		$C\norm[{C([0,1]^{d_1})}]{f}$ for the modulus of all weights and
		biases. The last bound holds according to Lemma \ref{lemma:blambda}.

		Finally, to combine all three networks we use the so-called
		``sparse-concatenation'' denoted by
		$\odot$, which was first introduced for ReLU networks in
		\cite[Definition 2.5]{PETERSEN2018296}, but which can be extended to
		ReLU$^{m-1}$ networks, see \cite[Section 2.2.3]{MR4376568}. That is, we set
		\begin{equation*}
			\tilde f:= \tilde f_3\odot\tilde f_2\odot\tilde f_1.
		\end{equation*}
		It is a consequence of the properties of sparse concatenation (see
		\cite{MR4376568}) that this defines a network realizing the function
		$\tilde f(x_1,\dots,x_{d_1})=\tilde f_3(\tilde f_2(\tilde
		f_1(x_1,\dots,x_{d_1})))$ such that the depth and width are bounded up
		to a multiplicative and additive constant by the sum of the depth
		and sparsity of the three subnetworks.  An upper bound on the
		modulus of the network's weights and biases is obtained, up to an
		additive constant, by the maximal bound of the three subnetworks for
		this quantity.  Finally, the sparsity of $\tilde f$ is bounded by
		the summed sparsity of $\tilde f_1$, $\tilde f_2$ and $\tilde f_3$
		together with the number of required connections between
		$\tilde f_1$ and $\tilde f_2$, as well as between $\tilde f_2$ and
		$\tilde f_3$. Since each of the $(m+n-1)^{d_1}$ networks $\tilde p$ in
		$\tilde f_2$ gets exactly $d_1$ inputs, the former is bounded by
		$O((m+n-1)^{d_1} d_1)$. Since $\tilde f_3$ merely computes a linear
		combination of the $(n+m-1)^{d_1}$ outputs of $\tilde f_2$, the latter
		is bounded by $O((n+m-1)^{d_1})$. Absorbing some terms into the
		constant, for the network $\tilde f$ this leads to the bounds
		\begin{align*}
			L&\le C\\
			W&\le C(1+m+n)^{d_1}=O(n^{d_1})\\
			S&\le C(1+d(m+n)^{d_1}+m(n+m){d_1})=O(n^{d_1})\\
			B&\le C(1+\norm[{C([0,1]^{d_1})}]{f}+nm)=O(n)
		\end{align*}
		for some $C=C(d_1,m)$ independent of $n$ and $f$, and where the
		constants in the $O(\cdot)$ notation only depend on $m$ and $d_1$.
		Substituting $N(n):=C(1+(m+n)^{d_1}+m(n+m)d_1)=O(n^{d_1})$ yields
		\eqref{eq:WLSBbound}, and \Cref{thm:splinemultid} implies
		\eqref{eq:err_relum}.
	\end{proof}

\begin{proof}[Proof of corollary \ref{cor:NNapproximation}]
		Denote $\sigma_{m}(x)=\max\{0,x\}^{m}$,
		$p:=\lceil \log_{m}(\max\{2,k\})\rceil\ge 1$ and
		$\tilde m:=(m)^{p}$. Then
		\begin{equation}\label{eq:sigmam}
                  \underbrace{\sigma_{m}\circ\cdots\circ\sigma_{m}}_{p\text{ times}} = \sigma_{\tilde m}
		\end{equation}
		and by definition $\tilde m\ge k$. Fix $N\in\N$.

                According to \Cref{thm:NNapproximation}, for each
                $j\in\{1,\dots,d_2\}$ there exists a ReLU$^{\tilde{m}}$ network
                $\tilde f_j\in\Phi^{d_1,1}(L_{j,1},W_{j,1},S_{j,1},B_{j,1})$ such
                that
		\begin{equation*}
			L_{j,1}\le C,\qquad W_{j,1}\le N,\qquad S_{j,1}\le N,\qquad B_{j,1}\le C\norm[{C([0,1]^{d_1})}]{f_j}+N^{1/d_1},
                \end{equation*}
                for some $C=C(d_1,k,\tilde m)$ independent of $j$
                and
		\begin{equation}\label{eq:fjerror}
			\norm[{W^{r,\infty}([0,1]^{d_1})}]{f_j-\tilde f_j}\le C  N^{-\frac{k-r}{d_1}}
			\snorm[{C^k([0,1]^{d_1})}]{f_j}\qquad\forall r\in \{0,\dots,k\}.
                      \end{equation}
                      
                Replacing each activation function $\sigma_{\tilde m}$
                with the composition \eqref{eq:sigmam}, we may
                interpret $\tilde f_j$ as a ReLU$^m$ network in $\Phi^{d_1,1}(L_{j,2},W_{j,2},S_{j,2},B_{j,2})$
                with
		\begin{equation*}
			L_{j,2}\le pC ,\qquad W_{j,2}\le N,\qquad S_{j,2}\le pN,\qquad B_{j,2}\le C\norm[{C([0,1]^{d_1})}]{f_j}+N^{1/d_1},
                      \end{equation*}
                      i.e.,  the depth and sparsity increase by 
                      the
                      multiplicative $k$ and $m$ dependent factor $p$,
                      but the width
                      and
                      bound on the weights are not affected.

                      Next observe that $x^m=\sigma_m(x)+(-1)^m\sigma_m(-x)$.
                      Since $x^m,(x+1)^m,\dots,(x+m)^m$ are linearly
                      independent functions, we can find
                      coefficients $c_0,\dots,c_m$ such that
                      $x=\sum_{j=0}^m c_j\sigma_m(x)+(-1)^m\sigma_m(-x)$,
                      i.e.,  the identity is expressible by a
                      network of width $2(m+1)$ and with one hidden layer.
                      By concatenating $\tilde f_j$ with $L_{j,2}-\lceil pC\rceil$ such identity networks, we may assume that all $\tilde f_j$ have the same
                      depth $\lceil pC\rceil$, i.e., 
                      $\tilde f_j\in \Phi^{d_1,1}(L_{j,3},W_{j,3},S_{j,3},B_{j,3})$ with
		\begin{align*}
                  L_{j,2}&= \lceil pC\rceil ,\qquad W_{j,2}\le \max\{N,2d_1(m+1)\},\qquad  S_{j,2}\le (pN+K)\\
                  B_{j,2}&\le \max\{C\norm[{C([0,1]^{d_1})}]{f_j}+N^{1/d_1},\max_{j=1,\dots,m}c_j\},
                      \end{align*}
                      where $K$ is an absolute constant representing the
                      size of the identity network of depth $\lceil p C\rceil$
                      (the maximal depth possible depth required).

                      Parallelizing these networks of the same depth,
                      yields one big ReLU$^m$network $(\tilde f_j)_{j=1}^{d_2}\in\Phi^{d_1,d_2}(L,W,S,B)$ with
		\begin{align*}
                  L&= \lceil pC\rceil ,\qquad W\le d_2 \max\{N,2d_1(m+1)\},\qquad S\le d_2 (pN+K)\\
                  B&\le \max\{C\norm[{C([0,1]^{d_1})}]{f_j}+N^{1/d_1},\max_{j=1,\dots,m}c_j\}.
                \end{align*}                      

                Setting $\tilde N(d_2,N):=d_2(pN+K)$ yields the claimed
                bounds \eqref{eq:WLSBbound2}, and
                \eqref{eq:err_relum2} follows by \eqref{eq:fjerror}
                and $N=\frac{\tilde N}{d_2}-O(1)$.
              \end{proof}

\end{document}